\documentclass[10pt]{amsart}  

\usepackage[utf8]{inputenc}

\usepackage[left=1in,right=1in,top=1in,bottom=1.17in]{geometry}

\usepackage{graphicx,amssymb,amsmath,amsthm,wrapfig}
\usepackage{caption}
\usepackage{subcaption}
\usepackage{indentfirst}
\usepackage{cancel}
\usepackage{lipsum}
\usepackage{epstopdf}
\usepackage{epsfig}
\usepackage{comment}
\usepackage{enumerate,color}
\usepackage{empheq}
\usepackage{color}
\usepackage{hyperref}

\everymath{\displaystyle}

\theoremstyle{plain}
\newtheorem{theorem}{Theorem}[section]
\newtheorem{corollary}[theorem]{Corollary}
\newtheorem{lemma}[theorem]{Lemma}
\newtheorem{definition}[theorem]{Definition}
\newtheorem{proposition}[theorem]{Proposition}

\numberwithin{equation}{section}

\theoremstyle{definition}
\newtheorem{remark}[theorem]{Remark}

\newcommand{\RNum}[1]{\uppercase\expandafter{\romannumeral #1\relax}}

\title[Spectral stability of shocks for the NSP system]{Spectral stability of shock profiles for the Navier--Stokes--Poisson system}  
\author{Wanyong Shim}
\address{(Wanyong Shim) Department of Mathematical Sciences, Korea Advanced Institute of Science and Technology, Daejeon, 34141, Korea}
\email{wyshim25@kaist.ac.kr}

\thanks{\textbf{Acknowledgment.} This work was supported by Basic Science Research Program through the National Research Foundation of Korea(NRF) funded by the Ministry of Education(RS-2025-25429122).}

\begin{document}

\begin{abstract}
We investigate the spectral stability of small-amplitude shock profiles for the one-dimensional isothermal Navier--Stokes--Poisson system, which describes ion dynamics in a collision-dominated plasma. Specifically, we establish (i) bounds on the essential spectrum, (ii) bounds on the point spectrum, and (iii) simplicity of the zero eigenvalue for the linearized operator about the profile in $L^2$. The result in (i) shows that the zero eigenvalue arising from translation invariance is embedded in the essential spectrum. Consequently, the standard Evans function approach cannot be applied directly to prove (iii). To resolve this, we employ an Evans-function framework that extends into regions of the essential spectrum. In this framework, simplicity of the embedded zero eigenvalue is understood to mean that the extended Evans function has a simple zero at the origin. We then compute the derivative of this function at the origin and obtain a factorization into two factors: one associated with transversality of the connecting profile and the other with hyperbolic stability of the corresponding shock of the quasi-neutral Euler system. We show that both factors are nonzero, thereby establishing simplicity of the zero eigenvalue. \\

\end{abstract}
\maketitle

\tableofcontents

\section{Introduction}
We consider the one-dimensional compressible Navier--Stokes--Poisson (NSP) system, which serves as a model for the dynamics of ions in an isothermal plasma in the collision-dominated regime \cite{GSKG}. In Lagrangian mass coordinates, the NSP system is written as
\begin{subequations} \label{NSP}
\begin{align}
& \label{NSP11} v_t - u_y = 0,\\
& \label{NSP22} u_t + \left( \frac{T}{v} \right)_y   =  \left( \frac{\nu u_y}{v} \right)_y - \frac{\phi_y}{v}, \\
& \label{NSP33} - \varepsilon^2 \bigg( \frac{\phi_y}{v} \bigg)_y = 1 - ve^{\phi}
\end{align}
\end{subequations}
for $t>0$ and $y \in \mathbb{R}$. Here $v=\tfrac{1}{n}$ is the specific volume for $n>0$, the density of ions, and $\phi$ is the electric potential. The constants $T>0$, $\nu>0$ and $\varepsilon>0$ represent the absolute temperature, viscosity coefficient and Debye length, respectively. In the Poisson equation \eqref{NSP33}, we have assumed that the electron density $n_e$ is determined by the Boltzmann relation, $n_e = e^\phi$, which is justified by the physical observation that electrons reach the equilibrium state much faster than ions for varying potential in a plasma \cite{Ch}.

The NSP system \eqref{NSP} admits a smooth traveling-wave solution, called a shock profile, connecting two distinct constant far-field states \cite{DLZ}. In this paper, we consider small-amplitude shock profiles of \eqref{NSP} and prove that the linearized operator about the profile has no spectrum in the closed right half-plane except at the origin. Moreover, the zero eigenvalue arising from translation invariance is simple. These two spectral stability conditions are recognized as necessary and sufficient conditions for linear and nonlinear orbital stability within the general stability theory for Lax-type viscous and relaxation shocks \cite{MZ1, MZ2, MZ3, MZ4, MZ5}.

As in the general case of viscous and relaxation shocks, the zero eigenvalue and the essential spectrum of the linearized operator are not separated, i.e., there is no spectral gap between them. This motivates us to employ the Evans-function framework of Zumbrun--Howard \cite{ZH} and Mascia--Zumbrun \cite{MZ1, MZ2}, based on the gap/conjugation lemma \cite{GZ, KS, MeZ}. This enables us to characterize the zero eigenvalue by analyzing an \emph{extended} Evans function at the origin. As a result, the first derivative of the Evans function at the origin is factored into a transversality coefficient and the Liu--Majda determinant for the corresponding quasi-neutral Euler shock; see Remark~\ref{Rem_E-cond}.

This factorization is consistent with corresponding results for viscous conservation laws and relaxation systems \cite{Z1}, particularly for the compressible Navier--Stokes (NS) equations \cite{ZJL}. However, the mode structure underlying this factorization differs in our problem. The (nonlocal) electric force in the momentum equation \eqref{NSP22} introduces two additional modes; the eigenvalue problem for the NSP shocks can be written as a five-dimensional first-order ODE system, whereas the corresponding system for the isentropic NS equations is three-dimensional \cite{BHRZ}. In the low-frequency regime, these two modes are \emph{fast} and consequently contribute to the transversality coefficient at the origin. The required transversality is then established through an analysis of the variational equation along the profile under the small-amplitude assumption. Together with the nonvanishing of the Liu--Majda determinant, this yields a simple zero of the extended Evans function at the origin. Further details of the overall argument are provided in Section~\ref{Outline}.

To put the present work in perspective, we briefly recall earlier stability results for one-dimensional shock profiles of the NSP system. Under a zero-mass type assumption on the initial perturbation, Duan, Liu, and Zhang \cite{DLZ} proved asymptotic stability using the classical energy method; see also \cite{LMY, Zh} for stability results within this zero-mass framework in different perturbation settings. The zero-mass assumption is closely tied to the use of anti-derivative variables and restricts the class of admissible perturbations, in particular excluding perturbations with nonzero total mass that lead to nontrivial modulation of the shock location.

More recently, Kang, Kwon, and the present author \cite{KKSh} applied a relative entropy method to establish asymptotic stability up to a time-dependent shift, without imposing the zero-mass restriction; see also \cite{Sh} for an extension to composite waves. In that approach, the shift is shown to be asymptotically negligible relative to the shock speed, thereby yielding asymptotic orbital stability. However, the estimates obtained in \cite{KKSh} do not suffice to show that the shift converges to a constant, so a limiting shock location is not identified.

The nonlinear stability results reviewed above are largely formulated at the level of energy estimates and therefore yield limited quantitative information on the long-time dynamics. In particular, they neither quantify the influence of the Poisson coupling nor provide sharp decay rates for the perturbations under consideration. The spectral stability conditions established in this paper complement these works by characterizing the spectral structure of the associated linearized operator, including the additional modes in the low-frequency regime induced by the nonlocal Poisson coupling. Moreover, our spectral analysis places the stability problem for NSP shocks within the scope of the pointwise semigroup theory of Zumbrun--Howard \cite{ZH} and Mascia--Zumbrun \cite{MZ1, MZ2}, in which such conditions yield linear and nonlinear orbital stability with sharp decay rates and identification of the limiting shock location. The corresponding analyses of linear and nonlinear stability for NSP shocks are left to future work.

\subsection{Shock profiles for the Navier--Stokes--Poisson system}
We introduce shock profiles for the NSP system \eqref{NSP} and their basic properties. A shock profile is a smooth traveling-wave solution of the form $(v,u,\phi)(t,y)= (\bar{v},\bar{u},\bar{\phi})(x)$, where $x=y-st$ with shock speed $s$, connecting two distinct constant far-field states $(v_-,u_-,\phi_-)$ and $(v_+,u_+,\phi_+)$. Substituting this ansatz into \eqref{NSP}, we obtain the governing ODEs for the shock profiles:
\begin{subequations} \label{waveeq}
\begin{align}
&\label{waveeq_1}- s \bar{v}_x - \bar{u}_x =0, \\
&\label{waveeq2}- s \bar{u}_x + \left(\frac{T}{\bar{v}} \right)_x = \left(\frac{\nu \bar{u}_x}{\bar{v}} \right)_x - \frac{\bar{\phi}_x}{\bar{v}}, \\
&\label{waveeq3}- \varepsilon^2 \left( \frac{\bar{\phi}_x}{\bar{v}}\right)_x = 1 - \bar{v} e^{\bar{\phi}},
\end{align}
\end{subequations}
with the far-field condition
\begin{equation} \label{ffcond}
\lim_{x \rightarrow \pm \infty} (\bar{v},\bar{u},\bar{\phi})(x) = (v_\pm,u_\pm,\phi_\pm), \quad \phi_\pm=-\ln v_\pm,
\end{equation}
where the quasi-neutral relation $\phi_\pm=-\ln v_\pm$ follows by taking the formal limit $x \to \pm \infty$ in \eqref{waveeq3}. The shock speed $s$ is determined by the Rankine--Hugoniot condition:
\begin{equation} \label{RH}
\begin{split}
& -s (v_+-v_-)-(u_+-u_-)=0, \\
& -s (u_+-u_-)+(T+1)\left(\frac{1}{v_+} - \frac{1}{v_-} \right) = 0,
\end{split}
\end{equation}
which yields
\begin{equation} \label{RH1}
s = s_\pm = \pm \sqrt{\frac{T+1}{v_+v_-}}.
\end{equation}
The second relation in \eqref{RH} is obtained by rewriting \eqref{waveeq2} in divergence form and integrating the resulting equation; see \eqref{divfor2}. In this paper, without loss of generality, we restrict ourselves to the \emph{2-shock profile} with $s=s_+$, which satisfies the Lax entropy condition
\begin{equation} \label{Lax}
v_+ > v_-.
\end{equation}
The existence and uniqueness of the small-amplitude 2-shock profile have been treated in \cite{DLZ}:

\begin{lemma}[\cite{DLZ}, Proposition~1.9] \label{SP}
Fix $T>0$, $\nu>0$, and $\varepsilon>0$. For given $(v_-,u_-)$ with $v_->0$, there exists a constant $\delta_0=\delta_0(T,\nu,\varepsilon,v_-,u_-)>0$ such that for any $(v_+,u_+,s)$ satisfying \eqref{RH}, \eqref{Lax}, and
\begin{equation*}
|v_+-v_-| =: \delta_S < \delta_0,
\end{equation*}
the ODE system \eqref{waveeq} admits a unique (up to a shift) shock profile $(\bar{v},\bar{u},\bar{\phi})(x)$ satisfying \eqref{ffcond} and
\begin{equation} \label{sprel}
s\bar{v}_x = -\bar{u}_x >0, \quad \underline{C} \bar{u}_x \leq \bar{\phi}_x \leq \overline{C} \bar{u}_x < 0
\end{equation}
for some positive constants $\underline{C}$, $\overline{C}$. Moreover, the unique solution satisfying $\textstyle \bar{v}(0)=\frac{v_-+v_+}{2}$ verifies the derivative bounds
\begin{equation} \label{spbound}
\begin{cases}
\left| \frac{d^k}{dx^k} (\bar{v}-v_+, \bar{u}-u_+,\bar{\phi}-\phi_+) \right| \leq C_k \delta_S^{k+1} e^{-\theta \delta_S |x|}, & x>0 \\
\left| \frac{d^k}{dx^k} (\bar{v}-v_-, \bar{u}-u_-,\bar{\phi}-\phi_-) \right| \leq C_k \delta_S^{k+1} e^{-\theta \delta_S |x|}, & x<0
\end{cases}
\end{equation}
for $k \in \mathbb{N} \cup \{0 \}$, where $C_k>0$ and $\theta>0$ are generic constants.
\end{lemma}

\subsection{Linearized operator and spectral terminology} \label{sec:terminology}

To state our result, we define the linearized operator $L$ about the shock profile described in Lemma~\ref{SP} on $L^2(\mathbb R;\mathbb C)\times L^2(\mathbb R;\mathbb C)$ with domain $D(L)=H^1(\mathbb R;\mathbb C)\times H^2(\mathbb R;\mathbb C)$:
\begin{equation} \label{Ldef}
LU := \left(A U \right)_x + \left( B U_x \right)_x + \left(D ( \mathcal{A}^{-1}\mathcal{B} U )_x \right)_x + \left(E ( \mathcal{A}^{-1}\mathcal{B} U )_{xx} \right)_x,
\end{equation}
where $U=(v,u)^{\mathrm{tr}}$ and
\begin{equation} \label{ABDE}
\begin{split}
A(x) &= \begin{pmatrix}
s & 1 \\
\frac{T+1}{\bar{v}^2} - \frac{\nu\bar{u}_x}{\bar{v}^2} - \frac{\varepsilon^2 \bar{\phi}_x^2}{\bar{v}^3} + \frac{\varepsilon^2}{\bar{v}} \left( \frac{\bar{\phi}_x}{\bar{v}^2} \right)_x + \frac{\varepsilon^2}{\bar{v}^2} \left( \frac{\bar{\phi}_x}{\bar{v}} \right)_x & s
\end{pmatrix} \\
B(x) &= \begin{pmatrix}
0 & 0 \\
\frac{\varepsilon^2 \bar{\phi}_x}{\bar{v}^3} &  \frac{\nu}{\bar{v}}
\end{pmatrix}, \quad D(x) = \begin{pmatrix}
0 & 0 \\
\frac{\varepsilon^2 \bar{\phi}_x}{\bar{v}^2} + \frac{\varepsilon^2\bar{v}_x}{\bar{v}^3} & 0
\end{pmatrix}, \quad E(x) = \begin{pmatrix}
0 & 0 \\
-\frac{\varepsilon^2}{\bar{v}^2} & 0
\end{pmatrix}.
\end{split}
\end{equation}
The nonlocal operator $\mathcal{A}^{-1}\mathcal{B}$ denotes the solution operator of the linearized Poisson equation (cf. Appendix~\ref{App_B}), where
\begin{equation} \label{calAB}
\begin{split}
\mathcal{A} & =\bar{v} e^{\bar{\phi}}+ \frac{\varepsilon^2 \bar{v}_x}{\bar{v}^2} \partial_x-\frac{\varepsilon^2}{\bar{v}} \partial_{xx}, \quad \mathcal{B} = -e^{\bar{\phi}} -\frac{\varepsilon^2 \bar{\phi}_{xx}}{\bar{v}^2} + \frac{2\varepsilon^2 \bar{v}_x \bar{\phi}_x}{\bar{v}^3}-\frac{\varepsilon^2 \bar{\phi}_x}{\bar{v}^2} \partial_x.
\end{split}
\end{equation}
We use the notational extension $\mathcal{A}^{-1}\mathcal{B}U := (\mathcal{A}^{-1}\mathcal{B}v,0)^{\mathrm{tr}}$, since $D$ and $E$ in \eqref{ABDE} have a vanishing second column, and hence only the first component is used in \eqref{Ldef}; for a detailed derivation of $L$, see Section~\ref{sec:lin}. The expression in \eqref{Ldef} is understood in the sense of distributions and belongs to $L^2\times L^2$ for every $U\in D(L)$; see Lemma~\ref{lemma:L_domain} in Appendix~\ref{App_closed}. We note that, as shown in Proposition~\ref{Prop:2.2}, the operator $L$ is closed and densely defined on $L^2(\mathbb{R})\times L^2(\mathbb{R})$.

We next specify our conventions for the essential and point spectra, following \cite[Definition~2.2.3]{KP}. We define $\sigma_{\mathrm{ess}}(L)$ as the set of those $\lambda\in\mathbb{C}$ for which $\lambda I-L$ is either not Fredholm or is Fredholm with nonzero index, and set $\sigma_{\mathrm{pt}}(L)=\sigma(L)\setminus\sigma_{\mathrm{ess}}(L)$. We remark that the zero eigenvalue arising from translation invariance is embedded in the essential spectrum and, under this convention, is not included in $\sigma_{\mathrm{pt}}(L)$; see Proposition~\ref{prop:ess} and Remark~\ref{nogap}.

We denote by $\mathcal D(\lambda)$ the Evans function associated with the eigenvalue problem for $L$, analytically extended to a neighborhood of $\lambda=0$. Its construction is given in Section~\ref{Construction}. Throughout this paper, simplicity of the zero eigenvalue means that this analytic extension has a simple zero at $\lambda=0$; cf. \cite[Section~1.2]{MZ1}.

\subsection{Main result}

The main result of this paper is as follows.

\begin{theorem} \label{Main}
Fix $T>0$, $\nu>0$, and $\varepsilon>0$. Let $(\bar{v},\bar{u},\bar{\phi})$ be the shock profile described in Lemma~\ref{SP}, and let $L$ denote the linear operator defined in \eqref{Ldef}--\eqref{calAB}. There exists a constant $\delta_1=\delta_1(T,\nu,\varepsilon,v_-,u_-)\in(0,\delta_0]$ such that if the shock strength $\delta_S$ is less than $\delta_1$, then the shock profile is strongly spectrally stable in $L^2$. That is,
\begin{equation} \label{sstability}
\sigma(L) \cap \{ \lambda \in \mathbb{C}: \operatorname{Re}{\lambda} \geq 0 \} = \{0\},
\end{equation}
where $\sigma(L)$ denotes the spectrum of $L$ on $L^2(\mathbb{R}) \times L^2(\mathbb{R})$. Moreover, $\lambda=0$ is an eigenvalue of $L$, and the Evans function $\mathcal D$ satisfies
\begin{equation}\label{Evans0}
\mathcal D'(0)\neq0.
\end{equation}
\end{theorem}

\begin{remark} \label{Rem_E-cond}
Under the hypotheses of the general stability theory for one-dimensional viscous shock profiles, the conditions \eqref{sstability} and \eqref{Evans0} are necessary and sufficient for linearized and nonlinear stability \cite[Section~1.4.2]{ZJL}; see also \cite{MZ3, MZ4, MZ5}. As shown in Proposition~\ref{D'rel}, the Evans function condition $\mathcal D'(0)\neq0$ is equivalent to
\begin{equation}
\Gamma \Delta \neq 0,
\end{equation}
where $\Gamma$ measures transversality of the connection between $(v_-,u_-,\phi_-)$ and $(v_+,u_+,\phi_+)$ given by the profile $(\bar{v},\bar{u},\bar{\phi})$, and $\Delta$ is the Liu--Majda determinant for the corresponding shock of the associated hyperbolic system, namely the quasi-neutral Euler system:
\begin{equation*}
\begin{split}
& v_t - u_y =0, \\
& u_t + (T+1) \left( \frac{1}{v} \right)_y =0,\\
& \phi = - \ln v.
\end{split}
\end{equation*}
This system can be obtained by taking the formal limits $\varepsilon \to 0$ and $\nu \to 0$ in \eqref{NSP}. For further discussion of the relation \eqref{gDel} in Proposition~\ref{D'rel} and the associated transversality coefficient in the context of viscous shock profiles, we refer the reader to \cite[Section~1.4]{MZ2}.
\end{remark}

\subsection{Relation to previous work}

We first recall several results on the stability of viscous and radiative shocks that are closely related to the present work.

Humpherys and Zumbrun \cite{HZ} studied shock profiles for symmetric hyperbolic--parabolic systems of the form
\begin{equation*}
z_t+f(z)_x=(B(z)z_x)_x, \quad z \in \mathbb{R}^n,
\end{equation*}
under suitable structural assumptions that allow degenerate viscosity while ensuring that the hyperbolic characteristic modes are coupled to viscous dissipation. Using spectral energy estimates and a standard analysis of the essential spectrum, they proved spectral stability of small-amplitude shock profiles in the sense of \eqref{sstability} in Theorem~\ref{Main}. For small-amplitude Lax shocks, the spectral stability result of \cite{HZ}, together with transversality of the profile and hyperbolic stability of the associated inviscid shock, yields the Evans-function condition required in the pointwise semigroup theory of Zumbrun and Howard \cite{ZH} and Mascia and Zumbrun \cite{MZ2}; see \cite[Remark~1.7]{MZ2}. Under this condition, Mascia and Zumbrun \cite{MZ2} established sharp pointwise Green-function bounds and linearized orbital stability. They then combined these bounds with energy estimates to establish nonlinear orbital stability \cite{MZ4}. This result applies, in particular, to the compressible Navier--Stokes equations.

For hyperbolic--elliptic coupled systems, Lattanzio, Mascia, Nguyen, Plaza, and Zumbrun \cite{LMNPZ} and Nguyen, Plaza, and Zumbrun \cite{NPZ} studied the stability of radiative shocks in the scalar and system cases, respectively. In \cite{NPZ}, the equations have the form
\begin{equation*}
\begin{split}
&z_t+f(z)_x+\ell q_x=0,\\
&-q_{xx}+q+g(z)_x=0,
\end{split}
\end{equation*}
where $z\in\mathbb R^n$, $q\in\mathbb R$, and $\ell\in\mathbb R^n$ is a constant vector. Here, $df(z)$, $\ell$, and $dg(z)$ are assumed to satisfy structural conditions ensuring that the elliptic coupling provides dissipation. Under the corresponding Evans-function condition, these works established nonlinear orbital stability by combining pointwise Green-function bounds with energy estimates. They also verified this condition for shocks of sufficiently small amplitude. In particular, \cite[Lemma~2.5]{NPZ} gives a factorization of the first derivatives of the Evans functions at the origin, corresponding to \eqref{gDel}. We note that, in both the radiative and NSP systems, the elliptic coupling gives rise to additional fast modes in the low-frequency regime, which contribute to the profile-dependent factor, denoted by $\Gamma$ for the NSP system. The eigenvalue and eigenvector expansions of these modes near $\lambda=0$, however, differ between the two systems, reflecting the structural differences discussed below.

In this context, Theorem~\ref{Main} establishes spectral stability of small-amplitude shock profiles for the NSP system \eqref{NSP}, a physical model with a hyperbolic--parabolic--elliptic structure. At the linear level, expressing the electric potential $\phi$ in terms of the specific volume $v$ yields the nonlocal operator $L$ in \eqref{Ldef}, rather than a local operator of the type treated in \cite{HZ,MZ2}. Moreover, the Poisson coupling does not have the dissipative structure assumed for the elliptic coupling in \cite{NPZ} and therefore plays a different role in the coupled system. This difference is reflected in the limiting dispersion relations. Indeed, for the radiative system, the dispersion relation at a constant end state $(z_*,0)$ in the moving frame is
\begin{equation*}
\det\left((\lambda-is\xi)I+i\xi\,df(z_*)+\frac{\xi^2}{1+\xi^2}\ell\,dg(z_*)\right)=0,\quad \xi\in\mathbb{R},
\end{equation*}
where $s$ denotes the shock speed. Under the structural assumptions of \cite{NPZ}, the elliptic coupling provides dissipation through the term involving $\ell\,dg(z_*)$. For the NSP system, linearization about a constant end state $(v_*,u_*,-\ln v_*)$ gives the dispersion relation
\begin{equation} \label{dispersion_NSP}
\det \left( (\lambda-is\xi)I + i\xi \begin{pmatrix}
0 & -1 \\
-T/v_*^2 & 0 
\end{pmatrix} + \frac{\nu\xi^2}{v_*} \begin{pmatrix}
0 & 0 \\
0 & 1
\end{pmatrix} - \frac{i\xi}{v_*(v_*+\varepsilon^2\xi^2)} \begin{pmatrix}
0 & 0 \\
1 & 0
\end{pmatrix} \right)=0
\end{equation}
for $\xi \in \mathbb{R}$. We note that setting $\nu=0$ in this dispersion relation gives purely imaginary roots, and hence that the Poisson coupling alone does not provide the dissipation assumed in \cite{NPZ}. Furthermore, the Poisson coupling interacts with viscosity, modifying the dissipative structure of the coupled system. In particular, the spectral energy estimates contain cross terms involving the fluid variables and the electric potential, and their control requires both viscous dissipation and estimates for the Poisson equation. These structural differences prevent a direct application of the stability results for hyperbolic--parabolic and hyperbolic--elliptic systems discussed above and require a separate spectral analysis for the NSP system.

\subsection{Main ideas of the proof} \label{Outline}

We now describe the main ideas of the proof of Theorem~\ref{Main}. The theorem follows from a bound on the essential spectrum, a bound on the point spectrum, and the simplicity of the zero eigenvalue. These three results are established in Proposition~\ref{prop:ess}, Proposition~\ref{Prop:3.1}, and Corollary~\ref{cor:zero_simple}, respectively.

We first prove the stability condition \eqref{sstability} for the essential spectrum. We establish Fredholm equivalences between $L-\lambda I$, an augmented local operator, and a first-order operator. The Fredholm criterion in \cite[Section~3]{Sandstede} then reduces the analysis to that of the limiting coefficient matrices, which we relate to the Fourier symbols of the limiting operators $L_\pm$, obtained by linearizing about the constant far-field states. Since $L_\pm$ have constant coefficients, Fourier transformation characterizes their essential spectra by the roots of the dispersion relations \eqref{dispersion_NSP} with $v_*=v_\pm$. We show that these roots satisfy the dissipativity bound \eqref{dissip}, with strictly negative real parts for $\xi\neq0$, and that both roots vanish at $\xi=0$. Together with the Fredholm index calculation, these bounds yield $\sigma_{\mathrm{ess}}(L)\cap\{\operatorname{Re}\lambda\ge0\}=\{0\}$. In particular, the zero eigenvalue arising from translation invariance is embedded in the essential spectrum.

To complete the proof of \eqref{sstability}, we show that $\sigma_{\mathrm{pt}}(L)\subset\mathbb C\setminus\Omega$, where $\Omega:=\{\lambda\in\mathbb C:\operatorname{Re}\lambda\ge0\}\setminus\{0\}$. For an $L^2$-eigenfunction $U=(v,u)^{\mathrm{tr}}$ with eigenvalue $\lambda\in\Omega$, we introduce the anti-derivatives $\Phi$ and $\Psi$ of $v$ and $u$, respectively. These form an eigenfunction of the associated integrated operator $\mathcal L$ with the same eigenvalue; see \eqref{integrated_operator} and Section~\ref{sec:eigenfunction_correspondence}. We combine the spectral energy approach of \cite{HZ} with the energy dissipation mechanism identified in \cite[Section~4.2]{DLZ}. More precisely, we establish a linearized, complex-valued analogue of the energy estimates in that work for the eigenvalue equations associated with $\mathcal L$ and obtain the desired bound by a contradiction argument.

In the energy estimate for $(\Phi,\Psi)$, the Poisson coupling gives rise to a term involving $\operatorname{Re}(\lambda\Phi\overline{\phi_x})$, where $\phi$ is determined by $\Phi_x$ through the linearized Poisson equation. To handle this term, which has no definite sign, we combine the estimate with an identity obtained from the complex conjugate of the Poisson equation. The leading $\lambda$-dependent terms then take the form $\operatorname{Re}\lambda\,\mathcal E$, where
\begin{equation*}
\mathcal E := \int_{\mathbb R} \bigg((T+1)|\Phi|^2+\bar v^2|\Psi|^2-2\varepsilon^2\operatorname{Re}(\Phi\overline{\phi_x})+\varepsilon^2\bar v|\phi|^2+\frac{\varepsilon^4e^{-\bar\phi}}{\bar v}|\phi_x|^2\bigg)\,dx. 
\end{equation*}
For $T>0$ and sufficiently small shock strength, the quadratic form
\begin{equation*}
(T+1)|\Phi|^2-2\varepsilon^2\operatorname{Re}(\Phi\overline{\phi_x})+\frac{\varepsilon^4e^{-\bar\phi}}{\bar v}|\phi_x|^2
\end{equation*}
appearing in $\mathcal E$ is uniformly positive definite. Consequently, $\mathcal E$ controls $\|(\Phi,\Psi,\phi,\phi_x)\|_{L^2}^2$. Combining the resulting energy estimate with elliptic and higher-order estimates allows us to control the remaining terms arising from the Poisson coupling.

Finally, we use the Evans-function framework of \cite{MZ1,MZ2,ZH} to prove that the eigenvalue $\lambda=0$ is simple, i.e., $\mathcal D'(0)\neq0$. The first-order system \eqref{5} associated with the eigenvalue problem is five-dimensional, compared with the three-dimensional system for isentropic Navier--Stokes shocks \cite{BHRZ}. By expanding the eigenvalues and eigenvectors of the limiting matrices near $\lambda=0$, we show that the two additional modes introduced by the Poisson coupling are fast; see Lemma~\ref{low1}. These expansions, together with the essential spectrum bounds obtained above, show that $B_+(0,r)=B(0,r)\cap\{\operatorname{Re}\lambda>0\}$ lies in the domain of consistent splitting for \eqref{5} when $r>0$ is sufficiently small. We apply the conjugation lemma to construct bases of solutions that depend analytically on $\lambda\in B(0,r)$ and, for $\lambda\in B_+(0,r)$, decay as $x\to+\infty$ and $x\to-\infty$, respectively. We then define the Evans function $\mathcal D(\lambda)$ on $B_+(0,r)$ using these decaying bases and extend it analytically to $B(0,r)$. With this construction of $\mathcal D$, a direct computation yields
\begin{equation*}
\mathcal{D}'(0)=\Gamma\Delta, 
\end{equation*}
where $\Gamma$ and $\Delta$ are as described in Remark~\ref{Rem_E-cond}. The two additional fast modes enter the determinant defining the transversality coefficient $\Gamma$. The Rankine--Hugoniot relation \eqref{RH} implies $\Delta\neq0$, and thus it remains to prove $\Gamma\neq0$. Here, we note that the transversality of the small-amplitude shock profiles in Lemma~\ref{SP} follows essentially from their construction in \cite[Section~2]{DLZ}. There the profiles are constructed via a center manifold reduction. In particular, for an augmented profile ODE system, the dynamics near a reference equilibrium reduces to a one-dimensional equation on the center manifold, with a heteroclinic orbit that is unique up to translation. Moreover, the linearized flow in the complementary subspace admits an exponential dichotomy. These two observations imply the transverse connection between corresponding stable and unstable manifolds at infinity, and hence $\Gamma\neq0$ for sufficiently small amplitudes. We record a detailed proof of this conclusion in Appendix~\ref{App_trans} using the tracking/reduction lemma of \cite{MZ2}. Since $\Delta\neq0$ and $\Gamma\neq0$, we conclude $\mathcal D'(0)\neq0$, which proves simplicity of the zero eigenvalue.

\medskip
\emph{Plan of the paper.} In Section~\ref{sec:lin}, we derive the linearized eigenvalue problem about the shock profiles. In Sections~\ref{sec:Ess} and~\ref{sec:Point}, we obtain the bounds on the essential spectrum and point spectrum, respectively. These bounds establish the strong spectral stability of the shock profiles stated in Theorem~\ref{Main}. Section~\ref{sec:Simplicity} proves the simplicity of the zero eigenvalue by showing that the Evans function has a simple zero at $\lambda=0$. The appendices collect background material and proofs: we recall the conjugation lemma of \cite{MeZ}, and we prove solvability of the linearized Poisson equation, the domain and closedness of the linearized operator, higher-order energy estimates deferred from Section~\ref{sec:Point}, and transversality of the shock profiles.

\section{Linearization about the profile} \label{sec:lin}
We rewrite the NSP system \eqref{NSP} in the moving frame $(t,x)$, where $x=y-st$, so that the shock profile defined in Lemma~\ref{SP} becomes a standing wave within the resulting system. The NSP system in the coordinates $(t,x)$ reads
\begin{subequations} \label{movingNSP}
\begin{align}
& \label{movingNSP1} v_t - sv_x - u_x =0, \\
& \label{movingNSP2} u_t - su_x + T \left( \frac{1}{v} \right)_x = \nu \left( \frac{u_x}{v} \right)_x - \frac{\phi_x}{v}, \\
& \label{movingNSP3} -\varepsilon^2 \left( \frac{\phi_x}{v} \right)_x = 1- ve^{\phi}.
\end{align}
\end{subequations}
From the Poisson equation \eqref{movingNSP3}, we derive the following identity:
\begin{equation*}
\frac{\phi_x}{v} = \left[ \frac{1}{v} + \frac{\varepsilon^2}{v} \left( \frac{\phi_x}{v} \right)_x - \frac{\varepsilon^2}{2} \left( \frac{\phi_x}{v} \right)^2 \right]_x.
\end{equation*}
Substituting this into \eqref{movingNSP2}, the system \eqref{movingNSP} is rewritten in divergence form as
\begin{subequations} \label{divfor}
\begin{align}
&\label{divfor1} v_t - sv_x - u_x =0, \\
&\label{divfor2} u_t - su_x + (T+1) \bigg( \frac{1}{v} \bigg)_x = \nu \left( \frac{u_x}{v} \right)_x - \varepsilon^2 \bigg[ \frac{1}{v} \bigg( \frac{\phi_x}{v} \bigg)_x -\frac{1}{2} \bigg( \frac{\phi_x}{v} \bigg)^2 \bigg]_x, \\
&\label{divfor3} -\varepsilon^2 \left( \frac{\phi_x}{v} \right)_x = 1- ve^{\phi}.
\end{align}
\end{subequations}
We note that this reformulation highlights the structure of the NSP system. In the momentum equation \eqref{divfor2}, we find that $\textstyle(\frac{1}{v})_x$ on the left-hand side contributes to the hyperbolic part of the system, coming from the electric force in \eqref{movingNSP2}, while the forcing term due to the non-neutral plasma density, depending on $\varepsilon$, remains separated on the right-hand side.

Define perturbations around the shock profile $(\bar{v},\bar{u},\bar{\phi})(x)$ as
\begin{equation} \label{defpert}
(\tilde{v},\tilde{u},\tilde{\phi})(t,x) := (v,u,\phi)(t,x) - (\bar{v},\bar{u},\bar{\phi})(x).
\end{equation}
Then, the perturbation equations are given by
\begin{equation*}
\begin{split}
& \tilde{v}_t - s\tilde{v}_x - \tilde{u}_x = 0, \\
& \tilde{u}_t - s\tilde{u}_x + (T+1) \left( \frac{1}{v} - \frac{1}{\bar{v}} \right)_x - \nu \left( \frac{u_x}{v} - \frac{\bar{u}_x}{\bar{v}} \right)_x \\
& \qquad \qquad \qquad = - \varepsilon^2 \bigg[ \frac{1}{v} \left( \frac{\phi_x}{v} \right)_x - \frac{1}{\bar{v}} \left( \frac{\bar{\phi}_x}{\bar{v}} \right)_x \bigg]_x + \frac{\varepsilon^2}{2} \bigg[ \left( \frac{\phi_x}{v} \right)^2 - \left(\frac{\bar{\phi}_x}{\bar{v}}\right)^2 \bigg]_x, \\
& -\varepsilon^2 \left( \frac{\phi_x}{v}- \frac{\bar{\phi}_x}{\bar{v}} \right)_x = \bar{v}e^{\bar{\phi}} - v e^\phi.
\end{split}
\end{equation*}
We linearize this system about the profile $(\bar{v},\bar{u},\bar{\phi})$ to obtain
\begin{subequations} \label{LNSP}
\begin{align}
&\label{LNSP1} v_t - sv_x - u_x = 0, \\
\label{LNSP2} \begin{split} & u_t - su_x - (T+1) \left( \frac{v}{\bar{v}^2} \right)_x - \nu \left( \frac{u_x}{\bar{v}} - \frac{\bar{u}_x}{\bar{v}^2}v \right)_x \\
&  \qquad \qquad \qquad  =  \varepsilon^2 \left( \frac{\bar{\phi}_x}{\bar{v}^2}\phi_x - \frac{\bar{\phi}_x^2}{\bar{v}^3} v \right)_x  - \varepsilon^2 \left[ \frac{1}{\bar{v}}\left( \frac{\phi_x}{\bar{v}} - \frac{\bar{\phi}_x}{\bar{v}^2}v  \right)_x - \left( \frac{\bar{\phi}_x}{\bar{v}} \right)_x \frac{v}{\bar{v}^2} \right]_x, 
\end{split} \\
& \label{LNSP3} -\varepsilon^2 \left( \frac{\phi_x}{\bar{v}} - \frac{\bar{\phi}_x}{\bar{v}^2}v \right)_x = - e^{\bar{\phi}}v - \bar{v}e^{\bar{\phi}}\phi.
\end{align}
\end{subequations}
In \eqref{LNSP}, we have suppressed tildes for notational convenience. By solvability of the linearized Poisson equation \eqref{LNSP3} for $\phi$ (see Lemma~\ref{theorem:6.1.1}), we may write 
\begin{equation} \label{phisol}
\phi = \mathcal{A}^{-1} \mathcal{B} v,
\end{equation}
with the differential operators $\mathcal{A}$ and $\mathcal{B}$ defined in \eqref{calAB}. Using \eqref{phisol}, we reduce the linearized NSP system \eqref{LNSP} to
\begin{equation} \label{2}
U_t = LU,
\end{equation}
where $U=(v,u)^{\mathrm{tr}}$ and the nonlocal linear operator $L$ is defined by \eqref{Ldef}--\eqref{ABDE}.

We further define the limiting operators of $L$ as $x\to \pm\infty$:
\begin{equation} \label{Lpm}
L_\pm U :=  A_\pm U_x + B_\pm U_{xx} + E_\pm \left( \mathcal{A}_\pm^{-1} \mathcal{B}_\pm U \right)_{xxx}
\end{equation}
with
\begin{equation} \label{ABpm}
\mathcal{A}_\pm = 1 - \frac{\varepsilon^2}{v_\pm} \partial_{xx}, \quad \mathcal{B}_\pm =  - \frac{1}{v_\pm},
\end{equation}
where $F_\pm$ denotes its limit $\lim_{x \rightarrow \pm \infty} F(x)$ for a matrix-valued function $F(x)$. For later use, we also write the linearized momentum equation in the (original) non-divergence form
\begin{equation} \label{LNSP2'}
u_t - su_x - T \left( \frac{v}{\bar{v}^2} \right)_x - \nu \left( \frac{u_x}{\bar{v}} - \frac{\bar{u}_x}{\bar{v}^2} v \right)_x = - \left( \frac{\phi_x}{\bar{v}} - \frac{ \bar{\phi}_x}{\bar{v}^2}v \right).
\end{equation}
After eliminating $\phi$ by \eqref{phisol}, this non-divergence formulation gives another representation of $L$ on $D(L)$; see Lemma~\ref{lemma:L_domain} in Appendix~\ref{App_closed}. Consequently, both formulations lead to the same spectral problem.

\section{Essential spectrum bounds} \label{sec:Ess}

In this section, we establish bounds for the essential spectrum of $L$, defined by \eqref{Ldef}--\eqref{calAB}. These bounds provide a partial proof of the spectral stability condition \eqref{sstability} in Theorem~\ref{Main}. Moreover, our analysis shows that there is no spectral gap between $\lambda=0$ and the essential spectrum; see Remark~\ref{nogap}. This motivates our use, in Section~\ref{sec:Simplicity}, of an extended Evans-function framework to establish simplicity of the zero eigenvalue. The main result of this section is the following proposition.

\begin{proposition} \label{prop:ess}
Under the assumptions in Theorem~\ref{Main}, we have
\begin{equation*}
\sigma_{\mathrm{ess}}(L) \cap \{ \lambda \in \mathbb{C} : \operatorname{Re} \lambda \geq 0 \} = \{ 0 \},
\end{equation*}
where $\sigma_{\mathrm{ess}}(L)$ is the essential spectrum of the linearized operator $L$ on $L^2$.
\end{proposition}

The remainder of this section is devoted to the proof of Proposition~\ref{prop:ess}.

\subsection{Fredholm properties of the linearized operator}

To analyze the essential spectrum of $L$, we first introduce the augmented local operator $\mathcal P_\lambda:D(L)\times H^2(\mathbb R;\mathbb C)\to L^2(\mathbb R;\mathbb C^3)$ for $\lambda\in\mathbb C$ by
\begin{equation} \label{augmented_operator}
\mathcal{P}_\lambda(U,\phi) := \begin{pmatrix}
sv_x + u_x - \lambda v \\
su_x + T \left( \frac{v}{\bar v^2} \right)_x + \nu \left( \frac{u_x}{\bar v} - \frac{\bar u_x}{\bar v^2} v \right)_x - \left( \frac{\phi_x}{\bar v} - \frac{\bar\phi_x}{\bar v^2}v \right) - \lambda u \\
\mathcal{A} \phi - \mathcal{B} v
\end{pmatrix},
\end{equation}
where $U=(v,u)^{\mathrm{tr}}$ and $D(L)=H^1(\mathbb R;\mathbb C)\times H^2(\mathbb R;\mathbb C)$, as in Appendix~\ref{App_closed}. Then the eigenvalue problem $(L-\lambda I)U=0$, written using the non-divergence representation \eqref{L_nondiv_domain}, is equivalent to $\mathcal P_\lambda(U,\phi)=0$ with $\phi=\mathcal A^{-1}\mathcal Bv$. By \eqref{spbound}, \eqref{calAB}, and \eqref{vphibd}, the operator $\mathcal P_\lambda$ is bounded for each fixed $\lambda$. The following lemma establishes the Fredholm equivalence between $\mathcal P_\lambda$ and $L-\lambda I$.

\begin{lemma} \label{lemma:Fredholm_augmented}
For each $\lambda\in\mathbb C$, the operator $\mathcal P_\lambda$ is Fredholm if and only if $L-\lambda I:D(L)\to L^2(\mathbb R;\mathbb C^2)$ is Fredholm, in which case
\begin{equation*}
\operatorname{ind}\mathcal P_\lambda=\operatorname{ind}(L-\lambda I),
\end{equation*}
where $\operatorname{ind}$ denotes the Fredholm index.
\end{lemma}

\begin{proof}
Let $f\in L^2(\mathbb{R};\mathbb{C}^2)$ and $h\in L^2(\mathbb{R};\mathbb{C})$ be arbitrary, and consider $\mathcal{P}_\lambda(U,\phi)=(f,h)$. By Lemma~\ref{theorem:6.1.1}, the operator $\mathcal{A} : H^2 \to L^2$ is a bounded isomorphism. Thus, the third equation of $\mathcal{P}_\lambda(U,\phi)=(f,h)$ gives
\begin{equation*}
\phi = \mathcal{A}^{-1}\mathcal{B}v + \mathcal{A}^{-1} h.
\end{equation*}
Substituting this expression into the first two equations and using the representation \eqref{L_nondiv_domain}, we obtain
\begin{equation*}
(L-\lambda I) U = f + \begin{pmatrix}
0 \\ 
\frac{(\mathcal A^{-1}h)_x}{\bar v}
\end{pmatrix}.
\end{equation*}
This suggests constructing bounded isomorphisms $\mathcal R$ and $\mathcal Q$ on the domain and codomain of $\mathcal P_\lambda$, respectively, such that
\begin{equation}\label{Fredholm_factorization}
\mathcal Q\mathcal P_\lambda\mathcal R=\begin{pmatrix}L-\lambda I&0\\0&\mathcal A\end{pmatrix}.
\end{equation}
Since $\mathcal{A}:H^2\to L^2$ is an isomorphism and composition with bounded isomorphisms preserves the Fredholm property and index, \eqref{Fredholm_factorization} implies that $\mathcal{P}_\lambda$ is Fredholm if and only if $L-\lambda I$ is Fredholm, and their indices agree.

To this end, we define maps $\mathcal{R}$ on $D(L)\times H^2$ and $\mathcal{Q}$ on $L^2(\mathbb{R};\mathbb{C}^2)\times L^2(\mathbb{R};\mathbb{C})$ by
\begin{equation*}
\mathcal{R}(U,\psi) := \left(U,\psi+\mathcal{A}^{-1}\mathcal{B}v\right), \quad \mathcal{Q}(f,h) := \left(f+\begin{pmatrix}0\\ \frac{(\mathcal{A}^{-1}h)_x}{\bar v}\end{pmatrix},h\right).
\end{equation*}
Then both maps are invertible, with inverses given by
\begin{equation*}
\mathcal{R}^{-1}(U,\phi) = \left(U,\phi-\mathcal{A}^{-1}\mathcal{B}v\right), \quad \mathcal{Q}^{-1}(f,h) = \left(f-\begin{pmatrix}0\\ \frac{(\mathcal{A}^{-1}h)_x}{\bar v}\end{pmatrix},h\right).
\end{equation*}
Moreover, by Lemma~\ref{theorem:6.1.1}, the operators $\mathcal A^{-1}:L^2\to H^2$ and $\mathcal A^{-1}\mathcal B:H^1\to H^2$ are bounded. Together with the boundedness of $\bar v^{-1}$ in \eqref{vphibd}, this shows that both maps and their inverses are bounded. By the definitions of $\mathcal{P}_\lambda$, $\mathcal{R}$, and $\mathcal{Q}$, together with the non-divergence representation \eqref{L_nondiv_domain} of $L$, we obtain
\begin{equation*}
\mathcal{Q}\mathcal{P}_\lambda\mathcal{R}(U,\psi) = \mathcal{Q} \left( (L-\lambda I)U-\begin{pmatrix}
0 \\ 
\frac{\psi_x}{\bar v}
\end{pmatrix}, \mathcal{A} \psi \right) = \left( (L-\lambda I)U,\mathcal{A}\psi \right)
\end{equation*}
for all $(U,\psi)\in D(L)\times H^2$, where the last equality follows from $\mathcal{A}^{-1}\mathcal{A}\psi=\psi$. This proves \eqref{Fredholm_factorization} and completes the proof.
\end{proof}

We next rewrite the augmented system $\mathcal{P}_\lambda(U,\phi) =0$ as a first-order system. For this, we introduce auxiliary variables $p$ and $q$ and define, for $Z=(v,u,p,\phi,q)^{\mathrm{tr}}\in H^1(\mathbb R;\mathbb C^5)$,
\begin{equation} \label{first_order_local_operator}
\mathcal T_\lambda Z:=\begin{pmatrix}
sv_x+\frac{\bar v}{\nu}p-\lambda v\\
\frac{s\bar v}{\nu}p+T\left(\frac{v}{\bar v^2}\right)_x+\left(p-\frac{\nu\bar u_x}{\bar v^2}v\right)_x-\frac{q}{\bar v}+\frac{\bar\phi_x}{\bar v^2}v-\lambda u\\
\bar v e^{\bar\phi}\phi-\varepsilon^2\left(\frac{q}{\bar v}\right)_x+e^{\bar\phi}v+\varepsilon^2\left(\frac{\bar\phi_x}{\bar v^2}v\right)_x\\
u_x-\frac{\bar v}{\nu}p\\
\phi_x-q
\end{pmatrix}.
\end{equation}
By \eqref{spbound} and \eqref{vphibd}, $\mathcal T_\lambda:H^1(\mathbb R;\mathbb C^5)\to L^2(\mathbb R;\mathbb C^5)$ is bounded for each fixed $\lambda$. When $p=\nu u_x/\bar v$ and $q=\phi_x$, the last two components vanish, and the first three components agree with $\mathcal P_\lambda(U,\phi)$. We compare the Fredholm properties of $\mathcal{P}_\lambda$ and $\mathcal{T}_\lambda$.

\begin{lemma} \label{lemma:Fredholm_first_order}
For each $\lambda\in\mathbb C$, the operator $\mathcal T_\lambda$ is Fredholm if and only if $\mathcal P_\lambda$ is Fredholm, in which case
\begin{equation*}
\operatorname{ind}\mathcal T_\lambda=\operatorname{ind}\mathcal P_\lambda.
\end{equation*}
\end{lemma}

\begin{proof}
Let $f\in L^2(\mathbb R;\mathbb C^3)$ and $h=(h_1,h_2)^{\mathrm{tr}}\in L^2(\mathbb R;\mathbb C^2)$ be arbitrary, and consider $\mathcal T_\lambda Z=(f,h)$. The last two equations of $\mathcal T_\lambda Z=(f,h)$ give
\begin{equation} \label{temp_uxphix}
u_x = \frac{\bar v}{\nu}p+h_1, \quad \phi_x = q+h_2.
\end{equation}
Since $h_1,h_2$ belong only to $L^2$, these equations do not imply $u,\phi\in H^2$, as required in the domain of $\mathcal P_\lambda$. To compare the Fredholm properties of the two operators, we therefore construct suitable changes of variables, adapting the argument in \cite[Appendix~A]{SandstedeScheel}.

More precisely, we seek functions $\eta_j\in H^1$ such that $(\eta_j)_x-h_j\in H^1$ for $j=1,2$. The identities in \eqref{temp_uxphix} then yield
\begin{equation} \label{tildeuphi}
\begin{split}
(u-\eta_1)_x &= \frac{\bar v}{\nu}p-\big((\eta_1)_x-h_1\big),\\
(\phi-\eta_2)_x &= q-\big((\eta_2)_x-h_2\big),
\end{split}
\end{equation}
and hence $u-\eta_1,\phi-\eta_2\in H^2$. We construct such functions by solving
\begin{equation} \label{etajODE}
(\eta_j)_x-\eta_j=h_j, \quad j=1,2.
\end{equation}
Multiplying \eqref{etajODE} by $e^{-x}$ and integrating from $x$ to $+\infty$, we find that any $H^1$ solution satisfies
\begin{equation}\label{etajformula}
\eta_j(x)=-\int_x^\infty e^{x-y}h_j(y)\,dy, \quad j=1,2.
\end{equation}
Applying Young's convolution inequality to this formula, we obtain
\begin{equation*}
\|\eta_j\|_{L^2}\le \|e^x\|_{L^1((-\infty,0])}\|h_j\|_{L^2}=\|h_j\|_{L^2}, \quad j=1,2.
\end{equation*}
Moreover, differentiating \eqref{etajformula} in the sense of distributions, we have $(\eta_j)_x=\eta_j+h_j\in L^2$. Thus, $\eta_j\in H^1$ solves \eqref{etajODE} and satisfies
\begin{equation} \label{etajH1bound}
\|\eta_j\|_{H^1}\le C\|h_j\|_{L^2}, \quad j=1,2.
\end{equation}
In particular, \eqref{etajODE} gives directly $(\eta_j)_x-h_j=\eta_j\in H^1$, as desired.

We now use the auxiliary functions $\eta_j$ to relate solutions of $\mathcal T_\lambda Z=(f,h)$ to those of an inhomogeneous equation for $\mathcal P_\lambda$. Setting $\tilde u:=u-\eta_1$ and $\tilde\phi:=\phi-\eta_2$ and using \eqref{etajODE} in \eqref{tildeuphi}, we have
\begin{equation*}
\tilde{u}_x = \frac{\bar v}{\nu}p-\eta_1, \quad \tilde{\phi}_x = q-\eta_2,
\end{equation*}
and hence $\tilde u,\tilde\phi\in H^2$. Substituting $u=\tilde u+\eta_1$, $p=(\nu/\bar v)(\tilde u_x+\eta_1)$, $\phi=\tilde\phi+\eta_2$, and $q=\tilde\phi_x+\eta_2$ into the first three equations of $\mathcal T_\lambda Z=(f,h)$, we obtain
\begin{equation} \label{first_order_forcing}
\mathcal P_\lambda\big((v,\tilde u)^{\mathrm{tr}},\tilde\phi\big)=f-\begin{pmatrix}
\eta_1\\
(s-\lambda)\eta_1+\nu\left(\frac{\eta_1}{\bar v}\right)_x-\frac{\eta_2}{\bar v}\\
\bar v e^{\bar\phi}\eta_2-\varepsilon^2\left(\frac{\eta_2}{\bar v}\right)_x
\end{pmatrix}=:g.
\end{equation}
Conversely, suppose that $(v,\tilde u,\tilde\phi)\in H^1\times H^2\times H^2$ solves \eqref{first_order_forcing}. Then
\begin{equation} \label{first_order_reconstruction}
Z=\left(v,\tilde u+\eta_1,\frac{\nu}{\bar v}(\tilde u_x+\eta_1),\tilde\phi+\eta_2,\tilde\phi_x+\eta_2\right)^{\mathrm{tr}}
\end{equation}
belongs to $H^1(\mathbb R;\mathbb C^5)$ by \eqref{spbound} and \eqref{vphibd}, and substitution into \eqref{first_order_local_operator} gives $\mathcal T_\lambda Z=(f,h)$.

To compare the closedness and codimensions of the ranges, we first show that the map $(f,h)\mapsto(g,h)$, with $g$ defined by \eqref{first_order_forcing}, is a bounded linear isomorphism on $L^2(\mathbb R;\mathbb C^3)\times L^2(\mathbb R;\mathbb C^2)$. Note from \eqref{etajformula} and \eqref{first_order_forcing} that $f-g$ depends linearly only on $h$. Thus, for each $h$, $f$ and $g$ uniquely determine each other, and hence the map $(f,h)\mapsto(g,h)$ is a linear bijection. Moreover, using \eqref{spbound}, \eqref{etajH1bound}, and \eqref{vphibd} in \eqref{first_order_forcing}, we obtain
\begin{equation*}
\|f-g\|_{L^2} \leq C_\lambda\big(\|\eta_1\|_{H^1} + \|\eta_2\|_{H^1}\big) \leq C_\lambda \|h\|_{L^2}. 
\end{equation*}
The map $(f,h)\mapsto(g,h)$ and its inverse are therefore bounded.

We next use this map to compare the ranges of the two operators. By \eqref{first_order_forcing} and \eqref{first_order_reconstruction}, for any $(f,h)\in L^2(\mathbb R;\mathbb C^3)\times L^2(\mathbb R;\mathbb C^2)$, we have
\begin{equation*}
(f,h)\in\operatorname{ran}\mathcal T_\lambda \quad\Longleftrightarrow\quad (g,h)\in\operatorname{ran} \mathcal P_\lambda\times L^2(\mathbb R;\mathbb C^2).
\end{equation*}
Since the map and its inverse are bounded, this equivalence implies that $\operatorname{ran}\mathcal T_\lambda$ is closed if and only if $\operatorname{ran}\mathcal P_\lambda$ is closed. Moreover, since $h\in L^2(\mathbb R;\mathbb C^2)$ is arbitrary, we also obtain
\begin{equation} \label{Fredholm_range_codim}
\operatorname{codim}_{L^2(\mathbb R;\mathbb C^5)}\operatorname{ran}\mathcal T_\lambda
=\operatorname{codim}_{L^2(\mathbb R;\mathbb C^3)}\operatorname{ran}\mathcal P_\lambda.
\end{equation}

Finally, we compare the kernels. When $f=h=0$, \eqref{etajformula} gives $\eta_1=\eta_2=0$, and \eqref{first_order_forcing} gives $g=0$. By \eqref{temp_uxphix}, \eqref{first_order_forcing}, and \eqref{first_order_reconstruction}, the map
\begin{equation*}
(U,\phi) \longmapsto \left(v,u,\frac{\nu}{\bar v}u_x,\phi,\phi_x\right)^{\mathrm{tr}} 
\end{equation*}
is a linear bijection from $\ker\mathcal P_\lambda$ onto $\ker\mathcal T_\lambda$, with inverse $(v,u,p,\phi,q)^{\mathrm{tr}}\mapsto((v,u)^{\mathrm{tr}},\phi)$. Hence,
\begin{equation} \label{Fredholm_kernel_dim}
\dim\ker\mathcal T_\lambda=\dim\ker\mathcal P_\lambda.
\end{equation}
Combining the equivalence of closed ranges with \eqref{Fredholm_range_codim} and \eqref{Fredholm_kernel_dim}, we conclude that $\mathcal T_\lambda$ is Fredholm if and only if $\mathcal P_\lambda$ is Fredholm, with $\operatorname{ind}\mathcal T_\lambda=\operatorname{ind}\mathcal P_\lambda$.
\end{proof}

To apply the Fredholm theory for first-order systems, we rewrite the equation $\mathcal T_\lambda Z=0$ in the form
\begin{equation}\label{first_order_normalization}
Z_x=\tilde{\mathbb A}(x,\lambda)Z,
\end{equation}
where
\begin{equation} \label{first_order_coefficient_matrix}
\tilde{\mathbb A}(x,\lambda) = \begin{pmatrix} \frac{\lambda}{s} & 0 & -\frac{\bar v}{s\nu} & 0 & 0 \\ 
0 & 0 & \frac{\bar v}{\nu} & 0 & 0 \\ 
-\left(\frac{T-\nu\bar u_x}{\bar v^2}\right)_x-\frac{\lambda(T-\nu\bar u_x)}{s\bar v^2}-\frac{\bar\phi_x}{\bar v^2} & \lambda & \frac{T-\nu\bar u_x}{s\nu\bar v}-\frac{s\bar v}{\nu} & 0 & \frac{1}{\bar v} \\ 
0 & 0 & 0 & 0 & 1 \\
\bar v \left(\frac{\bar\phi_x}{\bar v^2}\right)_x+\frac{\lambda\bar\phi_x}{s\bar v}+\frac{\bar v e^{\bar\phi}}{\varepsilon^2} & 0 & -\frac{\bar\phi_x}{s\nu} & \frac{\bar v^2e^{\bar\phi}}{\varepsilon^2} & \frac{\bar v_x}{\bar v} 
\end{pmatrix}. 
\end{equation}
We then establish the corresponding equivalence at the operator level. By \eqref{first_order_local_operator} and \eqref{first_order_coefficient_matrix}, we have
\begin{equation} \label{first_order_operator_identity} \big(\partial_x-\tilde{\mathbb A}(x,\lambda)\big)Z=\begin{pmatrix} \frac{1}{s}&0&0&0&0\\ 0&0&0&1&0\\ -\frac{T-\nu\bar u_x}{s\bar v^2}&1&0&0&0\\ 0&0&0&0&1\\ \frac{\bar\phi_x}{s\bar v}&0&-\frac{\bar v}{\varepsilon^2}&0&0 \end{pmatrix}\mathcal T_\lambda Z \end{equation}
for all $Z\in H^1(\mathbb R;\mathbb C^5)$. The matrix multiplying $\mathcal T_\lambda Z$ on the right-hand side of \eqref{first_order_operator_identity} has bounded entries by \eqref{spbound} and \eqref{vphibd}, and its determinant is $\bar v/(s\varepsilon^2)$, which is bounded away from zero by \eqref{RH1} and \eqref{vphibd}. Hence, multiplication by this matrix defines a bounded isomorphism on $L^2(\mathbb R;\mathbb C^5)$. This implies that $\mathcal T_\lambda$ is Fredholm if and only if $\partial_x-\tilde{\mathbb A}(x,\lambda)$ is Fredholm, in which case their indices agree.

Taking the limits in \eqref{first_order_coefficient_matrix} and using \eqref{ffcond} and \eqref{spbound}, we obtain
\begin{equation} \label{A_pm}
\lim_{x\to\pm\infty}\tilde{\mathbb A}(x,\lambda) = \begin{pmatrix}
\frac{\lambda}{s} & 0 & -\frac{v_\pm}{s\nu} & 0 & 0 \\
0 & 0 & \frac{v_\pm}{\nu} & 0 & 0 \\
-\frac{T\lambda}{sv_\pm^2} & \lambda & \frac{T}{s\nu v_\pm}-\frac{sv_\pm}{\nu} & 0 & \frac{1}{v_\pm} \\
0 & 0 & 0 & 0 & 1 \\
\frac{1}{\varepsilon^2} & 0 & 0 & \frac{v_\pm}{\varepsilon^2} & 0
\end{pmatrix} =: \mathbb A_\pm(\lambda).
\end{equation}
We call a matrix hyperbolic if it has no eigenvalues on the imaginary axis. The following lemma gives the Fredholm criterion for $L-\lambda I$ in terms of $\mathbb A_\pm(\lambda)$.

\begin{lemma} \label{lemma:Fredholm_limits}
For each $\lambda\in\mathbb C$, the operator $L-\lambda I$ is Fredholm if and only if both $\mathbb A_+(\lambda)$ and $\mathbb A_-(\lambda)$ are hyperbolic. In this case,
\begin{equation} \label{Fredholm_index_formula}
\operatorname{ind}(L-\lambda I)=i_-(\lambda)-i_+(\lambda),
\end{equation}
where $i_\pm(\lambda)$ denotes the number of eigenvalues of $\mathbb A_\pm(\lambda)$ with positive real part, counted with algebraic multiplicity.
\end{lemma}

\begin{proof}
By Lemmas~\ref{lemma:Fredholm_augmented} and \ref{lemma:Fredholm_first_order}, together with \eqref{first_order_operator_identity}, it suffices to prove the stated Fredholm criterion and \eqref{Fredholm_index_formula} for the operator
\begin{equation*}
\partial_x-\tilde{\mathbb A}(x,\lambda):H^1(\mathbb R;\mathbb C^5)\to L^2(\mathbb R;\mathbb C^5).
\end{equation*}
To this end, we verify the assumptions needed to apply the Fredholm theory in \cite[Section~3]{Sandstede}. By \eqref{first_order_coefficient_matrix}, the entries of $\tilde{\mathbb A}(x,\lambda)$ are polynomials of degree at most one in $\lambda$, with smooth real coefficients depending on $x$. This verifies Hypothesis~3.1 of \cite{Sandstede}. Moreover, these coefficients are bounded and converge exponentially to their limits as $x\to\pm\infty$ by \eqref{spbound} and \eqref{vphibd}, and the limiting matrices are given in \eqref{A_pm}. We can therefore apply \cite[Theorem~3.3]{Sandstede} to \eqref{first_order_normalization}. It then follows that this system has exponential dichotomies on $\mathbb R_+$ and $\mathbb R_-$ if and only if both $\mathbb A_+(\lambda)$ and $\mathbb A_-(\lambda)$ are hyperbolic. In this case, the unstable subspaces of these dichotomies have dimensions $i_+(\lambda)$ and $i_-(\lambda)$, respectively. By \cite[Theorem~3.2 and Remark~3.3]{Sandstede}, the operator $\partial_x-\tilde{\mathbb A}(x,\lambda)$ is Fredholm if and only if these dichotomies exist, with
\begin{equation*}
\operatorname{ind}\big(\partial_x-\tilde{\mathbb A}(x,\lambda)\big)=i_-(\lambda)-i_+(\lambda).
\end{equation*}
This proves the stated Fredholm criterion and \eqref{Fredholm_index_formula}.
\end{proof}

\subsection{Analysis of the limiting operators}

By Lemma~\ref{lemma:Fredholm_limits}, the operator $L-\lambda I$ is Fredholm if and only if both $\mathbb A_+(\lambda)$ and $\mathbb A_-(\lambda)$ are hyperbolic. We examine this condition through the dispersion relations of the limiting operators $L_\pm$ defined in \eqref{Lpm}. The Fourier symbols of $L_\pm$ are given by
\begin{equation} \label{symbols}
M_\pm(\xi) := i\xi A_\pm-\xi^2 B_\pm+\frac{i\xi^3}{v_\pm+\varepsilon^2\xi^2}E_\pm, \quad \xi\in\mathbb R.
\end{equation}
A direct computation using \eqref{A_pm} and \eqref{symbols} gives
\begin{equation} \label{spatial_temporal_relation}
\det\big(\mathbb A_\pm(\lambda)-i\xi I\big)=\frac{v_\pm(v_\pm+\varepsilon^2\xi^2)}{s\nu\varepsilon^2}\det\big(M_\pm(\xi)-\lambda I\big), \quad \xi\in\mathbb R.
\end{equation}
Let $\lambda_j^\pm(\xi)$, $j=1,2$, denote the eigenvalues of $M_\pm(\xi)$, i.e., the solutions of the dispersion relation
\begin{equation} \label{drelation}
\det\big(M_\pm(\xi)-\lambda I\big)=0.
\end{equation}
Since $v_\pm+\varepsilon^2\xi^2>0$, the relation \eqref{spatial_temporal_relation} shows that both $\mathbb A_+(\lambda)$ and $\mathbb A_-(\lambda)$ are hyperbolic precisely when $\lambda$ does not lie on any of the dispersion curves $\{\lambda_j^\pm(\xi):\xi\in\mathbb R\}$. The following lemma provides the required bounds on these curves.

\begin{lemma} \label{lemma:dissip}
There exists a constant $\theta_0>0$ such that
\begin{equation} \label{dissip}
\operatorname{Re} \lambda_j^\pm(\xi) \leq - \frac{\theta_0 \lvert \xi \rvert^2}{1 + \lvert \xi \rvert^2}, \quad j=1,2
\end{equation}
for all $\xi \in \mathbb{R}$, where $\lambda_j^\pm(\xi)$ are the eigenvalues of $M_\pm(\xi)$ defined by \eqref{symbols}. 
\end{lemma}

\begin{proof}
Substituting the definition \eqref{symbols} of $M_\pm$ into \eqref{drelation}, we have
\begin{equation} \label{chpoly}
\begin{split}
& \det (M_\pm(\xi) - \lambda I) \\ 
& \quad = (si\xi -\lambda)\left(si\xi-\frac{\nu}{v_\pm}\xi^2- \lambda \right)-i\xi \left(\frac{T+1}{v_\pm^2}i\xi - \frac{\varepsilon^2 i \xi^3}{\varepsilon^2 v_\pm^2 \xi^2 + v_\pm^3} \right) \\
& \quad = \lambda^2 - \left(2si\xi - \frac{\nu}{v_\pm} \xi^2 \right)\lambda - s^2 \xi^2 -\frac{\nu}{v_\pm} si\xi^3 + \frac{T+1}{v_\pm^2}\xi^2 - \frac{\varepsilon^2 \xi^4}{\varepsilon^2 v_\pm^2 \xi^2 + v_\pm^3} =0.
\end{split}
\end{equation}
The roots of \eqref{chpoly} are
\begin{equation} \label{ess}
\lambda_j^\pm (\xi) := \lambda_j (M_\pm) = \left( si\xi - \frac{\nu}{2 v_\pm} \xi^2 \right)  + \frac{(-1)^j}{2} \sqrt{f_\pm(\xi)}, \quad j=1,2,
\end{equation}
where
\begin{equation*}
f_\pm(\xi) := \frac{\nu^2}{v_\pm^2}\xi^4-\frac{4(T+1)}{v_\pm^2}\xi^2 + \frac{4 \varepsilon^2\xi^4}{\varepsilon^2 v_\pm^2 \xi^2 + v_\pm^3}.
\end{equation*}
For notational convenience, for a fixed sign $\pm$, we write $f(\xi)$ in place of $f_\pm(\xi)$, and similarly suppress the $\pm$-dependence in $\tilde{f}, g, \eta_0,$ and $\xi_0$ in what follows. The first term on the right-hand side of \eqref{ess} has real part $-\tfrac{\nu \xi^2}{2v_\pm}$ for all $\xi\in\mathbb{R}$. Thus, to determine the precise real part of $\lambda_j^\pm(\xi)$, it suffices to analyze the second term. For this purpose, we define
\begin{equation*}
\tilde{f}(\xi) := \varepsilon^2 \nu^2 \xi^4 + (\nu^2 v_\pm -4\varepsilon^2 T)\xi^2 - 4(T+1)v_\pm.
\end{equation*}
Then we observe that
\begin{equation*}
f(\xi) = \frac{\xi^2}{\varepsilon^2 v_\pm^2 \xi^2 + v_\pm^3} \tilde{f}(\xi),
\end{equation*}
and so,
\begin{equation} \label{signf}
\operatorname{sgn} (\tilde{f}(\xi)) = \operatorname{sgn}(f(\xi))
\end{equation}
for all $\xi \in \mathbb{R} \setminus \{0 \}$. We therefore analyze $\tilde{f}$ instead of $f$. Set $\eta := \xi^2 \geq 0$. Then the equation $\tilde{f}(\xi) = 0$ is equivalent to
\begin{equation} \label{etaeq}
g(\eta) := \varepsilon^2 \nu^2 \eta^2 + (\nu^2 v_\pm -4\varepsilon^2 T)\eta - 4(T+1)v_\pm = 0.
\end{equation}
The positive root $\eta_0$ of \eqref{etaeq} is given by
\begin{equation*}
\eta_0 = \frac{2T}{\nu^2} - \frac{v_\pm}{2\varepsilon^2} + \frac{ \sqrt{(\nu^2v_\pm -4\varepsilon^2T)^2 +16\varepsilon^2\nu^2 v_\pm (T+1) }}{2\varepsilon^2\nu^2}.
\end{equation*}
Moreover, the function $g$ is convex and satisfies
\begin{equation*}
g(0) = - 4(T+1)v_\pm <0.
\end{equation*}
Hence $g(\eta) < 0 $ for $0 < \eta < \eta_0$ and $g(\eta) >0$ for $\eta > \eta_0$. Returning to the $\xi$-variable, with $\xi_0 := \sqrt{\eta_0}$, we have
\begin{equation*}
\tilde{f}(\xi) < 0 \quad \text{for } |\xi| < \xi_0, \quad \tilde{f}(\xi) >0 \quad \text{for } |\xi| > \xi_0.
\end{equation*}
Combined with \eqref{signf}, this implies that $\textstyle \sqrt{f(\xi)}$ is real for $|\xi|>\xi_0$, purely imaginary for $|\xi|<\xi_0$, and vanishes at $|\xi|=\xi_0$. Recalling \eqref{ess}, we obtain
\begin{equation} \label{ReLj}
\operatorname{Re}\lambda_j^\pm (\xi) =
\begin{cases}
- \frac{\nu}{2 v_\pm} \xi^2  + \frac{(-1)^j}{2} \sqrt{f(\xi)}, & \quad  |\xi| > \xi_0, \\
 -\frac{\nu}{2v_\pm} \xi^2, & \quad  |\xi| \leq \xi_0.
\end{cases}
\end{equation}
We now show that the curves $\lambda_j^\pm$ satisfy \eqref{dissip}. For $ \xi \in [-\xi_0, \xi_0]$, we have
\begin{equation} \label{inest}
\operatorname{Re}\lambda_j^\pm (\xi) = -\frac{\nu \xi^2}{2v_\pm} \leq -\frac{\nu \lvert \xi \rvert^2}{2v_\pm(1+\lvert \xi \rvert^2)} = -\frac{\theta_1 \lvert \xi \rvert^2}{1+\lvert \xi \rvert^2},
\end{equation}
where we set $ \textstyle \theta_1 := \frac{\nu}{2v_\pm}$. We next consider the case $|\xi| > \xi_0$. We observe that the function $f(\xi)$ can be rewritten as
\begin{equation*}
\begin{split}
f(\xi) & = \left( \frac{\nu}{v_\pm} \xi^2 - \frac{2T}{\nu v_\pm} \right)^2 - \frac{4T^2}{\nu^2 v_\pm^2} - \frac{4 \xi^2}{v_\pm^2} + \frac{4\varepsilon^2 \xi^4}{\varepsilon^2 v_\pm^2 \xi^2 + v_\pm^3} \\
& = \left( \frac{\nu}{v_\pm} \xi^2 - \frac{2T}{\nu v_\pm} \right)^2 - \frac{4T^2}{\nu^2 v_\pm^2} - \frac{4 \xi^2}{\varepsilon^2 v_\pm \xi^2 + v_\pm^2}.
\end{split}
\end{equation*}
From this representation, we observe that
\begin{equation*}
f(\xi) < \left( \frac{\nu}{v_\pm} \xi^2 - \frac{2T}{\nu v_\pm} \right)^2,
\qquad \xi \neq 0.
\end{equation*}
Since $f(\xi)>0$ for $|\xi|>\xi_0$, the square root $\sqrt{f(\xi)}$ is real in this region, and hence
\begin{equation} \label{absva}
0 \le \sqrt{f(\xi)} < \left| \frac{\nu}{v_\pm} \xi^2 - \frac{2T}{\nu v_\pm} \right|.
\end{equation}
We now note that $\textstyle \xi_0 > \frac{\sqrt{2T}}{\nu}$. Indeed, setting $\textstyle \eta_* := \frac{2T}{\nu^2}$, a direct computation yields
\begin{equation*}
g(\eta_*) = - \frac{4\varepsilon^2 T^2}{\nu^2} - 2v_\pm (T+2) < 0.
\end{equation*}
Since $g$ is convex and has a unique positive zero $\eta_0$, it follows that $\eta_0 > \eta_*$, and thus $\textstyle \xi_0 = \sqrt{\eta_0} > \frac{\sqrt{2T}}{\nu}$. Consequently, the quantity inside the absolute value on the right-hand side of \eqref{absva} is positive for $|\xi|>\xi_0$, and we conclude that
\begin{equation} \label{tempf}
\sqrt{f(\xi)} < \frac{\nu}{v_\pm} \xi^2 - \frac{2T}{\nu v_\pm},
\qquad |\xi|>\xi_0.
\end{equation}
Substituting \eqref{tempf} into \eqref{ReLj} with $j=2$, we have the uniform bound
\begin{equation*}
\operatorname{Re} \lambda_2^\pm(\xi) < - \frac{T}{\nu v_\pm}.
\end{equation*}
As $\operatorname{Re}\lambda_1^\pm \leq \operatorname{Re}\lambda_2^\pm$, we obtain
\begin{equation*}
\operatorname{Re}\lambda_j^\pm (\xi) < - \frac{T}{\nu v_\pm} \leq - \frac{T |\xi|^2}{\nu v_\pm (1 + |\xi|^2)}, \quad |\xi| > \xi_0,
\end{equation*}
which, together with \eqref{inest} for $|\xi| \leq \xi_0$, yields \eqref{dissip} with some $\theta_0>0$. Specifically, in the 2-shock case, one can choose
\begin{equation*}
\theta_0 = \min \left\{ \frac{\nu}{2v_+},\frac{T}{\nu v_+} \right\},
\end{equation*}
which is valid for both signs $\pm$.
\end{proof}

\subsection{Proof of Proposition~\ref{prop:ess}} \label{sec:proof_ess}

Let $\Omega=\{\lambda\in\mathbb C:\operatorname{Re}\lambda\ge0\}\setminus\{0\}$. By Lemma~\ref{lemma:Fredholm_limits} and the definition of the essential spectrum in Section~\ref{sec:terminology}, it suffices to show that $\mathbb A_\pm(\lambda)$ are hyperbolic with $i_-(\lambda)=i_+(\lambda)$ for all $\lambda\in\Omega$, whereas $\mathbb A_\pm(0)$ are not hyperbolic. The hyperbolicity on $\Omega$ follows from \eqref{symbols}, \eqref{spatial_temporal_relation}, and Lemma~\ref{lemma:dissip}.

To determine $i_\pm(\lambda)$, we use Lemma~\ref{low1}, whose proof in Section~\ref{Construction} is independent of the results of this section. Together with \eqref{RH1} and \eqref{Lax}, this lemma shows that each of $\mathbb A_+(\lambda)$ and $\mathbb A_-(\lambda)$ has three eigenvalues with positive real part and two with negative real part for all sufficiently small real $\lambda>0$. We may therefore choose a sufficiently small real $\lambda_0>0$ such that
\begin{equation*}
i_-(\lambda_0)=i_+(\lambda_0)=3.
\end{equation*}
Since $\mathbb A_\pm(\lambda)$ depend continuously on $\lambda$ and are hyperbolic on the connected set $\Omega$, their numbers of eigenvalues with positive real part are constant on $\Omega$. Hence,
\begin{equation*}
i_-(\lambda)=i_+(\lambda)=3 \quad \text{for all } \lambda\in\Omega.
\end{equation*}
Finally, since $M_\pm(0)=0$ by \eqref{symbols}, setting $\lambda=\xi=0$ in \eqref{spatial_temporal_relation} gives $\det\mathbb A_\pm(0)=0$. Thus, $\mathbb A_\pm(0)$ are not hyperbolic, which completes the proof.

\begin{remark}[Absence of a spectral gap] \label{nogap}
By \eqref{symbols}, \eqref{spatial_temporal_relation}, and Lemma~\ref{lemma:Fredholm_limits}, the eigenvalues $\lambda_j^\pm(\xi)$ belong to $\sigma_{\mathrm{ess}}(L)$ for all $\xi\in\mathbb R$ and satisfy $\lambda_j^\pm(\xi)\to0$ as $\xi\to0$. Moreover, $\operatorname{Re}\lambda_j^\pm(\xi)<0$ for $\xi\ne0$ by \eqref{dissip}. Thus, the essential spectrum accumulates at the origin from the open left half-plane, and there is no spectral gap.
\end{remark}

\section{Point spectrum bounds} \label{sec:Point}

This section is devoted to establishing the nonexistence of $L^2$-eigenvalues of the linearized operator $L$ in the deleted unstable half-plane. We prove this result by introducing an operator in the anti-derivative variables and applying the spectral energy method to its eigenvalue problem. The precise statement is given in the proposition below. Together with Proposition~\ref{prop:ess}, this proves \eqref{sstability} in Theorem~\ref{Main}.

\begin{proposition} \label{Prop:3.1}
Under the assumptions in Theorem~\ref{Main}, we have
\begin{equation*}
\sigma_{\mathrm{pt}}(L) \subset \mathbb{C} \setminus \Omega,
\end{equation*}
where $\sigma_{\mathrm{pt}}(L) = \sigma(L) \setminus \sigma_{\mathrm{ess}}(L)$ and $\Omega := \{\lambda \in \mathbb{C} : \operatorname{Re}\lambda \geq 0 \} \setminus \{0\} $.
\end{proposition}

We prove Proposition~\ref{Prop:3.1} in the next two subsections. Throughout this section, we denote the shock profile $(\bar{v},\bar{u},\bar{\phi})$ by $(\hat{v},\hat{u},\hat{\phi})$ to avoid confusion with complex conjugation.

\subsection{The integrated eigenvalue problem}

Using the relation $\phi=\mathcal{A}^{-1}\mathcal{B}v$, we rewrite the eigenvalue equation $(L-\lambda I)U=0$ in the extended form:
\begin{equation} \label{EE''}
\begin{split}
&\lambda v - s v_x - u_x =0, \\
&\lambda u - s u_x -(T+1) \left( \frac{v}{\hat{v}^2} \right)_x - \nu \left(\frac{u_x}{\hat{v}} - \frac{\hat{u}_x}{\hat{v}^2}v \right)_x  = \varepsilon^2 \bigg( \frac{\hat{\phi}_x}{\hat{v}^2}\phi_x -\frac{\hat{\phi}_x^2}{\hat{v}^3} v \bigg)_x \\
& \qquad -\varepsilon^2 \bigg[ \bigg( -\frac{2\hat{\phi}_{xx}}{\hat{v}^3} +\frac{3\hat{v}_x\hat{\phi}_x}{\hat{v}^4} \bigg) v -\frac{\hat{\phi}_x}{\hat{v}^3}v_x -\frac{\hat{v}_x}{\hat{v}^3} \phi_x +\frac{\phi_{xx}}{\hat{v}^2} \bigg]_x, \\
&-\varepsilon^2 \bigg( \frac{\phi_x}{\hat{v}} - \frac{\hat{\phi}_x}{\hat{v}^2} v \bigg)_x = - e^{\hat{\phi}}v - \hat{v} e^{\hat{\phi}} \phi.
\end{split}
\end{equation}
Following \cite{G1, KMN, MN, ZH}, we introduce the anti-derivative variables
\begin{equation} \label{antiv}
\Phi (x) := \int_{-\infty}^x v(z) \, dz, \quad \Psi(x) := \int_{-\infty}^x u(z) \, dz.
\end{equation}
Formally integrating the first two equations of \eqref{EE''}, we have
\begin{subequations} \label{3.3}
\begin{align}
& \lambda \Phi - s \Phi_x - \Psi_x =0, \label{3.3a} \\
\begin{split} 
& \lambda \Psi - s \Psi_x - \frac{(T+1)}{\hat{v}^2}\Phi_x - \frac{\nu}{\hat{v}}\Psi_{xx} = - \frac{\nu \hat{u}_x}{\hat{v}^2}\Phi_x + \varepsilon^2 \bigg( \frac{\hat{\phi}_x}{\hat{v}^2}\phi_x - \frac{\hat{\phi}_x^2}{\hat{v}^3} \Phi_x \bigg) \\
& \qquad  -\varepsilon^2 \bigg[ \bigg( -\frac{2\hat{\phi}_{xx}}{\hat{v}^3} +\frac{3\hat{\phi}_x\hat{v}_x}{\hat{v}^4} \bigg)
\Phi_x -\frac{\hat{\phi}_x}{\hat{v}^3}\Phi_{xx} -\frac{\hat{v}_x}{\hat{v}^3}\phi_x +\frac{\phi_{xx}}{\hat{v}^2} \bigg],
\end{split} \label{3.3b} \\
& -\varepsilon^2 \bigg( \frac{\phi_x}{\hat{v}} - \frac{\hat{\phi}_x}{\hat{v}^2} \Phi_x \bigg)_x =  - e^{\hat{\phi}} \Phi_x - \hat{v} e^{\hat{\phi}} \phi \label{3.3c}.
\end{align}
\end{subequations}
We define the \emph{integrated} operator $\mathcal{L}$ associated with \eqref{3.3} on $L^2(\mathbb R;\mathbb C^2)$ with domain
\begin{equation} \label{integrated_domain}
D(\mathcal L):=H^1(\mathbb R;\mathbb C)\times H^2(\mathbb R;\mathbb C)
\end{equation}
by
\begin{equation} \label{integrated_operator}
\mathcal{L} V := AV_x + BV_{xx} + D (\mathcal{A}^{-1}\mathcal{B}V_x)_x + E (\mathcal{A}^{-1}\mathcal{B}V_x)_{xx},
\end{equation}
where $V := (\Phi,\Psi)^{\mathrm{tr}}$ and $A,B,D$, and $E$ are as in \eqref{ABDE}. The expression in \eqref{integrated_operator} is understood in the sense of distributions.

\subsubsection{Regularity of eigenfunctions of $\mathcal{L}$}

To investigate the regularity of the eigenfunctions of $\mathcal{L}$, we first rewrite $\mathcal{L}V$ using the linearized Poisson equation. Combining \eqref{3.3c} with \eqref{waveeq3}, we obtain
\begin{equation*}
- \frac{\varepsilon^2}{\hat v} \bigg( \frac{\phi_x}{\hat v} - \frac{\hat\phi_x}{\hat v^2}\Phi_x \bigg)_x + \frac{\varepsilon^2\Phi_x}{\hat v^2} \bigg( \frac{\hat\phi_x}{\hat v} \bigg)_x = - e^{\hat\phi}\phi - \frac{\Phi_x}{\hat v^2}.
\end{equation*}
Substituting this identity into the second component of \eqref{integrated_operator}, we have
\begin{equation} \label{integrated_reduced}
\mathcal{L} V = \begin{pmatrix}
s\Phi_x + \Psi_x \\
s\Psi_x + \frac{T}{\hat v^2}\Phi_x + \nu \left( \frac{\Psi_{xx}}{\hat v} - \frac{\hat u_x}{\hat v^2}\Phi_x \right) + \varepsilon^2 \frac{\hat\phi_x}{\hat v}\bigg(\frac{\phi_x}{\hat v} - \frac{\hat\phi_x}{\hat v^2} \Phi_x \bigg) - e^{\hat\phi} \phi
\end{pmatrix}.
\end{equation}
Here $\phi=\mathcal A^{-1}\mathcal B\Phi_x\in H^1$ by \eqref{Poisson_solution_H1}. Thus, by \eqref{spbound}, all terms in \eqref{integrated_reduced} belong to $L^2$ for every $V\in D(\mathcal L)$.

\begin{lemma} \label{lemma:integrated_regularity}
Let $\lambda\in\mathbb C$, and suppose that $V=(\Phi,\Psi)^{\mathrm{tr}}\in D(\mathcal L)$ satisfies $(\mathcal L-\lambda I)V=0$. Then $\Phi,\Psi\in H^3(\mathbb R;\mathbb C)$ and $\phi=\mathcal A^{-1}\mathcal B\Phi_x\in H^2(\mathbb R;\mathbb C)$.
\end{lemma}

\begin{proof}
By \eqref{integrated_domain}, we have $\Phi\in H^1$ and $\Psi\in H^2$. Rearranging the first equation \eqref{3.3a}, we obtain
\begin{equation*}
\Phi_x=\frac{1}{s}(\lambda\Phi-\Psi_x).
\end{equation*}
The right-hand side belongs to $H^1$, and hence $\Phi\in H^2$. Since $\Phi_x\in H^1$, Lemma~\ref{theorem:6.1.1}\emph{(iii)} gives $\phi\in H^2$. By \eqref{integrated_reduced}, the second component of $\mathcal L V=\lambda V$ gives
\begin{equation} \label{tempPsixx}
\Psi_{xx} = \frac{\hat v}{\nu} \bigg[ \lambda \Psi - s\Psi_x - \frac{T-\nu\hat u_x}{\hat v^2} \Phi_x - \varepsilon^2 \frac{\hat\phi_x}{\hat v} \bigg(\frac{\phi_x}{\hat v} - \frac{\hat\phi_x}{\hat v^2} \Phi_x \bigg) + e^{\hat\phi} \phi \bigg].
\end{equation}
Since $\Phi,\Psi,\phi\in H^2$, the right-hand side of \eqref{tempPsixx} belongs to $H^1$ by \eqref{spbound}. Consequently, $\Psi\in H^3$. We now have $\lambda\Phi-\Psi_x\in H^2$, so \eqref{3.3a} gives $\Phi\in H^3$ as well. This proves the lemma.
\end{proof}

\subsubsection{Correspondence between eigenfunctions of $L$ and $\mathcal{L}$} \label{sec:eigenfunction_correspondence}

Fix $\lambda\in\Omega$. By Proposition~\ref{prop:ess} and the definition of the essential spectrum in Section~\ref{sec:terminology}, the operator $L-\lambda I$ is Fredholm of index zero. Hence, $\lambda\in\sigma_{\mathrm{pt}}(L)$ if and only if $(L-\lambda I)U=0$ has a nontrivial solution $U\in D(L)$. We show that such solutions $U$ correspond to nontrivial solutions $V\in D(\mathcal L)$ of $(\mathcal L-\lambda I)V=0$, with $V$ given by \eqref{antiv} and $U=V_x$.

Suppose first that $U=(v,u)^{\mathrm{tr}}\in D(L)$ is a nontrivial solution of $(L-\lambda I)U=0$. By Lemma~\ref{theorem:6.1.1}\emph{(iii)}, we have $\phi=\mathcal A^{-1}\mathcal Bv \in H^2$ and
\begin{equation*}
Z:=\left(v,u,\frac{\nu}{\hat v}u_x,\phi,\phi_x\right)^{\mathrm{tr}}\in H^1(\mathbb R;\mathbb C^5).
\end{equation*}
Note that $Z$ satisfies the first-order system \eqref{first_order_normalization} by \eqref{first_order_local_operator} and \eqref{first_order_operator_identity}. For $\lambda\in\Omega$, the limiting matrices $\mathbb{A}_\pm$ of this system are hyperbolic and have the same number of eigenvalues with positive real part, as shown in Section~\ref{sec:proof_ess}. The system therefore has exponential dichotomies on both half-lines with equal Morse indices (see the proof of Lemma~\ref{lemma:Fredholm_limits}). Since $Z$ is a nontrivial $H^1$ solution of this system, it follows from \cite[Theorem~3.2]{Sandstede} that $Z$ decays exponentially as $x\to\pm\infty$. Moreover, by \eqref{first_order_normalization} and the boundedness of $\tilde{\mathbb A}(x,\lambda)$, we have $|Z_x(x)|\le C_\lambda |Z(x)|$, so $Z_x$ also decays exponentially at both ends.

Integrating the first two equations of \eqref{EE''} over $\mathbb R$ and using the decay of $Z$ and $Z_x$, we obtain
\begin{equation*}
\int_{\mathbb R}v(z)\,dz=0 \quad \text{and} \quad \int_{\mathbb R}u(z)\,dz=0
\end{equation*}
since $\lambda\ne0$. Therefore we may write the anti-derivatives in \eqref{antiv} as
\begin{equation*}
\Phi(x)=\int_{-\infty}^x v(z)\,dz=-\int_x^\infty v(z)\,dz, \quad \Psi(x)=\int_{-\infty}^x u(z)\,dz=-\int_x^\infty u(z)\,dz.
\end{equation*}
These representations show that $\Phi$ and $\Psi$ decay exponentially at both ends and hence belong to $L^2(\mathbb R)$. Together with $V_x=U\in H^1\times H^2$, this gives $V=(\Phi,\Psi)^{\mathrm{tr}}\in H^2\times H^3\subset D(\mathcal L)$. Integrating the first two equations of \eqref{EE''} from $-\infty$ to $x$ and using the decay established above, we obtain \eqref{3.3a}--\eqref{3.3b}. Together with \eqref{3.3c} and the definition \eqref{integrated_operator}, this gives $(\mathcal L-\lambda I)V=0$. This solution is nontrivial since $V_x=U\ne0$.

Conversely, suppose that $V=(\Phi,\Psi)^{\mathrm{tr}}\in D(\mathcal L)$ is a nontrivial solution of $(\mathcal L-\lambda I)V=0$. By Lemma~\ref{lemma:integrated_regularity}, we have $\Phi,\Psi\in H^3$ and $\phi=\mathcal A^{-1}\mathcal B\Phi_x\in H^2$, so $U:=V_x$ belongs to $H^2\times H^2\subset D(L)$. Differentiating \eqref{3.3a}--\eqref{3.3b} in the sense of distributions and using \eqref{3.3c}, we see that $U=V_x$ and $\phi$ satisfy \eqref{EE''}. Hence, $(L-\lambda I)U=0$ by Lemma~\ref{lemma:L_domain}. Since $V\in H^1$, we have $V(x)\to0$ as $x\to-\infty$. Integrating $U=V_x$ then gives $\textstyle V(x)=\int_{-\infty}^x U(z)\,dz$ as in \eqref{antiv}. In particular, $V\ne0$ implies $U\ne0$. This proves the asserted correspondence.

\subsection{Spectral energy estimates}

We now carry out energy estimates for \eqref{3.3a}--\eqref{3.3c} on the unstable half-plane $\Omega$. The following lemma can be viewed as a linearized, complex-valued analogue of the energy estimates in \cite[Section~4.2]{DLZ}, which were used to establish nonlinear stability for zero mass perturbations.

\begin{lemma} \label{Lemma:3.2}
Assume that the hypotheses of Theorem~\ref{Main} hold. Let $V=(\Phi,\Psi)^{\mathrm{tr}}\in D(\mathcal L)$ be an eigenfunction of $\mathcal L$ with eigenvalue $\lambda\in\Omega$. Then there exists a constant $c>0$ such that
\begin{equation} \label{3.4}
\operatorname{Re}{\lambda} \lVert (\Phi,\Psi) \rVert_{H^2(\mathbb{R})}^2 + c \left( \lVert \lvert \hat{v}_x \rvert^{1/2} \Psi \rVert_{L^2(\mathbb{R})}^2 + \lVert \Phi_x \rVert_{H^1(\mathbb{R})}^2 + \lVert \Psi_x \rVert_{H^2(\mathbb{R})}^2 \right) \leq 0.
\end{equation}
\end{lemma}

This lemma, together with the eigenfunction correspondence established in Section~\ref{sec:eigenfunction_correspondence}, implies Proposition~\ref{Prop:3.1}. Indeed, \eqref{3.4} yields $(\Phi,\Psi)\equiv(0,0)$, contradicting the assumption that $V$ is an eigenfunction of $\mathcal L$.

\begin{proof}
Throughout the subsequent computations, we denote by $C>0$ a generic constant, and by $\eta>0$ an arbitrarily small parameter (arising from applications of Young's inequality). Both may vary from line to line but remain independent of $\delta_S$, $\lambda$, and $(\Phi,\Psi)$.

We first derive a zeroth-order estimate. Multiplying \eqref{3.3a} by $(T+1)\overline{\Phi}$ and \eqref{3.3b} by $\hat{v}^2\overline{\Psi}$, respectively, and integrating over $\mathbb{R}$, we obtain
\begin{align*}
& \lambda \int_\mathbb{R} (T+1) |\Phi|^2 \, dx - \int_\mathbb{R} s(T+1) \overline{\Phi}\Phi_x \, dx - \int_\mathbb{R} (T+1) \overline{\Phi}\Psi_x \, dx =0, \\
\begin{split}
& \lambda \int_\mathbb{R} \hat{v}^2 |\Psi|^2 \, dx - \int_\mathbb{R} s \hat{v}^2 \overline{\Psi}\Psi_x \, dx - \int_\mathbb{R} (T+1) \overline{\Psi}\Phi_x \, dx - \int_\mathbb{R} \nu \hat{v} \overline{\Psi} \Psi_{xx} \, dx \\
& \quad = - \int_\mathbb{R} \nu \hat{u}_x \overline{\Psi}\Phi_x  \, dx + \int_\mathbb{R} \varepsilon^2 \overline{\Psi} \bigg( \hat{\phi}_x \phi_x -\frac{\hat{\phi}_x^2}{\hat{v}} \Phi_x \bigg) \, dx \\
& \qquad  - \int_\mathbb{R} \varepsilon^2 \overline{\Psi} \bigg[ \bigg( -\frac{2\hat{\phi}_{xx}}{\hat{v}} +\frac{3\hat{\phi}_x\hat{v}_x}{\hat{v}^2}\bigg)\Phi_x -\frac{\hat{\phi}_x}{\hat{v}}\Phi_{xx} -\frac{\hat{v}_x}{\hat{v}}\phi_x +\phi_{xx} \bigg] \, dx.
\end{split}
\end{align*}
Combining these two identities and taking the real part, we obtain, after integration by parts,
\begin{equation} \label{3.6}
\begin{split}
&\operatorname{Re}{\lambda} \int_\mathbb{R} \left( (T+1) \lvert \Phi \rvert^2 + \hat{v}^2 \lvert \Psi \rvert^2 \right) \, dx + \int_\mathbb{R} s\hat{v}\hat{v}_x \lvert \Psi \rvert^2 \, dx + \int_\mathbb{R} \nu \hat{v} \lvert \Psi_x \rvert^2 \, dx \\
& \quad  = - \int_\mathbb{R} \nu \hat{v}_x \operatorname{Re}{(\overline{\Psi}\Psi_x)} \, dx - \int_\mathbb{R} \nu \hat{u}_{x}\operatorname{Re}{( \overline{\Psi}\Phi_x)} \, dx + I,
\end{split}
\end{equation}
where
\begin{equation*}
I = \operatorname{Re} \int_\mathbb{R} \varepsilon^2 \overline{\Psi} \bigg[ \bigg(\frac{2\hat{\phi}_{xx}}{\hat{v}}- \frac{\hat{\phi}_x^2}{\hat{v}} -\frac{3\hat{\phi}_x\hat{v}_x}{\hat{v}^2} \bigg)\Phi_x +\frac{\hat{\phi}_x}{\hat{v}}\Phi_{xx} + \bigg( \hat{\phi}_x +\frac{\hat{v}_x}{\hat{v}} \bigg) \phi_x -\phi_{xx} \bigg] \, dx.
\end{equation*}
The first two terms on the right-hand side of \eqref{3.6} can be estimated by applying Young's inequality, together with the bounds \eqref{sprel}--\eqref{spbound} and the positivity of the parameters, as
\begin{equation} \label{3.7}
\begin{split}
& \left| - \int_\mathbb{R} \nu \hat{v}_x \operatorname{Re}{(\overline{\Psi}\Psi_x)} \, dx - \int_\mathbb{R} \nu \hat{u}_{x} \operatorname{Re}{( \overline{\Psi}\Phi_x)} \, dx \right|\\
& \quad \leq \eta \int_\mathbb{R} s \hat{v} \hat{v}_x |\Psi|^2 \, dx + \frac{C}{\eta} \delta_S \int_\mathbb{R} \left( \lvert \Psi_x \rvert^2 +\lvert \Phi_x \rvert^2 \right) \, dx,
\end{split}
\end{equation}
where $\delta_S = |v_+-v_-|$ as denoted in Lemma~\ref{SP}. Similarly, the remaining term $I$ satisfies
\begin{equation} \label{3.8}
\begin{split}
I \leq \eta \int_\mathbb{R} s\hat{v}\hat{v}_x |\Psi|^2 \, dx + \frac{C}{\eta} \delta_S \int_\mathbb{R} \left( |\Phi_x|^2 + |\Phi_{xx}|^2 + |\phi_x|^2 \right) \, dx - \int_\mathbb{R} \varepsilon^2 \operatorname{Re}{(\overline{\Psi} \phi_{xx})} \, dx.
\end{split}
\end{equation}
To estimate the last term, we rewrite it using \eqref{3.3a}:
\begin{equation} \label{3.9}
\begin{split}
- \int_\mathbb{R} \varepsilon^2 \operatorname{Re}{\left( \overline{\Psi}\phi_{xx} \right)} \, dx & = \int_\mathbb{R} \varepsilon^2 \operatorname{Re}{(\Psi_x \overline{\phi_x})} \, dx \\
& = \int_\mathbb{R} \varepsilon^2 \operatorname{Re}{(\lambda \Phi \overline{\phi_x})} \, dx - \int_\mathbb{R} s \varepsilon^2 \operatorname{Re}{( \Phi_x \overline{\phi_x})} \, dx.
\end{split}
\end{equation}
The second term can be further rewritten by using \eqref{3.3c}. From \eqref{3.3c}, we have
\begin{equation*}
\Phi_x = - \hat{v} \phi + \varepsilon^2 e^{-\hat{\phi}} \bigg( \frac{\phi_x}{\hat{v}} - \frac{\hat{\phi}_x}{\hat{v}^2} \Phi_x \bigg)_x.
\end{equation*}
Substituting this identity into the second term above, we obtain
\begin{equation*}
\begin{split}
- \int_\mathbb{R} s \varepsilon^2 \operatorname{Re}{(\Phi_x \overline{\phi_x})} \, dx & = \int_\mathbb{R} s \varepsilon^2 \hat{v} \operatorname{Re}{(\phi \overline{\phi_x})} \, dx - \int_\mathbb{R} \frac{s \varepsilon^4 e^{-\hat{\phi}}}{ \hat{v}} \operatorname{Re}{( \phi_{xx} \overline{\phi_x})} \, dx \\
& \quad + \int_\mathbb{R} s \varepsilon^4 e^{-\hat{\phi}} \operatorname{Re}{ \bigg[ \overline{\phi_x} \bigg( \frac{\hat{v}_x\phi_x}{\hat{v}^2} + \bigg( \frac{\hat{\phi}_x}{\hat{v}^2} \bigg)_x \Phi_x + \frac{\hat{\phi}_x}{\hat{v}^2} \Phi_{xx} \bigg) \bigg]} \, dx \\
& = - \int_\mathbb{R} \frac{s \varepsilon^2 \hat{v}_x}{2} |\phi|^2 \, dx + \int_\mathbb{R} \frac{s \varepsilon^4}{2} \bigg( \frac{e^{-\hat{\phi}}}{\hat{v}} \bigg)_x |\phi_x|^2 \, dx \\
& \quad + \int_\mathbb{R} s \varepsilon^4 e^{-\hat{\phi}} \operatorname{Re}{ \bigg[ \overline{\phi_x} \bigg( \frac{\hat{v}_x\phi_x}{\hat{v}^2} + \bigg( \frac{\hat{\phi}_x}{\hat{v}^2} \bigg)_x \Phi_x + \frac{\hat{\phi}_x}{\hat{v}^2} \Phi_{xx} \bigg) \bigg]} \, dx.
\end{split}
\end{equation*}
It then follows by applying Young's inequality together with \eqref{spbound} that
\begin{equation} \label{3.10}
- \int_\mathbb{R} s \varepsilon^2 \operatorname{Re}{(\Phi_x \overline{\phi_x})} \, dx \leq C \delta_S \int_\mathbb{R} \left( |\Phi_x|^2 + |\Phi_{xx}|^2 + |\phi|^2 + |\phi_x|^2 \right) \, dx.
\end{equation}
Collecting \eqref{3.6}--\eqref{3.10}, we have
\begin{equation} \label{3.11}
\begin{split}
& \operatorname{Re}{\lambda} \int_\mathbb{R} \left( (T+1)|\Phi|^2 + \hat{v}^2|\Psi|^2 \right) \, dx - \int_\mathbb{R} \varepsilon^2 \operatorname{Re}{(\lambda \Phi \overline{\phi_x})} \, dx \\
& \qquad + \int_\mathbb{R} s\hat{v}\hat{v}_x \lvert \Psi \rvert^2 \, dx + \int_\mathbb{R} \nu \hat{v} \lvert \Psi_x \rvert^2 \, dx \\
& \quad \leq \eta \int_\mathbb{R} s \hat{v} \hat{v}_x |\Psi|^2 \, dx + \frac{C}{\eta} \delta_S \int_\mathbb{R} \left( |\Phi_x|^2 + |\Phi_{xx}|^2 + |\Psi_x|^2 + |\phi|^2 + |\phi_x|^2 \right) \, dx
\end{split}
\end{equation}
for arbitrarily small $\eta>0$. To treat the coupling term $ \textstyle -\int \varepsilon^2 \operatorname{Re}{(\lambda \Phi \overline{\phi_x})}$ on the left-hand side, we derive a quadratic form with the aid of the linearized Poisson equation. Multiplying the complex conjugate of \eqref{3.3c} by $\varepsilon^2 \lambda e^{-\hat{\phi}}\phi$, then taking the real part and integrating over $\mathbb{R}$, we obtain
\begin{equation} \label{3.12}
\begin{split}
& - \int_\mathbb{R} \varepsilon^2 \operatorname{Re}{(\lambda \overline{\Phi} \phi_x)} \, dx +  \operatorname{Re}{\lambda} \int_\mathbb{R} \varepsilon^2 \hat{v} |\phi|^2 \, dx + \operatorname{Re}{\lambda} \int_\mathbb{R} \frac{\varepsilon^4e^{-\hat{\phi}}}{\hat{v}} |\phi_{x}|^2 \, dx \\
& \quad = - \int_\mathbb{R} \varepsilon^4 \bigg( \frac{e^{-\hat{\phi}}}{\hat{v}}\bigg)_x \operatorname{Re}{( \lambda \phi \overline{\phi_{x}})} \, dx \\
& \qquad - \int_\mathbb{R} \varepsilon^4 e^{-\hat{\phi}} \operatorname{Re}{ \bigg[ \lambda \phi \bigg( \frac{\hat{v}_x \overline{\phi_x}}{\hat{v}^2} + \bigg( \frac{\hat{\phi}_x}{\hat{v}^2}\bigg)_x \overline{\Phi_x} + \frac{\hat{\phi}_x}{\hat{v}^2} \overline{\Phi_{xx}} \bigg) \bigg] } \, dx.
\end{split}
\end{equation}
As in \eqref{3.10}, the right-hand side can be bounded by
\begin{equation*}
C \delta_S \int_\mathbb{R} \left( |\Phi_x|^2 + |\Phi_{xx}|^2 + |\lambda \phi|^2 + |\phi_x|^2 \right) \, dx.
\end{equation*}
Thus, by adding the identity \eqref{3.12} to \eqref{3.11}, we obtain
\begin{equation} \label{3.13}
\begin{split}
& \operatorname{Re}{\lambda} \int_\mathbb{R} \bigg( (T+1) |\Phi|^2 + \hat{v}^2 |\Psi|^2 - 2 \varepsilon^2 \operatorname{Re}{(\Phi \overline{\phi_x})} + \varepsilon^2\hat{v} |\phi|^2 + \frac{\varepsilon^4e^{-\hat{\phi}}}{\hat{v}}|\phi_x|^2 \bigg) \, dx \\
& \qquad + \int_\mathbb{R} s\hat{v}\hat{v}_x \lvert \Psi \rvert^2 \, dx + \int_\mathbb{R} \nu \hat{v}\lvert \Psi_x \rvert^2 \, dx \\
& \quad \leq \eta \int_\mathbb{R} s \hat{v} \hat{v}_x |\Psi|^2 \, dx + \frac{C}{\eta} \delta_S \int_\mathbb{R} \left( |\Phi_x|^2 + |\Phi_{xx}|^2 + |\Psi_x|^2 + |\phi|^2 + |\phi_x|^2 + |\lambda \phi|^2 \right) \, dx.
\end{split}
\end{equation}
Consider the quadratic form on the left-hand side:
\begin{equation} \label{3.14}
(T+1)|\Phi|^2 - 2\varepsilon^2 \operatorname{Re}{(\Phi \overline{\phi_x})} + \frac{\varepsilon^4 e^{-\hat{\phi}}}{\hat{v}} |\phi_x|^2.
\end{equation}
Its discriminant $D$ is given by
\begin{equation*}
D = 4\varepsilon^4 - 4\varepsilon^4(T+1)\frac{e^{-\hat{\phi}}}{\hat{v}} = 4 \varepsilon^4 \left(1-\frac{e^{-\hat{\phi}}}{\hat{v}}-T\frac{e^{-\hat{\phi}}}{\hat{v}} \right).
\end{equation*}
Here we note that, by the third equation of \eqref{waveeq} and the bounds \eqref{spbound},
\begin{equation} \label{3.15}
\bigg| 1- \frac{e^{-\hat{\phi}}}{\hat{v}} \bigg| = \bigg| \varepsilon^2 \frac{e^{-\hat{\phi}}}{\hat{v}} \bigg(\frac{\hat{\phi}_x}{\hat{v}} \bigg)_x \bigg| \leq C \delta_S.
\end{equation}
Hence, for sufficiently small $\delta_S$, the discriminant $D$ is strictly negative, and thus the quadratic form \eqref{3.14} is uniformly positive definite in $x$; in particular, there exists a constant $c_0>0$ satisfying
\begin{equation*}
(T+1)|\Phi|^2 - 2\varepsilon^2 \operatorname{Re}{(\Phi \overline{\phi_x})} + \frac{\varepsilon^4 e^{-\hat{\phi}}}{\hat{v}} |\phi_x|^2 \geq c_0 \left( |\Phi|^2 + |\phi_x|^2 \right).
\end{equation*}
Consequently, inserting this bound into \eqref{3.13} and choosing $\eta>0$ sufficiently small, we deduce
\begin{equation} \label{L2_est}
\begin{split}
& \operatorname{Re}{\lambda} \int_\mathbb{R} \left( |\Phi|^2 + \hat{v}^2 |\Psi|^2 + |\phi|^2 + |\phi_x|^2 \right) \, dx + \int_\mathbb{R} s\hat{v}\hat{v}_x |\Psi|^2 \, dx + \int_\mathbb{R} |\Psi_x|^2 \, dx \\
& \quad \leq C \delta_S \left( \lVert \Phi_x \rVert_{H^1}^2 +\lVert \Psi_x \rVert_{L^2}^2 +\lVert \phi \rVert_{H^1}^2 + \lVert \lambda \phi \rVert_{L^2}^2 \right)
\end{split}
\end{equation}
for sufficiently small $\delta_S < \delta_0$.

Next we estimate $\lVert \Phi_x \rVert_{L^2}$. Differentiating \eqref{3.3a} once and multiplying the resulting equation by $\overline{\Phi_x}$ and $- \nu^{-1} \hat{v}\overline{\Psi}$, we have
\begin{equation} \label{eq1}
\lambda |\Phi_x|^2 - s \Phi_{xx}\overline{\Phi_x} - \Psi_{xx} \overline{\Phi_x} = 0
\end{equation}
and
\begin{equation} \label{eq2}
- \frac{\hat{v}}{\nu} \lambda \Phi_x \overline{\Psi} + \frac{s \hat{v}}{\nu} \Phi_{xx}\overline{\Psi} + \frac{\hat{v}}{\nu} \Psi_{xx}\overline{\Psi} = 0,
\end{equation}
respectively. Multiplying \eqref{3.3b} by $-\nu^{-1} \hat{v} \overline{\Phi_x}$, we obtain in addition
\begin{equation} \label{eq3}
\begin{split} 
& - \frac{\hat{v}}{\nu} \lambda \Psi \overline{\Phi_x} + \frac{s \hat{v}}{\nu} \Psi_x \overline{\Phi_x} + \frac{(T+1)}{\nu\hat{v}} |\Phi_x|^2 + \Psi_{xx} \overline{\Phi_x} \\
& \quad = \frac{ \hat{u}_x}{\hat{v}} |\Phi_x|^2 - \frac{\varepsilon^2 \hat{v}}{\nu} \bigg( \frac{\hat{\phi}_x}{\hat{v}^2}\phi_x \overline{\Phi_x} - \frac{\hat{\phi}_x^2}{\hat{v}^3} |\Phi_x|^2 \bigg) \\
& \qquad + \frac{\varepsilon^2}{\nu} \bigg[ \bigg( -\frac{2\hat{\phi}_{xx}}{\hat{v}^2} +\frac{3\hat{\phi}_x\hat{v}_x}{\hat{v}^3} \bigg) |\Phi_x|^2 -\frac{\hat{\phi}_x}{\hat{v}^2}\Phi_{xx}\overline{\Phi_x} -\frac{\hat{v}_x}{\hat{v}^2}\phi_x\overline{\Phi_x} +\frac{\phi_{xx}}{\hat{v}}\overline{\Phi_x} \bigg].
\end{split}
\end{equation}
Combining the equations \eqref{eq1}--\eqref{eq3}, taking the real part, and integrating over $\mathbb{R}$, we obtain
\begin{equation} \label{3.17}
\begin{split}
& \operatorname{Re}{\lambda} \int_\mathbb{R} \left( |\Phi_x|^2 - \frac{2\hat{v}}{\nu} \operatorname{Re}{(\Psi \overline{\Phi_x})} \right) \, dx + \int_\mathbb{R} \frac{(T+1)}{\nu\hat{v}} |\Phi_x|^2 \, dx - \int_\mathbb{R} \frac{\varepsilon^2}{\nu\hat{v}} \operatorname{Re}{(\overline{\phi_{xx}}\Phi_x)} \, dx \\
& \quad = \int_\mathbb{R} \frac{s \hat{v}_x}{\nu} \operatorname{Re}{(\Phi_{x}\overline{\Psi})} \, dx + \int_\mathbb{R} \frac{\hat{v}_x}{\nu} \operatorname{Re}{(\Psi_{x}\overline{\Psi})} \, dx + \int_\mathbb{R} \frac{\hat{v}}{\nu} |\Psi_x|^2 \, dx + \int_\mathbb{R} \frac{\hat{u}_x}{\hat{v}} |\Phi_x|^2 \, dx \\
& \qquad - \int_\mathbb{R} \frac{\varepsilon^2 \hat{\phi}_x}{\nu\hat{v}} \bigg( \operatorname{Re}{(\phi_x \overline{\Phi_x})} - \frac{\hat{\phi}_x}{\hat{v}} |\Phi_x|^2 \bigg) \, dx \\
& \qquad + \int_\mathbb{R} \frac{\varepsilon^2}{\nu} \operatorname{Re}{\bigg[ \bigg( -\frac{2\hat{\phi}_{xx}}{\hat{v}^2} + \frac{3\hat{\phi}_x\hat{v}_x}{\hat{v}^3} \bigg) |\Phi_x|^2 -\frac{\hat{\phi}_x}{\hat{v}^2}\Phi_{xx}\overline{\Phi_x} -\frac{\hat{v}_x}{\hat{v}^2}\phi_x\overline{\Phi_x} \bigg]} \, dx \\
& \quad =: II,
\end{split}
\end{equation}
where we have used
\begin{equation*}
\begin{split}
& - \int_\mathbb{R} \frac{s \hat{v}}{\nu} \operatorname{Re}{(\Phi_{xx}\overline{\Psi})} \, dx - \int_\mathbb{R} \frac{\hat{v}}{\nu} \operatorname{Re}{(\Psi_{xx}\overline{\Psi})} \, dx \\
& \quad = \int_\mathbb{R} \frac{s \hat{v}_x}{\nu} \operatorname{Re}{(\Phi_{x}\overline{\Psi})} \, dx + \int_\mathbb{R} \frac{s \hat{v}}{\nu} \operatorname{Re}{(\Phi_{x}\overline{\Psi_x})} \, dx + \int_\mathbb{R} \frac{\hat{v}_x}{\nu} \operatorname{Re}{(\Psi_{x}\overline{\Psi})} \, dx + \int_\mathbb{R} \frac{\hat{v}}{\nu} |\Psi_x|^2 \, dx.
\end{split}
\end{equation*}
By Young's inequality, the term $II$ is bounded as
\begin{equation} \label{IIbd}
\begin{split}
|II| & \leq \eta \int_\mathbb{R} s\hat{v}\hat{v}_x |\Psi|^2 \, dx + \frac{C}{\eta} \delta_S \int_\mathbb{R} |\Phi_x|^2 \, dx \\
& \quad + C \delta_S \int_\mathbb{R} \left( |\Phi_x|^2 + |\Phi_{xx}|^2 + |\phi_x|^2 \right) \, dx + C \int_\mathbb{R} |\Psi_x|^2 \, dx.
\end{split}
\end{equation}
Analogously to \eqref{3.12}, we obtain the following inequality by multiplying \eqref{3.3c} by $- \varepsilon^2 \nu^{-1} \hat{v}^{-1} e^{-\hat{\phi}} \overline{\phi_{xx}}$:
\begin{equation} \label{quad2}
\begin{split}
& - \int_\mathbb{R} \frac{\varepsilon^2}{\nu\hat{v}} \operatorname{Re}{(\overline{\phi_{xx}} \Phi_x)} \, dx + \int_\mathbb{R} \frac{\varepsilon^2}{\nu} |\phi_x|^2 \, dx + \int_\mathbb{R} \frac{\varepsilon^4 e^{-\hat{\phi}}}{\nu\hat{v}^2} |\phi_{xx}|^2 \, dx \\
& \quad = \int_\mathbb{R} \frac{\varepsilon^4 e^{-\hat{\phi}}}{\nu\hat{v}} \operatorname{Re}{ \bigg[ \overline{\phi_{xx}} \bigg( \frac{\hat{v}_x \phi_x}{\hat{v}^2} + \bigg( \frac{\hat{\phi}_x}{\hat{v}^2} \bigg)_x \Phi_x + \frac{\hat{\phi}_x}{\hat{v}^2} \Phi_{xx} \bigg) \bigg] } \, dx \\
& \quad \leq C \delta_S \int_\mathbb{R} \left( |\Phi_x|^2 + |\Phi_{xx}|^2 + |\phi_x|^2 + |\phi_{xx}|^2 \right) \, dx.
\end{split}
\end{equation}
Combining \eqref{3.17}--\eqref{quad2}, we have
\begin{equation*}
\begin{split}
& \operatorname{Re}{\lambda} \int_\mathbb{R} \left( |\Phi_x|^2- \frac{2\hat{v}}{\nu} \operatorname{Re}{(\Psi \overline{\Phi_x})} \right) \, dx + \int_\mathbb{R} \frac{\varepsilon^2}{\nu} |\phi_x|^2 \, dx \\
& \qquad + \int_\mathbb{R} \frac{(T+1)}{\nu\hat{v}} |\Phi_x|^2 \, dx - \int_\mathbb{R} \frac{2\varepsilon^2}{\nu\hat{v}} \operatorname{Re}{(\overline{\phi_{xx}}\Phi_x)} \, dx + \int_\mathbb{R} \frac{\varepsilon^4 e^{-\hat{\phi}}}{\nu\hat{v}^2} |\phi_{xx}|^2 \, dx \\
& \quad \leq \eta \int_\mathbb{R} s\hat{v}\hat{v}_x |\Psi|^2 \, dx + \frac{C}{\eta} \delta_S \int_\mathbb{R} |\Phi_x|^2 \, dx + C \int_\mathbb{R} |\Psi_x|^2 \, dx \\
& \qquad + C \delta_S \int_\mathbb{R} \left( |\Phi_x|^2 + |\Phi_{xx}|^2 + |\phi_x|^2 + |\phi_{xx}|^2 \right) \, dx.
\end{split}
\end{equation*}
Notice that, for sufficiently small $\delta_S$ and some constant $c_1>0$, \eqref{3.15} yields
\begin{equation*}
\frac{(T+1)}{\nu\hat{v}} |\Phi_x|^2 - \frac{2\varepsilon^2}{\nu\hat{v}} \operatorname{Re}{(\overline{\phi_{xx}}\Phi_x)} + \frac{\varepsilon^4 e^{-\hat{\phi}}}{\nu\hat{v}^2} |\phi_{xx}|^2 \geq c_1 \left( |\Phi_x|^2 + |\phi_{xx}|^2 \right).
\end{equation*}
Consequently, we obtain
\begin{equation*}
\begin{split}
& \operatorname{Re}{\lambda} \int_\mathbb{R} \left( |\Phi_x|^2- \frac{2\hat{v}}{\nu} \operatorname{Re}{(\Psi \overline{\Phi_x})} \right) \, dx + \int_\mathbb{R}
\frac{\varepsilon^2}{\nu} |\phi_x|^2 \, dx + c_1 \int_\mathbb{R} \left(  |\Phi_x|^2 + |\phi_{xx}|^2 \right) \, dx \\
& \quad \leq \eta \int_\mathbb{R} s\hat{v}\hat{v}_x |\Psi|^2 \, dx + \frac{C}{\eta} \delta_S \int_\mathbb{R} |\Phi_x|^2 \, dx + C \int_\mathbb{R} |\Psi_x|^2 \, dx \\
& \qquad + C \delta_S \int_\mathbb{R} \left( |\Phi_x|^2 + |\Phi_{xx}|^2 + |\phi_x|^2 + |\phi_{xx}|^2 \right) \, dx.
\end{split}
\end{equation*}
We use \eqref{L2_est} to bound the term $\textstyle \int |\Psi_x|^2$ on the right-hand side. Then,
\begin{equation} \label{tempq}
\begin{split}
& \operatorname{Re}{\lambda} \int_\mathbb{R} \left( |\Phi_x|^2- \frac{2\hat{v}}{\nu} \operatorname{Re}{(\Psi \overline{\Phi_x})} \right) \, dx + \int_\mathbb{R} \frac{\varepsilon^2}{\nu} |\phi_x|^2 \, dx + c_1 \int_\mathbb{R} \left(  |\Phi_x|^2 + |\phi_{xx}|^2 \right) \, dx \\
& \quad \leq \eta \int_\mathbb{R} s\hat{v}\hat{v}_x |\Psi|^2 \, dx + \frac{C}{\eta} \delta_S \int_\mathbb{R} |\Phi_x|^2 \, dx \\
& \qquad + C \delta_S \left( \lVert \Phi_x \rVert_{H^1}^2 +\lVert \Psi_x \rVert_{L^2}^2 +\lVert \phi \rVert_{H^2}^2 + \lVert \lambda \phi \rVert_{L^2}^2 \right).
\end{split}
\end{equation}
Multiplying \eqref{L2_est} by the constant $\tfrac{2}{\nu^2}$, and then adding the result to \eqref{tempq}, we obtain
\begin{equation} \label{basic_est}
\begin{split}
& \operatorname{Re}{\lambda} \lVert (\Phi,\Phi_x,\Psi,\phi,\phi_x) \rVert_{L^2}^2 + \lVert \sqrt{s\hat{v}\hat{v}_x}\Psi \rVert_{L^2}^2 + \lVert (\Phi_x,\Psi_x) \rVert_{L^2}^2 + \lVert \phi_x \rVert_{H^1}^2 \\
& \quad \leq C \delta_S \left( \lVert \Phi_{xx} \rVert_{L^2}^2 + \lVert \phi \rVert_{L^2}^2 + \lVert \lambda \phi \rVert_{L^2}^2 \right)
\end{split}
\end{equation}
for sufficiently small $\delta_S<\delta_0$, where we have also used the fact that there exists a constant $c_2>0$ satisfying
\begin{equation*}
|\Phi_x|^2 - \frac{2\hat{v}}{\nu} \operatorname{Re}{(\Psi \overline{\Phi_x})} + \frac{2\hat{v}^2}{\nu^2} |\Psi|^2 \geq c_2 \left( |\Phi_x|^2 + |\Psi|^2 \right).
\end{equation*}

To complete the proof of the lemma, it remains to deal with the terms appearing on the right-hand side of \eqref{basic_est}. We begin with the terms $\lVert \phi \rVert_{H^1}$ and $\lVert \lambda \phi \rVert_{L^2}$. Multiplying \eqref{3.3c} by $\overline{\phi}$ and taking the real part, we have
\begin{equation*}
-\varepsilon^2 \operatorname{Re}{ \bigg( \overline{\phi} \bigg( \frac{\phi_x}{\hat{v}} - \frac{\hat{\phi}_x}{\hat{v}^2} \Phi_x \bigg)_x \bigg)} = - e^{\hat{\phi}} \operatorname{Re}{(\Phi_x \overline{\phi})} - \hat{v} e^{\hat{\phi}} |\phi|^2.
\end{equation*}
Integrating it over $\mathbb{R}$, we obtain after an application of integration by parts
\begin{equation*}
\begin{split}
\int_\mathbb{R} \bigg( \frac{\varepsilon^2}{\hat{v}} |\phi_x|^2 + \hat{v}e^{\hat{\phi}} |\phi|^2 \bigg) \, dx & = \int_\mathbb{R} \frac{\varepsilon^2 \hat{\phi}_x}{\hat{v}^2} \operatorname{Re}{(\Phi_x \overline{\phi_x})} \, dx - \int_\mathbb{R} e^{\hat{\phi}} \operatorname{Re}{(\Phi_x \overline{\phi})} \, dx.
\end{split}
\end{equation*}
We then apply Young's inequality together with the bounds \eqref{spbound} to obtain
\begin{equation*}
\lVert \phi \rVert_{H^1}^2 \leq C \delta_S \left( \lVert \Phi_x \rVert_{L^2}^2 + \lVert \phi_x \rVert_{L^2}^2 \right) + \eta \lVert \phi \rVert_{L^2}^2 + \frac{C}{\eta} \lVert \Phi_x \rVert_{L^2}^2.
\end{equation*}
This implies that there exists a constant $C>0$ such that
\begin{equation} \label{phi_est}
\lVert \phi \rVert_{H^1}^2 \leq C\lVert \Phi_x \rVert_{L^2}^2
\end{equation}
for sufficiently small $\delta_S < \delta_0$. Analogously, we derive an estimate for $\lVert \lambda \phi \rVert_{L^2}^2$. Multiplying \eqref{3.3c} by $|\lambda|^2 \overline{\phi}$, taking the real part, and integrating over $\mathbb{R}$, we have
\begin{equation*}
\int_\mathbb{R} \left( \frac{\varepsilon^2}{\hat{v}}|\lambda|^2 |\phi_x|^2 + \hat{v} e^{\hat{\phi}} |\lambda|^2 |\phi|^2 \right) \, dx = \int_\mathbb{R} \frac{\varepsilon^2\hat{\phi}_x}{\hat{v}^2} |\lambda|^2 \operatorname{Re}{(\Phi_x \overline{\phi_x})} \, dx - \int_\mathbb{R} e^{\hat{\phi}} |\lambda|^2 \operatorname{Re}{(\Phi_x \overline{\phi})} \, dx.
\end{equation*}
Here, using $|\lambda|^2 = \lambda \overline{\lambda}$ and \eqref{3.3a}, we estimate the right-hand side:
\begin{equation*}
\begin{split}
|R.H.S.| &= \bigg| \int_\mathbb{R} \frac{s\varepsilon^2\hat{\phi}_x}{\hat{v}^2} \operatorname{Re}{(\Phi_{xx} \overline{\lambda} \overline{\phi_x})} \, dx + \int_\mathbb{R} \frac{\varepsilon^2\hat{\phi}_x}{\hat{v}^2} \operatorname{Re}{(\Psi_{xx} \overline{\lambda} \overline{\phi_x})} \, dx \\
& \qquad - \int_\mathbb{R} s e^{\hat{\phi}} \operatorname{Re}{( \Phi_{xx} \overline{\lambda} \overline{\phi})} \, dx - \int_\mathbb{R} e^{\hat{\phi}} \operatorname{Re}{( \Psi_{xx} \overline{\lambda} \overline{\phi})} \, dx \bigg| \\
& \leq C \delta_S \left( \lVert \Phi_{xx} \rVert_{L^2}^2 + \lVert \Psi_{xx} \rVert_{L^2}^2 + \lVert \lambda\phi_x \rVert_{L^2}^2 \right) \\
& \quad + \eta \lVert \lambda \phi \rVert_{L^2}^2 + \frac{C}{\eta} \left( \lVert \Phi_{xx} \rVert_{L^2}^2 + \lVert \Psi_{xx} \rVert_{L^2}^2 \right).
\end{split}
\end{equation*}
Therefore, we deduce that
\begin{equation} \label{lphi_est}
\lVert \lambda \phi \rVert_{H^1}^2 \leq C \left( \lVert \Phi_{xx} \rVert_{L^2}^2 + \lVert \Psi_{xx} \rVert_{L^2}^2 \right)
\end{equation}
for sufficiently small $\delta_S < \delta_0$.

Substituting the bounds \eqref{phi_est} and \eqref{lphi_est} into \eqref{basic_est}, we have
\begin{equation} \label{zeroth}
\begin{split}
& \operatorname{Re}{\lambda} \lVert (\Phi,\Phi_x,\Psi,\phi,\phi_x) \rVert_{L^2}^2 + \lVert \sqrt{s\hat{v}\hat{v}_x}\Psi \rVert_{L^2}^2 + \lVert (\Phi_x,\Psi_x) \rVert_{L^2}^2 + \lVert \phi_x \rVert_{H^1}^2 \\
& \quad \leq C \delta_S \left( \lVert \Phi_{xx} \rVert_{L^2}^2 + \lVert \Psi_{xx} \rVert_{L^2}^2 \right)
\end{split}
\end{equation}
for sufficiently small $\delta_S < \delta_0$. To close the estimate, we now combine \eqref{zeroth} with the higher-order estimates established in Appendix~\ref{App_C}. Indeed, a suitable linear combination of \eqref{uL2}--\eqref{vxL2} yields
\begin{equation} \label{high}
\begin{split}
\operatorname{Re}{\lambda} \lVert (\Phi_{xx},\Psi_x,\Psi_{xx}) \rVert_{L^2}^2 + \lVert \Phi_{xx} \rVert_{L^2}^2 + \lVert \Psi_{xx} \rVert_{H^1}^2 \leq C \left( \lVert \Phi_x \rVert_{L^2}^2 + \lVert \Psi_x \rVert_{L^2}^2 + \lVert \phi_x \rVert_{L^2}^2 \right).
\end{split}
\end{equation}
Combining \eqref{zeroth} and \eqref{high}, we obtain
\begin{equation*}
\begin{split}
& \operatorname{Re}{\lambda} \left( \lVert (\Phi,\Psi) \rVert_{H^2}^2 + \lVert \phi \rVert_{H^1}^2 \right) + \lVert \sqrt{s\hat{v}\hat{v}_x}\Psi \rVert_{L^2}^2 + \lVert \Phi_x \rVert_{H^1}^2 + \lVert \Psi_x \rVert_{H^2}^2 + \lVert \phi_x \rVert_{H^1}^2 \\
& \quad \leq C\delta_S(\|\Phi_{xx}\|_{L^2}^2+\|\Psi_{xx}\|_{L^2}^2).
\end{split}
\end{equation*}
For $\delta_S$ sufficiently small, the right-hand side can be absorbed into the left-hand side, which implies \eqref{3.4}. This completes the proof of the lemma.
\end{proof}

\section{Simplicity of the zero eigenvalue} \label{sec:Simplicity}

In this section, we prove simplicity of the zero eigenvalue of the linearized operator $L$ by computing the first derivative of the associated Evans function at $\lambda=0$. Following \cite{GZ, MZ1, ZH}, we begin by rewriting the eigenvalue equation $(L-\lambda I)U=0$ as the first-order ODE system
\begin{equation}\label{5}
W' = \mathbb{A}(x,\lambda) W, \qquad x\in\mathbb{R},\ \lambda\in\mathbb{C},
\end{equation}
in the coordinates:
\begin{equation} \label{coord}
W = (v, u, \tilde{u}', \phi, \phi')^{\mathrm{tr}},
\end{equation}
where $\tilde{u} = b_1 v + b_2 u$, $b_j$ denotes the $(2,j)$-component of the matrix $B$ defined in \eqref{ABDE}, and $\phi = \mathcal{A}^{-1}\mathcal{B}v$ with $\mathcal{A}, \mathcal{B}$ in \eqref{calAB}. The coefficient matrix $\mathbb{A}$ is defined by the relations
\begin{subequations} \label{coordinate}
\begin{align}
\label{co1} v' & = \frac{\lambda}{s} v - \frac{1}{s}u', \\
\label{co2} u' & = \frac{1}{b_2} \left( \tilde{u}' -b_1' v - b_1v' - b_2'u \right), \\
\label{co3} \tilde{u}'' &= \Big( (b_1'-A_{21})v + (b_2'-s)u - D_{21}\phi' - E_{21}\phi'' \Big)' + \lambda u, \\
\label{co4} \phi'' & = \bigg( \frac{\bar{\phi}''}{\bar{v}} - \frac{2\bar{v}'\bar{\phi}'}{\bar{v}^2} + \frac{\bar{v}e^{\bar{\phi}}}{\varepsilon^2} \bigg) v + \frac{\bar{\phi}'}{\bar{v}} v' + \frac{\bar{v}^2e^{\bar{\phi}}}{\varepsilon^2} \phi + \frac{\bar{v}'}{\bar{v}}\phi',
\end{align}
\end{subequations}
and the limits $\textstyle \lim_{x\to\pm\infty}\mathbb A(x,\lambda)$ are given by the matrices $\mathbb A_\pm(\lambda)$ defined in \eqref{A_pm}.\footnote{At either constant end state, $b_1=0$ and $b_2=\nu/v_\pm$, so that $\tilde u'=(\nu/v_\pm)u'$. Thus the coordinates in \eqref{coord} agree with those used in \eqref{first_order_normalization}, and the two systems have the same limiting matrices.} By Lemma~\ref{SP}, the matrix $\mathbb{A}(x,\lambda)$ satisfies the assumption \textup{(H0)} in Appendix~\ref{Appendix_A}. We briefly indicate how the relations in \eqref{coordinate} define the coefficient matrix $\mathbb{A}(x,\lambda)$. First, since $sb_2-b_1\neq 0$, the first two relations \eqref{co1}--\eqref{co2} determine $(v',u')$ uniquely in terms of $W$. Next, differentiating \eqref{co4} and using the resulting expression for $v'$ allows us to write $\phi'''$ in terms of $W$ and $\tilde{u}''$. Substituting these expressions for $v',u',\phi'',$ and $\phi'''$ into \eqref{co3}, and using $\textstyle \frac{sb_2}{sb_2-b_1} \neq 0$, we obtain that $\tilde{u}''$ is uniquely determined by $W$. Consequently, \eqref{coordinate} closes to a first-order system of the form \eqref{5}.

We introduce notation for the eigenvalues and eigenvectors of $A_\pm$ and define the associated scalar diffusion rates using $B_\pm$. Here $A_\pm$ and $B_\pm$ are the constant matrices appearing in the limiting operators $L_\pm$; see \eqref{Lpm}. Let $a_j^\pm$ denote the eigenvalues of the matrix
\begin{equation*}
A_\pm  = \begin{pmatrix}
s & 1 \\
\tfrac{T+1}{v_\pm^2} & s
\end{pmatrix}.
\end{equation*}
Then
\begin{equation}\label{aj}
a_j^\pm = s + (-1)^j \sqrt{\frac{T+1}{v_\pm^2}}, \quad j=1,2.
\end{equation}
We choose the corresponding left and right eigenvectors as
\begin{equation}\label{lrj}
l_j^\pm = \begin{pmatrix}
\frac{1}{\sqrt{2}} & \frac{(-1)^j}{\sqrt{2}} \sqrt{\frac{v_\pm^2}{T+1}}
\end{pmatrix}, \quad
r_j^\pm = \begin{pmatrix}
\frac{1}{\sqrt{2}} & \frac{(-1)^j}{\sqrt{2}} \sqrt{\frac{T+1}{v_\pm^2}}
\end{pmatrix}^{\mathrm{tr}}, \quad j=1,2,
\end{equation}
respectively, normalized so that $l_i^\pm r_j^\pm = \delta_{ij}$. We define the scalar diffusion rates
\begin{equation}\label{betaj}
\beta_j^\pm := l_j^\pm B_\pm r_j^\pm = \frac{\nu}{2v_\pm}.
\end{equation}
With this notation in hand, we now turn to the construction of the Evans function.

\subsection{Construction of the Evans function} \label{Construction}

Consider the limiting ODE systems
\begin{equation}\label{limODE}
Z' = \mathbb{A}_\pm(\lambda)Z,
\end{equation}
where $\mathbb{A}_\pm$ are given in \eqref{A_pm}. For $\lambda$ in a sufficiently small neighborhood of $0$, we construct analytic bases of solutions of \eqref{limODE}, that is, for $\lambda\in B(0,r)$ with $r>0$ sufficiently small, where $B(0,r):=\{\lambda\in\mathbb C:\ |\lambda|<r\}$.

\begin{lemma}\label{low1}
Assume that the hypotheses of Theorem~\ref{Main} hold. There exists $r_0>0$ such that for all $\lambda \in B(0,r_0)$, the ODE system \eqref{limODE} with coefficient matrix $\mathbb{A}_\pm(\lambda)$ given by \eqref{A_pm} admits a basis of solutions of the form
\begin{equation*}
Z_j^\pm(x,\lambda) = e^{\mu_j^\pm(\lambda)x} \overline{V}_j^\pm(\lambda),\quad j=1,\dots,5,
\end{equation*}
where $\mu_j^\pm(\lambda)$ and $\overline{V}_j^\pm(\lambda)$ are analytic in $\lambda$. These solutions decompose into three ``fast'' modes and two ``slow'' modes:

\smallskip\noindent
\emph{(i) Fast modes.} For $j=1,2,3$,
\begin{equation}\label{fast}
\mu_j^\pm(\lambda)=\gamma_j^\pm+\mathcal O(\lambda),\qquad
\overline{V}_j^\pm(\lambda)=S_j^\pm+\mathcal O(\lambda),
\end{equation}
where $\gamma_j^\pm$ are the nonzero distinct eigenvalues of $\mathbb{A}_\pm(0)$ satisfying
\begin{equation} \label{C1}
\operatorname{Re}\gamma_1^+, \operatorname{Re}\gamma_2^+ <0 < \operatorname{Re}\gamma_3^+, \quad \operatorname{Re}\gamma_2^- <0<\operatorname{Re}\gamma_1^-, \operatorname{Re} \gamma_3^-,
\end{equation}
and $S_j^\pm$ are the corresponding right eigenvectors.

\smallskip\noindent
\emph{(ii) Slow modes.} For $j=4,5$,
\begin{equation}\label{slow}
\mu_j^\pm (\lambda) = \lambda/a_{j-3}^\pm - \lambda^2 \beta_{j-3}^\pm/(a_{j-3}^\pm)^3 + \mathcal{O}(\lambda^3), \quad \overline{V}_j^\pm (\lambda) =  \tilde{r}_{j-3}^\pm +\mathcal{O}(\lambda)
\end{equation}
with
\begin{equation} \label{revs}
\tilde{r}_{j-3}^\pm = \begin{pmatrix}
{r_{j-3}^\pm}^{\mathrm{tr}} & 0 & \displaystyle -\frac{1}{\sqrt{2}v_\pm} & 0
\end{pmatrix}^{\mathrm{tr}},
\end{equation}
where $a_{j-3}^\pm$, $\beta_{j-3}^\pm$ and $r_{j-3}^\pm$ are as defined in \eqref{aj}, \eqref{betaj}, and \eqref{lrj}, respectively.
\end{lemma}

\begin{proof}
The characteristic polynomials of $\mathbb{A}_\pm(0)$ are given by
\begin{equation*}
\det (\mathbb{A}_\pm(0)-\mu I) = -\mu^2 \left(\mu^3 +\frac{s^2 v_\pm^2 -T}{s\nu v_\pm}\mu^2 - \frac{v_\pm}{\varepsilon^2}\mu + \frac{T+1-s^2 v_\pm^2 }{s\nu \varepsilon^2} \right) =0.
\end{equation*}
Hence, each matrix $\mathbb{A}_\pm(0)$ has an eigenvalue $\mu=0$ of multiplicity $2$ and three eigenvalues $\gamma_j^\pm$, which are the roots of the cubic equation
\begin{equation*}
\gamma^3 +\frac{s^2 v_\pm^2 -T}{s\nu v_\pm}\gamma^2 - \frac{v_\pm}{\varepsilon^2}\gamma + \frac{T+1-s^2 v_\pm^2 }{s\nu \varepsilon^2}=0.
\end{equation*}
Let $S_j^\pm$ denote an eigenvector of $\mathbb{A}_\pm(0)$ corresponding to $\gamma_j^\pm$. Since the nonzero roots $\gamma_j^\pm$ are simple as shown below, the implicit function theorem applied to $\det(\mathbb{A}_\pm(\lambda)-\mu I)=0$ yields analytic eigenvalues $\mu_j^\pm(\lambda)$ and corresponding analytic eigenvectors $\overline{V}_j^\pm(\lambda)$ for $j=1,2,3$. The Taylor expansion about $\lambda=0$ then gives \eqref{fast}.

We next verify \eqref{C1} and the simplicity of $\gamma_j^\pm$. To this end, we set
\begin{equation*}
\tilde{p}_\pm (\gamma) := \gamma^3 +\frac{s^2 v_\pm^2 -T}{s\nu v_\pm}\gamma^2 - \frac{v_\pm}{\varepsilon^2}\gamma, \quad p_\pm :=\frac{T+1-s^2 v_\pm^2 }{s\nu \varepsilon^2} .
\end{equation*}
The zeros $\tilde{\gamma}^\pm_\ell$ of $\tilde{p}_\pm(\gamma)$ are given by
\begin{equation*}
\tilde{\gamma}_1^\pm =0, \quad \tilde{\gamma}_\ell^\pm = \frac{1}{2} \left(\frac{T-s^2 v_\pm^2}{s\nu v_\pm} + (-1)^{\ell-1} \sqrt{\left(\frac{s^2 v_\pm^2 -T}{s\nu v_\pm} \right)^2 +\frac{4 v_\pm}{\varepsilon^2} } \right), \quad \ell=2,3.
\end{equation*}
Since $\tilde{\gamma}_2^\pm \neq \tilde{\gamma}_3^\pm$ and $\tilde{\gamma}_2^\pm \tilde{\gamma}_3^\pm = -v_\pm/\varepsilon^2 \neq 0$, the three zeros $\tilde{\gamma}_\ell^\pm$ ($\ell=1,2,3$) are simple.
Moreover, by \eqref{RH1}, we have
\begin{equation*}
T-s^2v_\pm^2 =T -\frac{(T+1)v_\pm}{v_\mp} <0
\end{equation*}
for $v_+>v_->0$ with $\lvert v_+-v_- \rvert$ sufficiently small, which implies
\begin{equation} \label{rt}
\tilde{\gamma}_2^\pm <0, \quad \tilde{\gamma}_3^\pm >0.
\end{equation}
On the other hand,
\begin{equation*}
\begin{array}{l l}
p_+ =  \dfrac{T+1- s^2v_+^2}{s \nu \varepsilon^2} = \dfrac{(T+1)(1-v_+/v_-)}{s \nu \varepsilon^2} < 0, \\
p_- =  \dfrac{T+1 - s^2v_-^2}{s \nu \varepsilon^2} = \dfrac{(T+1)(1-v_-/v_+)}{s \nu \varepsilon^2} > 0. 
\end{array} 
\end{equation*}
Now we fix $v_-$ and consider $v_+$ as a parameter. Since $p_-= \mathcal{O}(\lvert v_+-v_- \rvert)$ and the roots $\tilde{\gamma}_\ell^-$ of $\tilde{p}_-(\gamma)$ are simple, the equation $\tilde{p}_-(\gamma)+p_-=0$ has three simple roots $\gamma_\ell^-$ for all $v_+$ sufficiently close to $v_-$ by continuity of the roots with respect to $p_-$. Thus, two roots, say $\gamma_2^-$ and $\gamma_3^-$, satisfy \eqref{C1} by continuity and \eqref{rt}. The remaining root, say $\gamma_1^-$, lies near $\tilde{\gamma}_1^-=0$ and satisfies $\operatorname{Re}{\gamma_1^-}>0$ since
\begin{equation*}
0< p_- = - \tilde{p}_-(\gamma) = \frac{v_-}{\varepsilon^2}\gamma \left(1 + \mathcal{O}(\gamma)\right) \quad \text{ as } \gamma \rightarrow 0.
\end{equation*}
Similarly, we fix $v_+$ and consider $v_-$ as a parameter. Then the roots $\gamma_1^+,\gamma_2^+,\gamma_3^+$ of $\tilde{p}_+(\gamma)+p_+=0$ satisfy \eqref{C1}. Indeed, $\gamma_2^+$ and $\gamma_3^+$ are close to $\tilde{\gamma}_2^+<0$ and $\tilde{\gamma}_3^+>0$, respectively, and the remaining root $\gamma_1^+$ lies near $0$ and satisfies $\operatorname{Re}{\gamma_1^+}<0$ since
\begin{equation*}
0> p_+ = - \tilde{p}_+(\gamma) = \frac{v_+}{\varepsilon^2}\gamma \left(1 + \mathcal{O}(\gamma)\right) \quad \text{ as } \gamma \rightarrow 0.
\end{equation*}
Therefore, the roots $\gamma_j^\pm$ are simple and distinct, and the fast-mode expansions \eqref{fast} follow.

To obtain the expansion \eqref{slow} for $\mu_4^\pm$ and $\mu_5^\pm$, we consider the characteristic equation $\det (\mathbb{A}_\pm(\lambda)-\mu I)=0$. For $|\mu|$ sufficiently small so that $\textstyle \mu^2 \neq \frac{v_\pm}{\varepsilon^2}$, a direct computation yields
\begin{equation} \label{mu45eq}
\left(\lambda - s\mu - \frac{\nu}{2v_\pm}\mu^2 \right)^2 = \frac{T+1}{v_\pm^2}\mu^2 + \frac{\nu^2}{4v_\pm^2}\mu^4 + \frac{\varepsilon^2 \mu^4}{v_\pm^3-\varepsilon^2v_\pm^2\mu^2},
\end{equation}
and the right-hand side equals $\mu^2 g_\pm(\mu)$ with $g_\pm$ analytic and $\textstyle g_\pm(0)=\frac{T+1}{v_\pm^2}\neq 0$. Thus there exists an analytic function $q_\pm(\mu)$ with $q_\pm(\mu)^2=g_\pm(\mu)$ for $|\mu|$ small, and \eqref{mu45eq} has two analytic solutions $\lambda=\lambda_j^\pm(\mu)$, $j=1,2$, corresponding to the choices $\textstyle \lambda - s\mu - \frac{\nu}{2v_\pm}\mu^2 = \pm \mu q_\pm(\mu)$. With this observation, expanding the right-hand side of \eqref{mu45eq} about $\mu=0$ and using the ansatz $\lambda= a\mu + \beta \mu^2 + \mathcal{O}(\mu^3)$, we obtain
\begin{equation*}
\lambda_j^\pm (\mu) = \left(s +(-1)^j \sqrt{\frac{T+1}{v_\pm^2}} \right) \mu + \frac{\nu}{2v_\pm}\mu^2 + \mathcal{O}(\mu^3) = a_j^\pm \mu + \beta_j^\pm \mu^2 + \mathcal{O}(\mu^3).
\end{equation*}
Since $a_j^\pm \neq 0$, we may apply the implicit function theorem to invert this relation and obtain analytic functions $\mu_{j+3}^\pm(\lambda)$, $j=1,2$. Reindexing $j+3$ as $j$, we have
\begin{equation*}
\mu_j^\pm (\lambda) = \lambda/a_{j-3}^\pm - \lambda^2 \beta_{j-3}^\pm/(a_{j-3}^\pm)^3 + \mathcal{O}(\lambda^3), \quad j=4,5,
\end{equation*}
which gives the expansion of $\mu_j^\pm(\lambda)$ in \eqref{slow}. Moreover, since $a_1^\pm \neq a_2^\pm$, the two slow eigenvalues $\mu_4^\pm(\lambda)$ and $\mu_5^\pm(\lambda)$ are distinct for $|\lambda|>0$ small.

Finally, using \eqref{A_pm} and \eqref{mu45eq}, we may choose the corresponding eigenvectors as
\begin{equation*}
\overline{V}_j^\pm(\lambda) = \frac{1}{\sqrt{2}} \begin{pmatrix} 
1 \\
\frac{\lambda}{\mu_j^\pm(\lambda)}-s \\ 
\frac{\nu}{v_\pm}(\lambda-s\mu_j^\pm(\lambda)) \\ 
-\frac{1}{v_\pm-\varepsilon^2(\mu_j^\pm(\lambda))^2} \\ 
-\frac{\mu_j^\pm(\lambda)}{v_\pm-\varepsilon^2(\mu_j^\pm(\lambda))^2}
\end{pmatrix}, \quad j=4,5.
\end{equation*}
By the expansion of $\mu_j^\pm(\lambda)$ in \eqref{slow}, the quotient $\lambda/\mu_j^\pm(\lambda)$ extends analytically to $\lambda=0$ with value $a_{j-3}^\pm$. Hence, these eigenvectors are analytic near $\lambda=0$ and satisfy $\overline{V}_j^\pm(0)=\tilde r_{j-3}^\pm$ with $\tilde r_{j-3}^\pm$ given by \eqref{revs}. This gives the expansion of $\overline{V}_j^\pm(\lambda)$ in \eqref{slow}.
\end{proof}

For the standard definition of the Evans function, we require analytic bases of solutions of the system \eqref{5} that decay as $x\to +\infty$ and as $x\to -\infty$, respectively, within the \emph{domain of consistent splitting} (as defined in Appendix~\ref{Appendix_A}). The next two lemmas provide these ingredients near $\lambda=0$.

Combining Lemma~\ref{Gap} with Lemma~\ref{low1}, we immediately obtain:

\begin{lemma} \label{lemma0.2}
Assume that the hypotheses of Theorem~\ref{Main} hold. Let $r_0>0$ be as in Lemma~\ref{low1}.  There exists a constant $r_1 \in (0,r_0)$ such that for $\lambda \in B(0,r_1)$, the ODE system \eqref{5} admits solutions $W_j^\pm (x;\lambda)$, $C^1$ in $x$ and analytic in $\lambda$, satisfying
\begin{equation} \label{0.8}
\begin{split}
W_j^+ (x;\lambda) & = e^{\mu_j^+ (\lambda) x} V_j^+ (x;\lambda), \quad x \geq 0, \\
W_j^- (x;\lambda) & = e^{\mu_j^- (\lambda) x} V_j^- (x;\lambda), \quad x \leq 0, \\
(\partial / \partial \lambda)^k V_j^\pm (x ;\lambda) & = (\partial / \partial \lambda)^k \overline{V}_j^\pm (\lambda) +\mathcal{O} \left( \lvert \overline{V}_j^\pm (\lambda) \rvert e^{-\tilde{\theta}\lvert x \rvert} \right), \quad x \gtrless 0,
\end{split}
\end{equation}
for any $k \geq 0$ and some $\tilde{\theta}>0$, where $\mu_j^\pm(\lambda)$ and $\overline{V}_j^\pm(\lambda)$ are as in Lemma~\ref{low1}, and the $\mathcal{O}(\cdot)$ term depends only on $k$ and $\tilde{\theta}$.
\end{lemma}

We observe that the eigenvalues $\mu_j^\pm$ in Lemmas~\ref{low1}--\ref{lemma0.2} satisfy
\begin{equation} \label{mu_js}
\begin{split}
& \operatorname{Re} \mu_1^+(\lambda), \operatorname{Re}\mu_2^+(\lambda) < 0 < \operatorname{Re}\mu_3^+(\lambda), \operatorname{Re}\mu_4^+(\lambda), \operatorname{Re}\mu_5^+(\lambda), \\
& \operatorname{Re}\mu_2^-(\lambda), \operatorname{Re}\mu_4^-(\lambda) < 0 < \operatorname{Re}\mu_1^-(\lambda), \operatorname{Re}\mu_3^-(\lambda), \operatorname{Re}\mu_5^-(\lambda)
\end{split}
\end{equation}
for all $\lambda \in B(0,r) \cap \{\lambda \in \mathbb{C} : \operatorname{Re} \lambda >0 \}$ with $r >0$ sufficiently small. This is a consequence of Lemma~\ref{low1} together with $a_2^\pm >0$ and $a_1^- < 0 < a_1^+$. We therefore have the following lemma.

\begin{lemma}\label{DCS}
Assume that the hypotheses of Theorem~\ref{Main} hold. Let $r_1>0$ be as in Lemma~\ref{lemma0.2}. There exists a constant $r_2 \in (0,r_1)$ such that the set $B_+(0,r_2) := B(0,r_2) \cap \{\lambda \in \mathbb{C}: \operatorname{Re} \lambda > 0 \}$ is a simply connected subset of the domain of consistent splitting for \eqref{5}. In particular, for every $\lambda\in B_+(0,r_2)$, the stable and unstable subspaces $\mathcal{S}_\pm(\lambda)$ and $\mathcal{U}_\pm(\lambda)$ of $\mathbb{A}_\pm(\lambda)$ satisfy $\dim \mathcal{S}_\pm(\lambda)=2$ and $\dim \mathcal{U}_\pm(\lambda)=3$, respectively.
\end{lemma}

\begin{proof}
Choose $r_2\in(0,r_1)$ so that \eqref{mu_js} holds on $B_+(0,r_2)$. Then $\mathbb{A}_\pm(\lambda)$ admits stable and unstable subspaces $\mathcal{S}_\pm(\lambda)$ and $\mathcal{U}_\pm(\lambda)$ with $\dim \mathcal{S}_\pm(\lambda)=2$ and $\dim \mathcal{U}_\pm(\lambda)=3$; in particular, there is no center subspace. Finally, $B_+(0,r_2)$ is convex and thus simply connected. Consequently, $B_+(0,r_2)$ is contained in the domain of consistent splitting for \eqref{5}.
\end{proof}

We now construct the Evans function for \eqref{5} in a small neighborhood of $\lambda=0$
following the framework in \cite[Section~3.1]{MZ2}. Fix $r\in(0,r_2)$, where $r_2$ is as in Lemma~\ref{DCS}, and consider $\lambda \in B_+(0,r)$. Then, in view of \eqref{mu_js}, the solutions $W_1^+(x;\lambda)$ and $W_2^+(x;\lambda)$ given in Lemma~\ref{lemma0.2} decay as $x\to+\infty$, while $W_1^-(x;\lambda)$, $W_3^-(x;\lambda)$, and $W_5^-(x;\lambda)$ decay as $x\to-\infty$. We set
\begin{equation} \label{sb1}
\varphi_1^+(x;\lambda):=W_1^+(x;\lambda),\quad
\varphi_2^+(x;\lambda):=W_2^+(x;\lambda),
\end{equation}
and
\begin{equation} \label{sb2}
\varphi_3^-(x;\lambda):=W_1^-(x;\lambda),\quad
\varphi_4^-(x;\lambda):=W_3^-(x;\lambda),\quad
\varphi_5^-(x;\lambda):=W_5^-(x;\lambda).
\end{equation}
Using these decaying solutions, we define the Evans function by
\begin{equation}\label{Evans}
\mathcal D(\lambda)
:=\det\big(\varphi_1^+,\varphi_2^+,\varphi_3^-,\varphi_4^-,\varphi_5^-\big)\big|_{x=0}
\end{equation}
for $\lambda \in B_+(0,r)$. By Lemma~\ref{lemma0.2}, the functions $\varphi_j^\pm(x;\lambda)$ extend analytically to $B(0,r)$. Hence the determinant \eqref{Evans} yields an analytic extension of $\mathcal D(\lambda)$ to $B(0,r)$. We use the same notation $\mathcal D$ for this extension.

\begin{remark} \label{agree}
On any simply connected subset of the domain of consistent splitting, zeros of the Evans function correspond to isolated eigenvalues of $L$, counted with algebraic multiplicity; see \cite{GJ1,GJ2}. In the \emph{extended spectral theory}, the same correspondence holds between zeros of the analytically extended Evans function and \emph{effective} eigenvalues; see \cite[Proposition~6.15]{MZ2} and \cite[Theorem~6.3]{ZH}. For the embedded eigenvalue $\lambda=0$ in our problem, we use this notion of multiplicity, given by the order of vanishing of $\mathcal D$ at the origin.
\end{remark}

\begin{remark}
The choice of bases in \eqref{Evans} is not unique. One may replace $\{\varphi_1^+,\varphi_2^+\}$ and $\{\varphi_3^-,\varphi_4^-,\varphi_5^-\}$ by any other analytic bases spanning the same decaying manifolds as $x\to+\infty$ and as $x\to-\infty$. Under any such change of basis that is analytic and invertible near $\lambda=0$, $\mathcal D(\lambda)$ is multiplied by a nonvanishing analytic factor, and hence its zeros (with multiplicity) are unchanged; cf. \cite[Remark~3.5]{MZ2}.
\end{remark}

We complete the construction of the Evans function by fixing a normalization at $\lambda=0$. By the translation invariance of the shock, the linearized operator $L$ has a zero eigenvalue with the corresponding eigenfunction $\bar{U}'=(\bar{v}',\bar{u}')^{\mathrm{tr}}$. Equivalently, a direct computation shows that
\begin{equation*}
W_0:= \big(\bar{v}',\bar{u}',(b_1\bar{v}' + b_2\bar{u}')',\bar{\phi}',\bar{\phi}'' \big)^{\mathrm{tr}}
\end{equation*}
solves the first-order ODE system \eqref{5} at $\lambda=0$, i.e.,
\begin{equation*}
W_0' = \mathbb{A}(x,0) W_0.
\end{equation*}
Moreover, $W_0(x)$ decays exponentially as $x\to\pm\infty$ by Lemma~\ref{SP}. At $\lambda=0$, the modes $\varphi_1^+$, $\varphi_2^+$, $\varphi_3^-$, and $\varphi_4^-$ are fast, whereas $\varphi_5^-$ is a slow (center) mode with $\mu_5^-(0)=0$. Therefore, any solution that decays exponentially as $x\to-\infty$ lies in the span of the fast modes $\varphi_3^-$ and $\varphi_4^-$, and similarly any solution that decays exponentially as $x\to+\infty$ lies in the span of $\varphi_1^+$ and $\varphi_2^+$. Combining these observations, we obtain
\begin{equation*}
W_0\in\operatorname{span}\{\varphi_1^+(\cdot;0),\varphi_2^+(\cdot;0)\}\cap\operatorname{span}\{\varphi_3^-(\cdot;0),\varphi_4^-(\cdot;0)\}.
\end{equation*}
Since $W_0$ is nonzero, we may choose it as a basis element in each of these two-dimensional spaces and relabel the resulting bases so that
\begin{equation}\label{relabel}
\varphi_1^+(x;0)\equiv \varphi_3^-(x;0)\equiv W_0(x), \quad x\in\mathbb{R}.
\end{equation}
The modified pairs $\{\varphi_1^+,\varphi_2^+\}$ and $\{\varphi_3^-,\varphi_4^-\}$ then still span the corresponding fast decaying manifolds, while $\varphi_5^-$ is unchanged. These changes of basis can be taken independent of $\lambda$ and hence preserve analyticity.

\subsection{Computation of the Evans function}

We now complete the proof of Theorem~\ref{Main} by establishing simplicity of the zero eigenvalue through a computation of $\mathcal D'(0)$. We begin with a key proposition.

\begin{proposition} \label{D'rel}
Assume that the hypotheses of Theorem~\ref{Main} hold. Let $\mathcal D(\lambda)$ be the Evans function for \eqref{5} defined by \eqref{Evans}--\eqref{relabel}. Then, we have
\begin{equation}\label{gDel}
\mathcal D'(0) = \Gamma \Delta,
\end{equation}
where $\Gamma$ is a transversality coefficient and $\Delta$ is defined by
\begin{equation} \label{LMdet}
\Delta := \det \left( U_+ - U_-, r_2^- \right).
\end{equation}
Moreover, $\Gamma\neq0$ if and only if any solution of \eqref{5} at $\lambda=0$ that decays exponentially as $x\to\pm\infty$ belongs to $\operatorname{span}\{W_0\}$.
\end{proposition}

\begin{proof}
In the following, we denote $\varphi_j^\pm(x;0)$ by $\varphi_j^\pm(x)$ for simplicity. With the choice of bases in Section~\ref{Construction}, we have
\begin{equation} \label{D0cal}
\begin{split}
\mathcal{D}'(0) & = \det \left( \varphi_{1\lambda}^+  (x), \varphi_2^+(x), W_0(x), \varphi_4^-(x), \varphi_5^- (x) \right) \big|_{x=0} \\
& \quad + \det \left( W_0(x), \varphi_2^+(x),  \varphi_{3\lambda}^-(x), \varphi_4^-(x), \varphi_5^-(x) \right) \big|_{x=0} \\
& = \det \left( (z^+ - z^-)(x), \varphi_2^+(x), W_0(x), \varphi_4^-(x), \varphi_5^- (x) \right) \big|_{x=0},
\end{split}
\end{equation}
where $z^+(x) := \varphi_{1\lambda}^+(x)$ and $z^-(x) := \varphi_{3\lambda}^-(x)$. 

We first derive integrated identities for $\varphi_j^\pm(x)$ and $z^\pm(x)$ using their governing equations and far-field limits. By Lemma~\ref{lemma0.2}, the modes $\varphi_j^\pm(x)$ satisfy
\begin{equation} \label{tembd1}
\lim_{x \rightarrow +\infty} \varphi_1^+ (x) = \lim_{x \rightarrow +\infty}\varphi_2^+ (x) = \lim_{x \rightarrow -\infty}\varphi_3^- (x) = \lim_{x \rightarrow -\infty}\varphi_4^- (x) =0,
\end{equation}
and
\begin{equation} \label{tembd2}
\lim_{x \rightarrow -\infty} \varphi_5^- (x) = \begin{pmatrix}
{r_2^{-}}^{\mathrm{tr}} & 0 & \displaystyle -\frac{1}{\sqrt{2}v_-} & 0
\end{pmatrix}^{\mathrm{tr}}.
\end{equation}
Recall that any solution of \eqref{5} satisfies
\begin{equation} \label{temODE1}
(AU)_x + (BU_x)_x + \left( D \begin{pmatrix} \phi_x \\ 0 \end{pmatrix} \right)_x + \left(E\begin{pmatrix} \phi_{xx} \\ 0 \end{pmatrix} \right)_x - \lambda U = 0,
\end{equation}
where $U=(v,u)^{\mathrm{tr}}$, the matrices $A,B,D,E$ are as in \eqref{ABDE}, and $\phi$ satisfies the linearized Poisson equation \eqref{LNSP3}. In particular, the modes $\varphi_j^\pm(x;\lambda)$ satisfy these equations with $U=(\varphi_{j,1}^\pm,\varphi_{j,2}^\pm)^{\mathrm{tr}}$ and $\phi=\varphi_{j,4}^\pm$.

At $\lambda=0$, substituting $(\varphi_{2,1}^+(x), \varphi_{2,2}^+(x))^{\mathrm{tr}}$ into \eqref{temODE1} and integrating the resulting equation from $x$ to $+\infty$, we have
\begin{equation} \label{temf}
B \begin{pmatrix}
\varphi_{2,1}^+ \\ \varphi_{2,2}^+
\end{pmatrix}' + A \begin{pmatrix}
\varphi_{2,1}^+ \\ \varphi_{2,2}^+
\end{pmatrix} = - \left[ \begin{pmatrix}
0 \\ D_{21} \varphi_{2,5}^+
\end{pmatrix} + \begin{pmatrix}
0 \\ E_{21} \left( \varphi_{2,5}^+ \right)'
\end{pmatrix} \right],
\end{equation}
where $'$ denotes $\textstyle \frac{d}{dx}$. Similarly, for $\varphi_4^-(x)$, we integrate \eqref{temODE1} from $x$ to $-\infty$ to obtain
\begin{equation} \label{temf1}
B \begin{pmatrix}
\varphi_{4,1}^- \\ \varphi_{4,2}^-
\end{pmatrix}' + A \begin{pmatrix}
\varphi_{4,1}^- \\ \varphi_{4,2}^-
\end{pmatrix} = - \left[ \begin{pmatrix}
0 \\ D_{21} \varphi_{4,5}^-
\end{pmatrix} + \begin{pmatrix}
0 \\ E_{21} \left( \varphi_{4,5}^- \right)'
\end{pmatrix} \right].
\end{equation}
By the far-field condition \eqref{tembd2} and $A_- r_2^- = a_2^- r_2^-$, the slow mode $\varphi_5^-(x)$ satisfies
\begin{equation} \label{temf2}
B \begin{pmatrix}
\varphi_{5,1}^- \\ \varphi_{5,2}^-
\end{pmatrix}' + A \begin{pmatrix}
\varphi_{5,1}^- \\ \varphi_{5,2}^-
\end{pmatrix} = - \left[ \begin{pmatrix}
0 \\ D_{21} \varphi_{5,5}^-
\end{pmatrix} + \begin{pmatrix}
0 \\ E_{21} \left( \varphi_{5,5}^- \right)'
\end{pmatrix} - a_2^- r_2^- \right].
\end{equation}
Differentiating \eqref{temODE1} with respect to $\lambda$ and evaluating at $\lambda=0$ for $\varphi_1^+$ and $\varphi_3^-$, we have the equation for $z^\pm(x)$:
\begin{equation*}
\left[ A \begin{pmatrix}
z_1^\pm \\ z_2^\pm
\end{pmatrix} + B \begin{pmatrix}
z_1^\pm \\ z_2^\pm
\end{pmatrix}' + \begin{pmatrix}
0 \\ D_{21} z_5^\pm
\end{pmatrix} +  \begin{pmatrix}
0 \\ E_{21} \left(z_5^\pm \right)'
\end{pmatrix} \right]' - \begin{pmatrix}
\bar{v}' \\ \bar{u}'
\end{pmatrix} =0.
\end{equation*}
We then integrate this from $x$ to $\pm\infty$ to obtain
\begin{equation} \label{temf3}
B \begin{pmatrix}
z_1^\pm \\ z_2^\pm
\end{pmatrix}' +  A \begin{pmatrix}
z_1^\pm \\ z_2^\pm
\end{pmatrix} = - \left[ \begin{pmatrix}
0 \\ D_{21} z_5^\pm
\end{pmatrix} +  \begin{pmatrix}
0 \\ E_{21} \left(z_5^\pm \right)'
\end{pmatrix} - \begin{pmatrix}
\bar{v} - v_\pm \\ \bar{u} - u_\pm
\end{pmatrix} \right].
\end{equation}

Next we introduce a change of coordinates $R$ and the associated matrix $M$:
\begin{equation*}
R := \begin{pmatrix}
b_1 & b_2 & 0 & 0 & 0 \\
A_{11} & A_{12} & 0 &0 & 0 \\
A_{21} - b_1' & A_{22} - b_2' & 1 & 0 & 0 \\
0 & 0 & 0 & 1 & 0 \\
0 & 0 & 0 & 0 & 1
\end{pmatrix}, \quad M := \begin{pmatrix} b_1 & b_2 \\ A_{11} & A_{12}\end{pmatrix},
\end{equation*}
so that for $W=(v,u,\tilde u',\phi,\phi')^{\mathrm{tr}}$ we may write
\begin{equation*}
RW = \begin{pmatrix}
\tilde{u} \\
BU' + AU \\
\phi \\
\phi'
\end{pmatrix} =: \zeta.
\end{equation*}
Note that $R$ is non-singular since
\begin{equation*}
\det R = \det M = \frac{\varepsilon^2 \bar{\phi}'}{\bar{v}^3} - \frac{s \nu}{\bar{v}} < 0 \quad \text{for all }  x \in \mathbb{R},
\end{equation*}
where we used $\bar{\phi}' <0$ in Lemma~\ref{SP}.

Multiplying the columns of \eqref{D0cal} by $R(0)$ gives
\begin{equation}\label{D0cal_zeta}
\mathcal D'(0)
= \frac{1}{\det R(0)}\,
\det \left( \zeta_z,\ \zeta_2^+,\ \zeta_0,\ \zeta_4^-,\ \zeta_5^- \right) \big|_{x=0},
\end{equation}
where
\begin{equation*}
\zeta_z := R(z^+-z^-),\quad
\zeta_2^+ := R\varphi_2^+,\quad
\zeta_0 := RW_0,\quad
\zeta_4^- := R\varphi_4^-,\quad
\zeta_5^- := R\varphi_5^- .
\end{equation*}
For a fast mode $\varphi$ (e.g., $\varphi=\varphi_2^+$ or $\varphi=\varphi_4^-$), we have
\begin{equation*}
R \varphi = \begin{pmatrix}
b_1 \varphi_1 + b_2 \varphi_2 \\
0 \\
- D_{21} \varphi_5 - E_{21} (\varphi_5)'\\
\varphi_4 \\
\varphi_5
\end{pmatrix}.
\end{equation*}
We claim that the third component of $R\varphi$ can be written as a linear combination of the remaining components. Recall that the last relation in \eqref{coordinate} reads
\begin{equation} \label{LPrel}
\varphi_5' = \left( \frac{\bar{\phi}''}{\bar{v}} - \frac{2\bar{v}'\bar{\phi}'}{\bar{v}^2} + \frac{\bar{v}e^{\bar{\phi}}}{\varepsilon^2} \right) \varphi_1 + \frac{\bar{\phi}'}{\bar{v}} \varphi_1' + \frac{\bar{v}^2 e^{\bar{\phi}}}{\varepsilon^2} \varphi_4 + \frac{\bar{v}'}{\bar{v}} \varphi_5.
\end{equation}
On the other hand, \eqref{temf} together with \eqref{co1} evaluated at $\lambda=0$ yields
\begin{equation} \label{feq}
(b_1 - sb_2)\varphi_1' =  - A_{21} \varphi_1 - A_{22}\varphi_2 - D_{21} \varphi_5 - E_{21} \varphi_5'.
\end{equation}
We then combine \eqref{LPrel}--\eqref{feq} to obtain
\begin{equation} \label{phi5rel}
\begin{split}
\varphi_5' & = \left( 1 + \frac{E_{21} \bar{\phi}'}{\bar{v}(b_1 - sb_2)} \right)^{-1} \bigg[ \left( \frac{\bar{\phi}''}{\bar{v}} - \frac{2\bar{v}'\bar{\phi}'}{\bar{v}^2} + \frac{\bar{v}e^{\bar{\phi}}}{\varepsilon^2} - \frac{A_{21} \bar{\phi}'}{\bar{v}(b_1 - sb_2)} \right) \varphi_1 - \frac{A_{22} \bar{\phi}'}{\bar{v}(b_1 - sb_2)} \varphi_2\\
& \qquad  + \frac{\bar{v}^2 e^{\bar{\phi}}}{\varepsilon^2} \varphi_4 + \left( \frac{\bar{v}'}{\bar{v}} - \frac{D_{21}\bar{\phi}'}{\bar{v}(b_1 - sb_2)} \right) \varphi_5\bigg].
\end{split}
\end{equation}
Here, using $\textstyle b_1=\frac{\varepsilon^2 \bar{\phi}'}{\bar{v}^3}, b_2 = \frac{\nu}{\bar{v}}, E_{21}= -\frac{\varepsilon^2}{\bar{v}^2}$, and $\bar{\phi}' <0$, we have
\begin{equation*}
\kappa(x) := 1 + \frac{E_{21} \bar{\phi}'}{\bar{v}(b_1 - sb_2)} = 1 + \frac{-\varepsilon^2 \bar{\phi}' / \bar{v}^2}{\varepsilon^2 \bar{\phi}'/\bar{v}^2 - s\nu} = \frac{ - s\nu}{\varepsilon^2 \bar{\phi}'/\bar{v}^2 - s\nu} > 0.
\end{equation*}
Since $(\tilde{u},(AU)_1)^{\mathrm{tr}} = M (v,u)^{\mathrm{tr}}$, it follows from \eqref{temf}--\eqref{temf1} that
\begin{equation} \label{LPrel1}
\begin{pmatrix}
\varphi_1 \\ \varphi_2
\end{pmatrix}  = M^{-1} \begin{pmatrix}
b_1 \varphi_1 + b_2 \varphi_2 \\
0
\end{pmatrix} = \frac{1}{\det M} \begin{pmatrix}
A_{12} (b_1 \varphi_1 + b_2 \varphi_2) \\
- A_{11} (b_1 \varphi_1 + b_2 \varphi_2)
\end{pmatrix}.
\end{equation}
Therefore, $\varphi_5'$ can be written as a linear combination of $(b_1\varphi_1 +b_2 \varphi_2), \varphi_4, \varphi_5$, and hence the third component of $R\varphi$ is a linear combination of the remaining components. This proves the claim.

From \eqref{temf2} and \eqref{temf3}, we find
\begin{equation*}
\zeta_5^- = R \varphi_5^- = \begin{pmatrix}
b_1 \varphi_{5,1}^- + b_2 \varphi_{5,2}^- \\
a_2^- r_{2,1}^- \\
- D_{21} \varphi_{5,5}^- - E_{21} (\varphi_{5,5}^-)' + a_2^- r_{2,2}^- \\
\varphi_{5,4}^- \\
\varphi_{5,5}^-
\end{pmatrix}
\end{equation*}
and
\begin{equation*}
\zeta_z = R (z^+ - z^-) = \begin{pmatrix}
b_1(z_1^+ - z_1^-) + b_2 (z_2^+ - z_2^-) \\
- (v_+ - v_-) \\
-D_{21} (z_5^+ - z_5^-) - E_{21} (z_5^+ - z_5^-)' - (u_+ - u_-) \\
(z_4^+ - z_4^-) \\
(z_5^+ - z_5^-)
\end{pmatrix},
\end{equation*}
respectively. To eliminate the terms involving $D_{21}$ and $E_{21}$ in the third components of $\zeta_5^-$ and $\zeta_z$ by row operations, we first verify that $z^+-z^-$ also satisfies \eqref{5} at $\lambda=0$. Differentiating \eqref{5} with respect to $\lambda$ and evaluating at $\lambda=0$, we have
\begin{equation*}
(z^\pm)'=\mathbb A(x,0)z^\pm+\partial_\lambda\mathbb A(x,0)W_0,
\end{equation*}
where we used $\varphi_1^+(x;0)=\varphi_3^-(x;0)=W_0(x)$. Subtracting the two equations, we conclude that $z^+-z^-$ satisfies \eqref{5} at $\lambda=0$.

For all five solutions in \eqref{D0cal}, \eqref{co1} gives $u'=-sv'$. Hence, by the definition of $\zeta=RW$, we have
\begin{equation*}
(b_1-sb_2)v'=\zeta_3-A_{21}v-A_{22}u.
\end{equation*}
Substituting this identity into \eqref{co4}, we find that the coefficient of $\zeta_3$ in $\phi''$ is $\bar\phi'/[\bar v(b_1-sb_2)]$. Since $(v,u)^{\mathrm{tr}}=M^{-1}(\zeta_1,\zeta_2)^{\mathrm{tr}}$, it follows that $\zeta_3+D_{21}\phi'+E_{21}\phi''$ can be written as $\kappa(x)\zeta_3$ plus a linear combination of $\zeta_1,\zeta_2,\zeta_4,\zeta_5$, with the same coefficients for all five columns. Thus, by multiplying the third row of \eqref{D0cal_zeta} by $\kappa(0)$ and adding suitable multiples of the other rows, we may replace its entries by $\zeta_3+D_{21}\phi'+E_{21}\phi''$. This multiplies the determinant by $\kappa(0)$. Using \eqref{temf}--\eqref{temf3} and permuting the rows and columns, we arrive at
\begin{equation} \label{55det}
\mathcal{D}'(0) = -\frac{1}{\kappa(0)\det R(0)} \det \begin{pmatrix}
v_- - v_+ & a_2^- r_{2,1}^- & 0 & 0 & 0 \\
u_- - u_+ & a_2^- r_{2,2}^- & 0 & 0 & 0 \\
\zeta_{z,1} & \zeta_{5,1}^- & \zeta_{2,1}^+ & \zeta_{0,1} & \zeta_{4,1}^- \\
\zeta_{z,4} & \zeta_{5,4}^- & \zeta_{2,4}^+ & \zeta_{0,4} & \zeta_{4,4}^- \\
\zeta_{z,5} & \zeta_{5,5}^- & \zeta_{2,5}^+ & \zeta_{0,5} & \zeta_{4,5}^- \\
\end{pmatrix}\Bigg|_{x=0}.
\end{equation}
This can be factored as
\begin{equation}\label{Dprime_factored}
\mathcal D'(0) = \Gamma \det\big( U_+ - U_-,\ r_2^- \big),
\end{equation}
where the factor $\Gamma$ is given by
\begin{equation}\label{gamma_def}
\Gamma := \frac{a_2^-}{\kappa(0)\det R(0)} \det \begin{pmatrix}
\zeta_{2,1}^+ & \zeta_{0,1} & \zeta_{4,1}^-\\
\zeta_{2,4}^+ & \zeta_{0,4} & \zeta_{4,4}^-\\
\zeta_{2,5}^+ & \zeta_{0,5} & \zeta_{4,5}^-
\end{pmatrix}\Bigg|_{x=0}.
\end{equation}

Finally, we prove the last assertion. Since $a_2^-/[\kappa(0)\det R(0)]\neq0$, it follows from \eqref{gamma_def} that $\Gamma\neq0$ if and only if the last three columns in the determinant in \eqref{55det} are linearly independent. The coordinate transformation and row operations carried out above preserve linear independence. Hence, by uniqueness of solutions to \eqref{5}, this condition is equivalent to the linear independence of $\varphi_2^+$, $W_0$, and $\varphi_4^-$.

By the construction in Section~\ref{Construction} and \eqref{relabel}, the pairs $\{W_0,\varphi_2^+\}$ and $\{W_0,\varphi_4^-\}$ form bases for the spaces of solutions of \eqref{5} at $\lambda=0$ decaying exponentially as $x\to+\infty$ and as $x\to-\infty$, respectively. Since each pair is linearly independent, the three solutions $\varphi_2^+$, $W_0$, and $\varphi_4^-$ are linearly independent if and only if
\begin{equation*}
\operatorname{span}\{W_0,\varphi_2^+\}\cap\operatorname{span}\{W_0,\varphi_4^-\}=\operatorname{span}\{W_0\}.
\end{equation*}
This means that any solution of \eqref{5} at $\lambda=0$ that decays exponentially as $x\to\pm\infty$ is a multiple of $W_0$, which proves the last assertion.
\end{proof}

In the following, we conclude $\mathcal{D}'(0) \neq 0$ by verifying $\Gamma \neq 0$ and $\Delta \neq 0$. We first establish that $\Gamma\neq 0$.

\begin{lemma}[Transversality] \label{trans}
Assume that the hypotheses of Theorem~\ref{Main} hold. The transversality coefficient $\Gamma$ in Proposition~\ref{D'rel} satisfies $\Gamma\neq 0$.
\end{lemma}

\begin{proof}
The proof requires a separate analysis of ODEs and is therefore deferred to Appendix~\ref{App_trans}.
\end{proof}

As a consequence of Proposition~\ref{D'rel} and Lemma~\ref{trans}, we obtain the following corollary.

\begin{corollary}[Simplicity of $\lambda=0$] \label{cor:zero_simple}
Assume that the hypotheses of Theorem~\ref{Main} hold. The linear operator $L$ defined in \eqref{Ldef}--\eqref{calAB} has a simple zero eigenvalue, i.e., $\mathcal{D}'(0)\neq0$.
\end{corollary}

\begin{proof}
It is straightforward to check that $L\bar{U}'=0$, and hence $L$ has a zero eigenvalue. Since the normalization \eqref{relabel} implies $\mathcal D(0)=0$, it remains to show that $\mathcal D'(0)\neq0$. We compute
\begin{equation} \label{DeltaCom}
\begin{split}
\Delta = \det \left( U_+-U_-, r_2^- \right) & = \det \begin{pmatrix}
v_+ - v_- & \textstyle \frac{1}{\sqrt{2}} \\
u_+ - u_- & \textstyle \frac{1}{\sqrt{2}} \sqrt{\frac{T+1}{v_-^2}}
\end{pmatrix} \\
& = \sqrt{\frac{T+1}{2v_-^2}}(v_+-v_-) - \frac{1}{\sqrt{2}} (u_+-u_-) \\
& = \frac{v_+-v_-}{\sqrt{2}} \left( \sqrt{\frac{T+1}{v_-^2}} + s \right) >0,
\end{split}
\end{equation}
where in the last equality we used the Rankine--Hugoniot relation \eqref{RH},
$u_+-u_-=-s(v_+-v_-)$, while the strict positivity follows from $\textstyle s=\sqrt{\frac{T+1}{v_+v_-}}>0$ and $v_+>v_-$. The simplicity of $\lambda=0$ then follows from Proposition~\ref{D'rel} together with Lemma~\ref{trans}.
\end{proof}

\appendix

\section{The conjugation lemma} \label{Appendix_A}

We state the conjugation lemma of \cite{MeZ}, omitting the proof, which is a refinement of the gap lemma of \cite{GZ, KS}. For a detailed proof, we refer the reader to \cite[Appendix~C]{MZ2}.

Consider a family of $N \times N$ first-order ODE systems
\begin{equation} \label{55}
W' = \mathbb{A} (x, \lambda) W, \quad x \in \mathbb{R}, \ \lambda \in \Omega,
\end{equation}
where $'$ denotes $\textstyle \frac{d}{dx}$ and the domain $\Omega$ is an open subset of $\mathbb{C}$. An eigenvalue problem may be reformulated as \eqref{55} by parametrizing $\lambda$ and, if necessary, rewriting it as a first-order system. We make the following assumption:

\medskip
\noindent\textbf{(H0)} The map $\lambda\mapsto \mathbb{A}(\cdot,\lambda)$ is analytic from $\Omega$ into $L^\infty(\mathbb{R})$. Moreover, there exist limiting matrices $\mathbb{A}_\pm(\lambda)$ such that $\mathbb{A}(x,\lambda)\to\mathbb{A}_\pm(\lambda)$ exponentially as $x\to\pm\infty$, with uniform bounds on $x$-derivatives: for some fixed $K\in\mathbb{N}$ and for $\lambda$ in compact subsets of $\Omega$,
\begin{equation}\label{eq:evans:expconv}
  \left| (\partial/\partial x)^k \left(\mathbb{A}(x,\lambda)-\mathbb{A}_\pm(\lambda) \right)\right| \leq  C e^{-\theta |x|},   \quad x \gtrless 0, \ \ 0 \leq k \leq K
\end{equation}
for some constants $C,\theta >0$.

\medskip
We write the constant-coefficient limiting systems
\begin{equation}\label{A.2}
Z_\pm'=\mathbb A_\pm(\lambda) Z_\pm.
\end{equation}
The following lemma is a restatement of \cite[Lemma~3.1]{MZ2}. 

\begin{lemma}[\cite{MeZ}]\label{Gap}
Assume \textup{(H0)} and fix $\lambda_0 \in \Omega$. There exist $r>0$ such that $\overline{B(\lambda_0,r)} \subset \Omega$ and there exists a pair of linear transformations defined for $\lambda \in B(\lambda_0,r)$,
\begin{equation*}
\begin{split}
&P_+(x,\lambda) = I + \Theta_+(x,\lambda), \quad x \geq 0, \\
&P_-(x,\lambda) = I + \Theta_-(x,\lambda), \quad x \leq 0,
\end{split}
\end{equation*}
where $\Theta_+$ and $\Theta_-$ are analytic in $\lambda$ as functions from $B(\lambda_0,r)$ to $L^\infty \left([0,\infty);\mathbb{C}^{N\times N}\right)$ and $L^\infty\left((-\infty,0];\mathbb{C}^{N\times N}\right)$, respectively. Moreover,
\begin{equation} \label{eq:evans:ThetaBounds}
\left| (\partial/\partial \lambda)^j (\partial / \partial x)^k \Theta_\pm(x,\lambda) \right| \leq C_j \tilde{C} e^{- \tilde{\theta} |x|}, \quad x \gtrless 0, \ \ 0 \leq k \leq K+1, \ \ j \geq 0
\end{equation}
for some constants $\tilde{\theta}, \tilde{C}>0$, where $C_j>0$ depends only on $j$, $\tilde{\theta}$, the modulus of the entries of $\mathbb{A}$ at $\lambda_0$, and the modulus of continuity of $\mathbb{A}$ on $\overline{B(\lambda_0,r)}$. Furthermore, the transformations $P_\pm$ satisfy the following properties:
\begin{enumerate}[(i)]
\item $P_+$ and $P_+^{-1}$ are uniformly bounded on $x \geq 0$, and $P_-$ and $P_-^{-1}$ are uniformly bounded on $x \leq 0$.
\item The change of coordinates $W:=P_\pm Z_\pm$ reduces \eqref{55} on $x \geq 0$ and $x \leq 0$, respectively, to the limiting systems \eqref{A.2}. Equivalently, solutions of \eqref{55} may be factorized as
\begin{equation}\label{eq:evans:factorization}
  W = \left(I+\Theta_\pm \right) Z_\pm,
\end{equation}
where $Z_\pm$ are solutions of \eqref{A.2} and $\Theta_\pm$ satisfy \eqref{eq:evans:ThetaBounds}.
\end{enumerate}
\end{lemma}

We next introduce the domain of consistent splitting, following \cite{AGJ}.

\begin{definition} \label{DCSdef}
The domain of consistent splitting for \eqref{55} is defined as the (open) set of $\lambda$ such that the limiting matrices $\mathbb{A}_\pm$ are hyperbolic, and the dimensions of their stable subspaces $\mathcal{S}_+$ and $\mathcal{S}_-$ (equivalently, of unstable subspaces $\mathcal{U}_+$ and $\mathcal{U}_-$) agree.
\end{definition}

\section{Solvability of the linearized Poisson equation} \label{App_B}

In this appendix, we prove solvability of the linearized Poisson equation \eqref{LNSP3}. This allows us to define the solution operator $\mathcal{A}^{-1}\mathcal{B}$ in \eqref{phisol}.

\begin{lemma} \label{theorem:6.1.1}
Let $(\bar{v},\bar{u},\bar{\phi})$ be the shock profiles described in Lemma~\ref{SP}. There exists a constant $C>0$ such that the following statements hold:
\begin{enumerate}[(i)]
\item The operator $\mathcal A:H^2(\mathbb R;\mathbb C)\to L^2(\mathbb R;\mathbb C)$ defined in \eqref{calAB} is a bounded isomorphism.

\item For any $v\in L^2(\mathbb R;\mathbb C)$, the linearized Poisson equation \eqref{LNSP3} admits a unique weak solution $\phi\in H^1(\mathbb R;\mathbb C)$ satisfying
\begin{equation} \label{Poisson_solution_H1}
\lVert\phi\rVert_{H^1}\le C\lVert v\rVert_{L^2}.
\end{equation}

\item If $v\in H^1(\mathbb R;\mathbb C)$, then $\phi\in H^2(\mathbb R;\mathbb C)$ and
\begin{equation} \label{Poisson_solution_H2}
\lVert\phi\rVert_{H^2}\le C\lVert v\rVert_{H^1}.
\end{equation}

\item If $v\in C^1(\mathbb R;\mathbb C)\cap H^1(\mathbb R;\mathbb C)$, then $\phi\in C^2(\mathbb R;\mathbb C)\cap H^2(\mathbb R;\mathbb C)$.
\end{enumerate}
\end{lemma}

\begin{proof}
The linearized Poisson equation \eqref{LNSP3} can be rewritten as
\begin{equation} \label{5.1}
\left(\bar{v} e^{\bar{\phi}} + \frac{\varepsilon^2 \bar{v}_x}{\bar{v}^2}\partial_x -\frac{\varepsilon^2}{\bar{v}}\partial_{xx}\right)\phi = \left(-e^{\bar{\phi}} -\frac{\varepsilon^2 \bar{\phi}_{xx}}{\bar{v}^2} -\frac{\varepsilon^2 \bar{\phi}_x}{\bar{v}^2}\partial_x +\frac{2\varepsilon^2 \bar{\phi}_x \bar{v}_x}{\bar{v}^3} \right)v.
\end{equation}
We first prove \emph{(i)}. For any $f\in L^2(\mathbb R;\mathbb C)$, we consider
\begin{equation} \label{Aphif}
\mathcal A\phi=\left(\bar{v}e^{\bar{\phi}}+\frac{\varepsilon^2\bar{v}_x}{\bar{v}^2}\partial_x-\frac{\varepsilon^2}{\bar{v}}\partial_{xx}\right)\phi=f.
\end{equation}
To apply the Lax--Milgram Theorem, we define the sesquilinear form $B: H^1(\mathbb{R}) \times H^1(\mathbb{R}) \to \mathbb{C}$ and the conjugate-linear functional $\ell:H^1(\mathbb R)\to \mathbb C$ as
\begin{equation*}
B[w,z] = \int_\mathbb{R} \left( \bar{v}e^{\bar{\phi}} w \overline z + \frac{\varepsilon^2}{\bar{v}}w_x \overline{z_x} \right) \, dx \quad \text{and} \quad \ell(\psi)=\int_{\mathbb R} f \overline\psi \,dx.
\end{equation*}
By the Cauchy--Schwarz inequality, we obtain the bound on $B[w,z]$:
\begin{equation} \label{boundB1}
| B[w,z]| \leq C \int_\mathbb{R} \left( |w||z| + |w_x||z_x| \right) \, dx \leq C \lVert w \rVert_{H^1} \lVert z \rVert_{H^1}.
\end{equation}
Moreover,
\begin{equation} \label{boundB2}
B[w,w] = \int_\mathbb{R} \left( \bar{v} e^{\bar{\phi}} |w|^2 + \frac{\varepsilon^2}{\bar{v}} |w_x|^2 \right) \, dx \geq \frac{v_-}{v_+} \lVert w \rVert_{L^2}^2 + \frac{\varepsilon^2}{v_+} \lVert w_x \rVert_{L^2}^2 \geq c \lVert w \rVert_{H^1}^2
\end{equation}
for some constant $c>0$, where in the first inequality we used
\begin{equation} \label{vphibd}
v_- < \bar{v}(x) < v_+, \quad - \ln v_+ < \bar{\phi}(x) < - \ln v_- \quad \text{for} \ x \in \mathbb{R}.
\end{equation}
Lastly, the conjugate-linear functional $\ell$ is bounded on $H^1(\mathbb{R})$. Indeed, by the Cauchy--Schwarz inequality, $|\ell(\psi)|\le \|f\|_{L^2}\|\psi\|_{L^2}\le \|f\|_{L^2}\|\psi\|_{H^1}$. Therefore, together with the bounds \eqref{boundB1} and \eqref{boundB2}, the Lax--Milgram Theorem gives existence and uniqueness of a function $\phi \in H^1$ satisfying
\begin{equation}\label{Biform}
B[\phi,\psi]=\int_{\mathbb R} f \overline\psi \,dx \quad \text{for all} \ \psi\in H^1(\mathbb R).
\end{equation}
This implies that $\phi$ satisfies \eqref{Aphif} in the sense of distributions. Taking $\psi=\phi$ and using \eqref{boundB2}, we obtain $\|\phi\|_{H^1}\le C\|f\|_{L^2}$. Moreover, rearranging \eqref{Aphif}, we have
\begin{equation} \label{Poisson_regularity}
\phi_{xx}=\frac{\bar v}{\varepsilon^2}\left(\bar v e^{\bar\phi}\phi+\frac{\varepsilon^2\bar v_x}{\bar v^2}\phi_x-f\right).
\end{equation}
By \eqref{spbound} and \eqref{vphibd}, the right-hand side belongs to $L^2$. Hence, $\phi\in H^2$ and
\begin{equation} \label{Poisson_inverse_H2}
\lVert\phi\rVert_{H^2}\le C\left(\lVert\phi\rVert_{H^1}+\lVert f\rVert_{L^2}\right)\le C\lVert f\rVert_{L^2}.
\end{equation}
The boundedness of $\mathcal A:H^2\to L^2$ follows from the boundedness of its coefficients. Together with the existence and uniqueness established above, this proves that $\mathcal A:H^2\to L^2$ is a bounded isomorphism.

We next prove \emph{(ii)}. Suppose first that $v\in H^1$. Since $\mathcal Bv\in L^2$, the solution $\phi=\mathcal A^{-1}\mathcal Bv$ is well-defined by \emph{(i)}. Substituting
\begin{equation*}
f=\mathcal Bv=-e^{\bar\phi}v-\varepsilon^2\left(\frac{\bar\phi_x}{\bar v^2}v\right)_x
\end{equation*}
into \eqref{Biform} and integrating by parts, we obtain
\begin{equation} \label{Poisson_weak_v}
B[\phi,\psi]=-\int_{\mathbb R}e^{\bar\phi}v\overline\psi\,dx+\varepsilon^2\int_{\mathbb R}\frac{\bar\phi_x}{\bar v^2}v\overline{\psi_x}\,dx
\end{equation}
for all $\psi\in H^1$. Taking $\psi=\phi$ in \eqref{Poisson_weak_v} and using the Cauchy--Schwarz inequality together with \eqref{spbound} and \eqref{vphibd}, we obtain
\begin{equation*}
\begin{split}
B[\phi,\phi] &\leq \left|\int_{\mathbb R}e^{\bar\phi}v\overline\phi\,dx\right|+\varepsilon^2\left|\int_{\mathbb R}\frac{\bar\phi_x}{\bar v^2}v\overline{\phi_x}\,dx\right| \\
&\leq \lVert e^{\bar\phi}\rVert_{L^\infty}\lVert v\rVert_{L^2}\lVert\phi\rVert_{L^2}+\varepsilon^2\left\lVert\frac{\bar\phi_x}{\bar v^2}\right\rVert_{L^\infty}\lVert v\rVert_{L^2}\lVert\phi_x\rVert_{L^2} \\
&\leq C\lVert v\rVert_{L^2}\lVert\phi\rVert_{H^1}.
\end{split}
\end{equation*}
Combining this estimate with \eqref{boundB2}, we arrive at
\begin{equation*}
c\lVert\phi\rVert_{H^1}^2\le B[\phi,\phi]\le C\lVert v\rVert_{L^2}\lVert\phi\rVert_{H^1},
\end{equation*}
which proves \eqref{Poisson_solution_H1} for $v\in H^1$. We now extend this result to arbitrary $v\in L^2$ by density. Choose a sequence $\{v_n\}\subset H^1$ such that $v_n\to v$ in $L^2$, and set $\phi_n:=\mathcal A^{-1}\mathcal Bv_n$. By linearity and \eqref{Poisson_solution_H1}, we have
\begin{equation*} 
\lVert\phi_n-\phi_m\rVert_{H^1}\le C\lVert v_n-v_m\rVert_{L^2}. 
\end{equation*}
Thus, $\{\phi_n\}$ is a Cauchy sequence in $H^1$ and converges to a function $\phi\in H^1$ satisfying \eqref{Poisson_solution_H1}. By \eqref{boundB1} and the boundedness of the coefficients on the right-hand side of \eqref{Poisson_weak_v}, the identity \eqref{Poisson_weak_v} also holds for $\phi$ and $v$. Hence, $\phi$ is a weak solution of \eqref{LNSP3}. Uniqueness follows from \eqref{boundB2} applied to the difference of two weak solutions. This proves \emph{(ii)} and gives the unique bounded extension of $\mathcal A^{-1}\mathcal B$ to an operator from $L^2$ to $H^1$.

To prove \emph{(iii)}, let $v\in H^1$. Boundedness of the coefficients in \eqref{5.1} gives $f=\mathcal Bv\in L^2$ and $\|f\|_{L^2}\le C\|v\|_{H^1}$. By \emph{(i)}, the solution $\phi=\mathcal A^{-1}f$ belongs to $H^2$, and \eqref{Poisson_inverse_H2} yields \eqref{Poisson_solution_H2}.

Finally, we prove \emph{(iv)}. If $v\in C^1\cap H^1$, then $f=\mathcal Bv\in C\cap L^2$. By \emph{(iii)} and the Sobolev embedding, both $\phi$ and $\phi_x$ are continuous. The right-hand side of \eqref{Poisson_regularity} is therefore continuous, which gives $\phi\in C^2\cap H^2$.
\end{proof}

\section{Domain and closedness of the linearized operator} \label{App_closed}

We show that the linearized operator $L$ given by \eqref{Ldef}--\eqref{ABDE} is closed and densely defined on $L^2(\mathbb R;\mathbb C^2)$. Its domain $D(L)$ is given by
\begin{equation*}
D(L) = \mathcal{H} := \{ U=(v,u)^{\mathrm{tr}} : v \in H^1(\mathbb R;\mathbb C), \ u\in H^2(\mathbb R;\mathbb C) \}.
\end{equation*}

We first establish a non-divergence representation of $L$ on $\mathcal H$.

\begin{lemma} \label{lemma:L_domain}
For any $U=(v,u)^{\mathrm{tr}}\in\mathcal H$, let $\phi=\mathcal A^{-1}\mathcal Bv$, where $\mathcal A$ and $\mathcal B$ are defined in \eqref{calAB}. Then the following representation holds in the sense of distributions:
\begin{equation} \label{L_nondiv_domain}
LU=\begin{pmatrix}sv_x+u_x\\ su_x+T\left(\dfrac{v}{\bar v^2}\right)_x+\nu\left(\dfrac{u_x}{\bar v}-\dfrac{\bar u_x}{\bar v^2}v\right)_x-\left(\dfrac{\phi_x}{\bar v}-\dfrac{\bar\phi_x}{\bar v^2}v\right)\end{pmatrix}.
\end{equation}
Moreover, there exists a constant $C>0$ such that
\begin{equation} \label{L_domain_bound}
\lVert LU\rVert_{L^2}\le C\left(\lVert v\rVert_{H^1}+\lVert u\rVert_{H^2}\right)
\end{equation}
for all $U\in\mathcal H$.
\end{lemma}

\begin{proof}
The first component in \eqref{Ldef} agrees with that in \eqref{L_nondiv_domain}. Thus, it suffices to treat the second component. Since $v\in H^1$, we have $\phi\in H^2$ by Lemma~\ref{theorem:6.1.1}. All identities below are understood in the sense of distributions. From \eqref{Ldef}--\eqref{ABDE}, we have
\begin{equation*}
\begin{split}
(LU)_2 &= \bigg[ su + \frac{(T+1)v}{\bar v^2} + \nu \left(\frac{u_x}{\bar v}-\frac{\bar u_x}{\bar v^2}v\right) + \varepsilon^2\frac{\bar\phi_x}{\bar v}\left(\frac{\phi_x}{\bar v}-\frac{\bar\phi_x}{\bar v^2}v\right) \\
&\qquad - \frac{\varepsilon^2}{\bar v}\left(\frac{\phi_x}{\bar v}-\frac{\bar\phi_x}{\bar v^2}v\right)_x + \frac{\varepsilon^2 v}{\bar v^2}\left(\frac{\bar\phi_x}{\bar v}\right)_x \bigg]_x.
\end{split}
\end{equation*}
Using \eqref{waveeq3} and \eqref{LNSP3}, we obtain
\begin{equation*}
\begin{split}
-\frac{\varepsilon^2}{\bar v}\left(\frac{\phi_x}{\bar v}-\frac{\bar\phi_x}{\bar v^2}v\right)_x + \frac{\varepsilon^2 v}{\bar v^2}\left(\frac{\bar\phi_x}{\bar v}\right)_x = -\frac{e^{\bar\phi}v}{\bar v} - e^{\bar\phi}\phi + \frac{\varepsilon^2 v}{\bar v^2}\left(\frac{\bar\phi_x}{\bar v}\right)_x = -e^{\bar\phi}\phi - \frac{v}{\bar v^2}.
\end{split}
\end{equation*}
Hence,
\begin{equation*}
(LU)_2 = \left[ su + \frac{Tv}{\bar v^2} + \nu \left(\frac{u_x}{\bar v}-\frac{\bar u_x}{\bar v^2}v\right) + \varepsilon^2\frac{\bar\phi_x}{\bar v}\left(\frac{\phi_x}{\bar v}-\frac{\bar\phi_x}{\bar v^2}v\right) - e^{\bar\phi}\phi \right]_x.
\end{equation*}
For the last two terms, we use \eqref{LNSP3} to obtain
\begin{equation*}
\begin{split}
\left[\varepsilon^2\frac{\bar\phi_x}{\bar v}\left(\frac{\phi_x}{\bar v}-\frac{\bar\phi_x}{\bar v^2}v\right)-e^{\bar\phi}\phi\right]_x & = \varepsilon^2\left(\frac{\bar\phi_x}{\bar v}\right)_x\left(\frac{\phi_x}{\bar v}-\frac{\bar\phi_x}{\bar v^2}v\right) + e^{\bar\phi}\left(\frac{\bar\phi_x}{\bar v}v-\phi_x\right) \\
& = -\left(\frac{\phi_x}{\bar v}-\frac{\bar\phi_x}{\bar v^2}v\right),
\end{split}
\end{equation*}
where the last equality follows from \eqref{waveeq3}. This proves \eqref{L_nondiv_domain}. Moreover, by \eqref{spbound} and \eqref{vphibd}, all terms on the right-hand side of \eqref{L_nondiv_domain} belong to $L^2$, and
\begin{equation*}
\lVert LU\rVert_{L^2}\le C\left(\lVert v\rVert_{H^1}+\lVert u\rVert_{H^2}+\lVert\phi\rVert_{H^1}\right)\le C\left(\lVert v\rVert_{H^1}+\lVert u\rVert_{H^2}\right),
\end{equation*}
where the last inequality follows from \eqref{Poisson_solution_H1}. This proves \eqref{L_domain_bound}.
\end{proof}

\begin{lemma}
There exists a constant $C>0$ such that
\begin{equation} \label{2.1}
\lVert v\rVert_{H^1}+\lVert u\rVert_{H^2}\leq C\left(\lVert LU\rVert_{L^2}+\lVert U\rVert_{L^2}\right)
\end{equation}
for all $U=(v,u)^{\mathrm{tr}}\in\mathcal H$.
\end{lemma}

\begin{proof}
Let $f=(f_1,f_2)^{\mathrm{tr}}:=LU$. By \eqref{L_nondiv_domain}, with $\phi=\mathcal A^{-1}\mathcal Bv$, we have
\begin{subequations}
\begin{align}
&\label{2.4a}sv_x+u_x=f_1,\\
&\label{2.4b}su_x+T\left(\frac{v}{\bar v^2}\right)_x+\nu\left(\frac{u_x}{\bar v}-\frac{\bar u_x}{\bar v^2}v\right)_x-\left(\frac{\phi_x}{\bar v}-\frac{\bar\phi_x}{\bar v^2}v\right)=f_2.
\end{align}
\end{subequations}
Taking the $L^2$-inner product of \eqref{2.4a} with $sv_x$, taking the real part, and integrating by parts, we have
\begin{equation*}
\begin{split}
\int_{\mathbb R}|sv_x|^2\,dx&=\operatorname{Re}\int_{\mathbb R}s\overline{v_x}f_1\,dx+\operatorname{Re}\int_{\mathbb R}s\overline v u_{xx}\,dx\\
&\leq\frac12\int_{\mathbb R}|sv_x|^2\,dx+\eta\int_{\mathbb R}|u_{xx}|^2\,dx+\frac C\eta\int_{\mathbb R}|v|^2\,dx+C\int_{\mathbb R}|f_1|^2\,dx,
\end{split}
\end{equation*}
where the inequality follows from Young's inequality with an arbitrarily small constant $\eta>0$. Hence,
\begin{equation} \label{tmp11}
\int_{\mathbb R}|v_x|^2\,dx\leq\eta\int_{\mathbb R}|u_{xx}|^2\,dx+\frac C\eta\int_{\mathbb R}|v|^2\,dx+C\int_{\mathbb R}|f_1|^2\,dx.
\end{equation}
Similarly, one may obtain by taking the real part of the $L^2$-inner product of \eqref{2.4a} with $u_x$ that
\begin{equation} \label{tmp12}
\int_{\mathbb R}|u_x|^2\,dx\leq C\int_{\mathbb R}\left(|v_x|^2+|f_1|^2\right)\,dx\leq\eta\int_{\mathbb R}|u_{xx}|^2\,dx+\frac C\eta\int_{\mathbb R}|v|^2\,dx+C\int_{\mathbb R}|f_1|^2\,dx,
\end{equation}
where in the second inequality, we used \eqref{tmp11}. We take the real part of the $L^2$-inner product of \eqref{2.4b} with $\textstyle\nu u_{xx}$:
\begin{equation*}
\begin{split}
\int_{\mathbb R}\frac{|\nu u_{xx}|^2}{\bar v}\,dx&=-\operatorname{Re}\int_{\mathbb R}\left(\frac{\nu T}{\bar v^2}\overline{u_{xx}}v_x-\frac{2\nu T\bar v_x}{\bar v^3}\overline{u_{xx}}v\right)\,dx\\
&\quad+\operatorname{Re}\int_{\mathbb R}\left(\frac{\nu^2\bar v_x}{\bar v^2}\overline{u_{xx}}u_x+\nu\left(\frac{\nu\bar u_x}{\bar v^2}\right)_x\overline{u_{xx}}v+\frac{\nu^2\bar u_x}{\bar v^2}\overline{u_{xx}}v_x\right)\,dx\\
&\quad-\operatorname{Re}\int_{\mathbb R}\left(\frac{\nu\bar\phi_x}{\bar v^2}\overline{u_{xx}}v-\frac{\nu}{\bar v}\overline{u_{xx}}\phi_x\right)\,dx+\operatorname{Re}\int_{\mathbb R}\nu\overline{u_{xx}}f_2\,dx.
\end{split}
\end{equation*}
Here we used $\operatorname{Re}\int_{\mathbb R}u_x\overline{u_{xx}}\,dx=0$. By Young's inequality and the bounds \eqref{spbound} and \eqref{vphibd},
\begin{equation*}
\int_{\mathbb R}|u_{xx}|^2\,dx\leq\eta\int_{\mathbb R}|u_{xx}|^2\,dx+\frac C\eta\int_{\mathbb R}\left(|v_x|^2+|v|^2+|u_x|^2+|\phi_x|^2+|f_2|^2\right)\,dx.
\end{equation*}
Absorbing the first term on the right-hand side and then using \eqref{tmp11}--\eqref{tmp12} with a sufficiently small $\eta>0$, we obtain
\begin{equation} \label{tmp13}
\begin{split}
\int_{\mathbb R}|u_{xx}|^2\,dx&\leq C\int_{\mathbb R}\left(|v_x|^2+|v|^2+|u_x|^2+|\phi_x|^2+|f_2|^2\right)\,dx\\
&\leq\eta\int_{\mathbb R}|u_{xx}|^2\,dx+\frac C\eta\int_{\mathbb R}|v|^2\,dx+C\int_{\mathbb R}\left(|v|^2+|\phi_x|^2+|f_1|^2+|f_2|^2\right)\,dx.
\end{split}
\end{equation}
Adding \eqref{tmp11}--\eqref{tmp13} and choosing $\eta$ sufficiently small, we arrive at
\begin{equation*}
\lVert v_x\rVert_{L^2}^2+\lVert u_x\rVert_{L^2}^2+\lVert u_{xx}\rVert_{L^2}^2\leq C\left(\lVert v\rVert_{L^2}^2+\lVert f_1\rVert_{L^2}^2+\lVert f_2\rVert_{L^2}^2\right),
\end{equation*}
which implies \eqref{2.1}. Here we used \eqref{Poisson_solution_H1} to estimate $\lVert\phi\rVert_{H^1}^2\leq C\lVert v\rVert_{L^2}^2$.
\end{proof}

\begin{proposition} \label{Prop:2.2}
The linear operator $L$ is closed and densely defined on $L^2(\mathbb R;\mathbb C)\times L^2(\mathbb R;\mathbb C)$ with domain $\mathcal H$.
\end{proposition}

\begin{proof}
To prove closedness, we consider a sequence $\{U_k\}_{k=1}^\infty=\{(v_k,u_k)^{\mathrm{tr}}\}_{k=1}^\infty\subset\mathcal H$ satisfying
\begin{equation*}
U_k\to U=(v,u)^{\mathrm{tr}}\quad\text{and}\quad LU_k\to F=(f,g)^{\mathrm{tr}}\quad\text{as}\quad k\to\infty
\end{equation*}
in $L^2$. By \eqref{2.1}, we have
\begin{equation*}
\lVert v_k-v_l\rVert_{H^1}+\lVert u_k-u_l\rVert_{H^2}\leq C\left(\lVert L(U_k-U_l)\rVert_{L^2}+\lVert U_k-U_l\rVert_{L^2}\right)
\end{equation*}
for all $k$ and $l$. Hence $\{U_k\}$ is a Cauchy sequence in $\mathcal H$. Since $\mathcal H$ is complete, there exists $\widetilde U=(\tilde v,\tilde u)^{\mathrm{tr}}\in\mathcal H$ such that $U_k\to\widetilde U$ in $\mathcal H$. In particular, $U_k\to\widetilde U$ in $L^2$, and by uniqueness of the limit in $L^2$ we have $\widetilde U=U$. Thus $U\in\mathcal H$. Moreover, since $U_k\to U$ in $\mathcal H$ and $L:\mathcal H\to L^2$ is continuous by \eqref{L_domain_bound}, we obtain $LU_k\to LU$ in $L^2$. Comparing with $LU_k\to F$ in $L^2$ yields $F=LU$. Therefore $L$ is closed, and it is densely defined in $L^2$ since $\mathcal H$ is dense in $L^2$.
\end{proof}

\section{Higher-order energy estimates} \label{App_C}
 
Differentiating \eqref{3.3a}--\eqref{3.3b} yields the first two equations of the eigenvalue equation \eqref{EE''} for $(v,u) = (\Phi_x,\Psi_x)$, while \eqref{3.3c} coincides with the third equation of \eqref{EE''}. For the higher-order energy estimates, it is convenient to work with the following non-divergence form of the eigenvalue problem associated with \eqref{LNSP1} and \eqref{LNSP2'}:
\begin{subequations} \label{3.32}
\begin{align}
& \label{3.32a} \lambda v - s v_x - u_x = 0, \\
& \label{3.32b} \lambda u - s u_x - T\Big(\frac{v}{\hat{v}^2}\Big)_x = \nu\Big(\frac{1}{\hat{v}}u_x - \frac{\hat{u}_x}{\hat{v}^2}v\Big)_x - \Big(\frac{1}{\hat{v}}\phi_x - \frac{\hat{\phi}_x}{\hat{v}^2}v\Big).
\end{align}
\end{subequations}
This Appendix provides estimates for $|u|,|u_x|$, and $|v_x|$, which are used in the proof of Lemma~\ref{Lemma:3.2}. By Lemma~\ref{lemma:integrated_regularity}, we have $v,u,\phi\in H^2$. This regularity justifies the integrations by parts below.

\begin{lemma}
Under the assumptions in Lemma~\ref{Lemma:3.2}, there exists a constant $C>0$ such that
\begin{equation} \label{uL2}
\operatorname{Re}{\lambda} \lVert u \rVert_{L^2}^2 + \lVert u_x \rVert_{L^2}^2 \leq C \delta_S \lVert v_x \rVert_{L^2}^2 + C \left( \lVert v \rVert_{L^2}^2 + \lVert u \rVert_{L^2}^2 + \lVert \phi_x \rVert_{L^2}^2 \right).
\end{equation}
\end{lemma}

\begin{proof}
Multiplying \eqref{3.32b} by $\overline{u}$, taking the real part, and integrating over $\mathbb{R}$, we obtain after several applications of integration by parts
\begin{equation*}
\begin{split}
& \operatorname{Re}{\lambda} \int_\mathbb{R} |u|^2 \, dx + \int_\mathbb{R} \frac{\nu}{\hat{v}} |u_x|^2 \, dx \\
& \quad = - \int_\mathbb{R} \frac{T}{\hat{v}^2} \operatorname{Re}{(v\overline{u_x})} \, dx - \int_\mathbb{R} \frac{1}{\hat{v}} \operatorname{Re}{(\phi_x \overline{u})} \, dx \\
& \qquad + \int_\mathbb{R} \frac{\hat{\phi}_x}{\hat{v}^2} \operatorname{Re}{ ( v \overline{u})} \, dx - \int_\mathbb{R} \nu \operatorname{Re}{ \bigg[ \bigg( \frac{\hat{u}_x}{\hat{v}^2} \bigg)_x v\overline{u} + \frac{\hat{u}_x}{\hat{v}^2} v_x \overline{u} \bigg] } \, dx.
\end{split}
\end{equation*}
Applying Young's inequality, we obtain a bound on the first two terms on the right-hand side:
\begin{equation*}
\begin{split}
& \bigg| - \int_\mathbb{R} \frac{T}{\hat{v}^2} \operatorname{Re}{(v \overline{u}_x)} \, dx - \int_\mathbb{R} \frac{1}{\hat{v}} \operatorname{Re}{(\phi_x \overline{u})} \, dx \bigg| \\
& \quad \leq \eta \left( \lVert u \rVert_{L^2}^2 + \lVert u_x \rVert_{L^2}^2 \right) + \frac{C}{\eta} \left( \|v\|_{L^2}^2 + \|\phi_x\|_{L^2}^2 \right).
\end{split}
\end{equation*}
Thanks to \eqref{spbound}, the remaining terms are bounded by
\begin{equation*}
C \delta_S \left( \|v\|_{H^1}^2 + \|u\|_{L^2}^2 \right).
\end{equation*}
Combining the estimates and taking $\eta$ sufficiently small, we arrive at the desired inequality.
\end{proof}

\begin{lemma}
Under the assumptions in Lemma~\ref{Lemma:3.2}, there exists a constant $C>0$ such that
\begin{equation} \label{uxL2}
\operatorname{Re}{\lambda} \lVert u_x \rVert_{L^2}^2 + \lVert u_{xx} \rVert_{L^2}^2 \leq C \delta_S \left( \lVert v \rVert_{L^2}^2 + \lVert u \rVert_{L^2}^2 \right) + C \left( \lVert v_x \rVert_{L^2}^2 + \lVert \phi_x \rVert_{L^2}^2 \right).
\end{equation}
\end{lemma}

\begin{proof}
Multiplying \eqref{3.32b} by $-\overline{u_{xx}}$, taking the real part, and integrating over $\mathbb{R}$, we obtain
\begin{equation*}
\begin{split}
\operatorname{Re}{\lambda} \int_\mathbb{R} |u_x|^2 \, dx + \int_\mathbb{R} \frac{\nu}{\hat{v}} |u_{xx}|^2 \, dx & = - \int_\mathbb{R} T \operatorname{Re}{ \bigg[ \overline{u_{xx}}  \bigg( \frac{1}{\hat{v}^2} v_x - \frac{2\hat{v}_x}{\hat{v}^3} v \bigg) \bigg] } \, dx \\
& \quad + \int_\mathbb{R} \nu \operatorname{Re}{ \bigg[ \overline{u_{xx}} \bigg( \frac{\hat{v}_x}{\hat{v}^2} u_x + \bigg( \frac{\hat{u}_x}{\hat{v}^2} \bigg)_x v + \frac{\hat{u}_x}{\hat{v}^2} v_x \bigg) \bigg] } \, dx \\
& \quad + \int_\mathbb{R} \operatorname{Re}{ \bigg[ \overline{u_{xx}} \bigg(\frac{1}{\hat{v}}\phi_x - \frac{\hat{\phi}_x}{\hat{v}^2} v \bigg) \bigg] } \, dx.
\end{split}
\end{equation*}
By Young's inequality, \eqref{spbound}, and \eqref{uL2}, the right-hand side is bounded by
\begin{equation*}
\eta \lVert u_{xx} \rVert_{L^2}^2 + \frac{C}{\eta} \left( \lVert v_x \rVert_{L^2}^2 + \lVert \phi_x \rVert_{L^2}^2 \right) + C \delta_S \left( \lVert v \rVert_{H^1}^2 + \lVert u \rVert_{L^2}^2 \right)
\end{equation*}
for sufficiently small $\delta_S$. Therefore we obtain \eqref{uxL2}.
\end{proof}

\begin{lemma}
Under the assumptions in Lemma~\ref{Lemma:3.2}, there exists a constant $C>0$ such that
\begin{equation} \label{vxL2}
\operatorname{Re}{\lambda} \lVert v_x \rVert_{L^2}^2 + \lVert v_x \rVert_{L^2}^2 \leq C \left( \lVert v \rVert_{L^2}^2 + \lVert u \rVert_{L^2}^2 + \lVert \phi_x \rVert_{L^2}^2 \right).
\end{equation}
\end{lemma}

\begin{proof}
By multiplying \eqref{3.32a} and \eqref{3.32b} by $-\nu \overline{v_{xx}}$ and $-\hat{v}\overline{v_x}$, respectively, then taking the real part and summing the two equations, 
we obtain
\begin{equation*}
\begin{split}
& - \nu \operatorname{Re}{(\lambda \overline{v_{xx}} v)} - \hat{v} \operatorname{Re}{(\lambda \overline{v_x} u)} + s \nu \operatorname{Re}{(\overline{v_{xx}} v_x)} + \nu \operatorname{Re}({\overline{v_{xx}} u_x)} \\
& \qquad + s \hat{v} \operatorname{Re}({\overline{v_x} u_x)} + T \operatorname{Re}{\bigg[\overline{v_x} \bigg( \frac{1}{\hat{v}} v_x - \frac{2\hat{v}_x}{\hat{v}^2} v \bigg)\bigg] } \\
& \quad = -\nu \operatorname{Re}{ \bigg[ \overline{v_x} \bigg( u_{xx} - \frac{\hat{v}_x}{\hat{v}} u_x - \hat{v}\bigg( \frac{\hat{u}_x}{\hat{v}^2} \bigg)_x v - \frac{\hat{u}_x}{\hat{v}} v_x \bigg) \bigg] } +  \operatorname{Re}{\bigg[\overline{v_x} \bigg( \phi_x -\frac{\hat{\phi}_x}{\hat{v}}v \bigg)\bigg]}.
\end{split}
\end{equation*}
We integrate this identity over $\mathbb{R}$ and apply integration by parts, which yields
\begin{equation*}
\begin{split}
\operatorname{Re}{\lambda}\int_\mathbb{R} \nu |v_x|^2 \, dx + \int_\mathbb{R} \frac{T}{\hat{v}} |v_x|^2 \, dx = I + II + III,
\end{split}
\end{equation*}
where
\begin{equation*}
\begin{split}
I & := \int_\mathbb{R} \hat{v} \operatorname{Re}{(\lambda \overline{v_x} u)} \, dx, \\
II & := - \int_\mathbb{R} s \hat{v} \operatorname{Re}{(\overline{v_x}u_x)} \, dx + \int_\mathbb{R} \operatorname{Re}{(\overline{v_x}\phi_x)} \, dx, \\
III & := \int_\mathbb{R} \bigg( \frac{2T\hat{v}_x}{\hat{v}^2} - \frac{\hat{\phi}_x}{\hat{v}} \bigg) \operatorname{Re}{(\overline{v_x}v)} \, dx \\
& \qquad + \int_\mathbb{R} \nu \operatorname{Re}{ \bigg[ \overline{v_x} \bigg( \frac{\hat{v}_x}{\hat{v}} u_x + \hat{v} \bigg( \frac{\hat{u}_x}{\hat{v}^2} \bigg)_x v + \frac{\hat{u}_x}{\hat{v}} v_x \bigg) \bigg] } \, dx.
\end{split}
\end{equation*}
We estimate the terms $I$, $II$, and $III$. First, we estimate $I$ using \eqref{3.32a}. Taking the complex conjugate of \eqref{3.32a}, differentiating with respect to $x$, multiplying by $\hat{v} u$, taking the real part, and integrating over $\mathbb{R}$, we obtain
\begin{equation*}
\int_\mathbb{R} \hat{v} \operatorname{Re}(\overline{\lambda} \overline{v_x} u) \, dx - \int_\mathbb{R} s \hat{v} \operatorname{Re}(\overline{v_{xx}} u) \, dx - \int_\mathbb{R} \hat{v} \operatorname{Re}(\overline{u_{xx}} u) \, dx = 0.
\end{equation*}
Let $b:=v_x\overline{u}$ so that $\overline{b}=\overline{v_x}u$. Using the identity
$\operatorname{Re}(a\overline{b}) + \operatorname{Re}(\overline{a} \overline{b}) = 2\operatorname{Re}(a) \operatorname{Re}(b)$ with $a=\lambda$,
we have
\begin{equation*}
\begin{split}
I & = \operatorname{Re} \lambda \int_\mathbb{R} 2\hat{v} \operatorname{Re}(v_x \overline{u}) \, dx - \int_\mathbb{R} s \hat{v}\operatorname{Re}(\overline{v_{xx}} u) \, dx - \int_\mathbb{R} \hat{v} \operatorname{Re}(\overline{u_{xx}} u) \, dx \\
& = \operatorname{Re} \lambda \int_\mathbb{R} 2\hat{v} \operatorname{Re}(v_x \overline{u}) \, dx + \int_\mathbb{R} s \hat{v}_x \operatorname{Re} (\overline{v_x} u) \, dx + \int_\mathbb{R} s\hat{v} \operatorname{Re}(\overline{v_x}u_x) \, dx \\
& \quad + \int_\mathbb{R} \hat{v}_x \operatorname{Re}(\overline{u_x} u) \, dx + \int_\mathbb{R} \hat{v} |u_x|^2 \, dx.
\end{split}
\end{equation*}
By Young's inequality and \eqref{spbound}, for any $\eta>0$,
\begin{equation*}
\begin{split}
|I| & \leq \eta \left( \operatorname{Re} \lambda \int_\mathbb{R} |v_x|^2 \, dx + \int_\mathbb{R} |v_x|^2 \, dx \right) + \frac{C}{\eta} \left( \operatorname{Re} \lambda \int_\mathbb{R} |u|^2 \, dx + \int_\mathbb{R} |u_x|^2 \, dx \right) \\
& \quad + C \delta_S \int_\mathbb{R} \left( |v_x|^2 + |u|^2 + |u_x|^2 \right) \, dx.
\end{split}
\end{equation*}
We apply Young's inequality to estimate $II$:
\begin{equation*}
|II| \leq \eta \int_\mathbb{R} |v_x|^2 \, dx + \frac{C}{\eta} \int_\mathbb{R} \left( |u|^2 + |u_x|^2 + |\phi_x|^2 \right) \, dx.
\end{equation*}
Lastly, using the bound \eqref{spbound}, we have
\begin{equation*}
|III| \leq C \delta_S \int_\mathbb{R} \left( |v|^2 + |v_x|^2 + |u_x|^2 \right) \, dx.
\end{equation*}
Collecting the above estimates and choosing $\eta$ small enough, we get
\begin{equation*}
\operatorname{Re}{\lambda} \lVert v_x \rVert_{L^2}^2 + \lVert v_x \rVert_{L^2}^2 \leq C \delta_S \lVert v \rVert_{H^1}^2 + C \left( \operatorname{Re}{\lambda} \lVert u \rVert_{L^2}^2 + \lVert u \rVert_{H^1}^2 + \lVert \phi_x \rVert_{L^2}^2 \right).
\end{equation*}
Here, to control the terms $\operatorname{Re}{\lambda} \lVert u \rVert_{L^2}^2 + \lVert u_x \rVert_{L^2}^2$, we make use of \eqref{uL2}. Then, we have
\begin{equation*}
\operatorname{Re}{\lambda} \lVert v_x \rVert_{L^2}^2 + \lVert v_x \rVert_{L^2}^2 \leq C \delta_S \lVert v \rVert_{H^1}^2 + C \left( \lVert v \rVert_{L^2}^2 + \lVert u \rVert_{L^2}^2 + \lVert \phi_x \rVert_{L^2}^2 \right).
\end{equation*}
Finally, using the smallness of $\delta_S$, the desired inequality follows.
\end{proof}

\section{Transversality of the shock profiles} \label{App_trans}

In this appendix, we provide the proof of Lemma~\ref{trans}, establishing the transversality of the small-amplitude shock profiles described in Lemma~\ref{SP}.

\begin{proof}[Proof of Lemma~\ref{trans}]

At $\lambda=0$, we integrate the first and third relations in \eqref{coordinate} from $x$ to $+\infty$ for $\varphi_1^+$ and $\varphi_2^+$, and from $x$ to $-\infty$ for $\varphi_3^-$ and $\varphi_4^-$. Using the exponential decay of these fast modes, we obtain
\begin{equation} \label{Aconstraints}
v = - \frac{1}{s}u, \quad \tilde{u}' =  (b_1' - A_{21})v + (b_2'-s)u - D_{21}\phi' - E_{21}\phi''.
\end{equation}
These constraints reduce \eqref{5} to a three-dimensional ODE system, namely,
\begin{equation} \label{v_eq}
V'=\widetilde{A}(x;s)V,
\end{equation}
where $V = (u,\phi,\phi')^{\mathrm{tr}}$. Note that $\bar{V}' := (\bar{u}',\bar{\phi}',\bar{\phi}'')^{\mathrm{tr}}$ is a nontrivial solution of \eqref{v_eq} and is uniformly bounded on $\mathbb{R}$ by the bound \eqref{spbound}. A standard fact (see, e.g., \cite[Section~10.4]{ZH} and \cite[Section~1.4]{MZ2}) is that the intersection of the unstable and stable manifolds of the three-dimensional first-order system corresponding to \eqref{waveeq} is transverse if and only if the (variational) equation \eqref{v_eq} admits a one-dimensional space of solutions bounded on $\mathbb{R}$, necessarily spanned by $\bar{V}'$. Thus, it suffices to show that \eqref{v_eq} has no globally bounded solution linearly independent of $\bar{V}'$.

Let
\begin{equation*}
A_0 := \lim_{x \to -\infty} \widetilde{A}(x;v_-^{-1}\sqrt{T+1}) = \begin{pmatrix}
- \frac{1}{\nu\sqrt{T+1}} & \frac{1}{\nu} & 0 \\
0 & 0 & 1 \\
-\frac{v_-}{\varepsilon^2 \sqrt{T+1}} & \frac{v_-}{\varepsilon^2} & 0
\end{pmatrix}.
\end{equation*}
The eigenvalues of $A_0$ are given by
\begin{equation*}
\begin{split}
\sigma_- & = -\frac{1}{2\nu\sqrt{T+1}} - \frac{1}{2} \sqrt{\frac{1}{\nu^2(T+1)} + \frac{4v_-}{\varepsilon^2}} <0, \\
\sigma_0 & = 0, \\
\sigma_+ & = -\frac{1}{2\nu\sqrt{T+1}} + \frac{1}{2} \sqrt{\frac{1}{\nu^2(T+1)} + \frac{4v_-}{\varepsilon^2}} >0.
\end{split}
\end{equation*}
Moreover, using the bound \eqref{spbound} and $\textstyle s=\sqrt{(T+1)/(v_+v_-)}$, we find
\begin{equation*}
\left| \frac{\sqrt{T+1}}{v_-} - s  \right| = \left |\frac{\sqrt{T+1}}{v_-} - \sqrt{\frac{T+1}{v_+v_-}} \right| = \left| \frac{\sqrt{T+1}}{\sqrt{v_+}v_-(\sqrt{v_+}+\sqrt{v_-})} (v_+-v_-) \right| \leq C \delta_S,
\end{equation*}
and hence, by the defining relations \eqref{coordinate} and \eqref{Aconstraints} for $\widetilde{A}$,
\begin{equation*}
\lVert \widetilde{A}(x;s) - A_0 \rVert_{L^\infty(\mathbb{R})} \leq C \delta_S.
\end{equation*}
Let $r_-$, $r_0$, and $r_+$ be eigenvectors of $A_0$ associated with $\sigma_-$, $0$, and $\sigma_+$, and set $\widetilde{R}:= (r_-,r_0,r_+)$. Since the eigenvalues are distinct, $\widetilde{R}$ is invertible, and with the change of variables $V=\widetilde{R}Z$ the system \eqref{v_eq} becomes
\begin{equation} \label{Z_system}
Z' = (\widetilde{M} + \widetilde{\Theta}(x)) Z, \quad \widetilde{M} := \operatorname{diag}(\sigma_-,0,\sigma_+), \quad \widetilde{\Theta}(x) := \widetilde{R}^{-1}(\widetilde{A}(x;s) - A_0)\widetilde{R},
\end{equation}
with
\begin{equation} \label{Theta_bd}
\lVert \widetilde{\Theta} \rVert_{L^\infty(\mathbb{R})} \leq C \delta_S.
\end{equation}
Set
\begin{equation*}
\eta := \frac{1}{2} \min \{ -\sigma_-, \sigma_+ \} >0.
\end{equation*}
For $\delta_S$ sufficiently small, \eqref{Theta_bd} implies $\lVert \widetilde{\Theta} \rVert_{L^\infty} \ll \eta$. We now apply the tracking/reduction lemma \cite[Proposition~3.9]{MZ2} to obtain invariant graphs. We then show that all globally bounded solutions of \eqref{Z_system} lie in both graphs.

We write $Z=(Z_{sc},Z_+)$ with $Z_{sc}:=(Z_-,Z_0) \in \mathbb{R}^2$ and $Z_+\in\mathbb{R}$. Then $\widetilde{M} = \operatorname{diag} (M_{sc},\sigma_+)$ with $M_{sc} := \operatorname{diag} (\sigma_-,0)$. For $x\geq y$, the decoupled flows satisfy
\begin{equation*}
\lVert e^{M_{sc}(x-y)}\rVert\lvert e^{-\sigma_+(x-y)}\rvert\leq e^{-2\eta(x-y)}.
\end{equation*}
Thus, the exponential separation assumption of \cite[Proposition~3.9]{MZ2} holds. Together with \eqref{Theta_bd} and the smallness of $\lVert\widetilde{\Theta}\rVert_{L^\infty}/\eta$, this proposition gives a $C^1$ family of linear maps
\begin{equation} \label{Phi+_bd}
\widetilde{\Phi}_+(x) : \mathbb{R}^2 \to \mathbb{R}, \quad \sup_{x \in \mathbb{R}}{\lVert \widetilde{\Phi}_+(x) \rVert} \leq C \frac{\lVert \widetilde{\Theta} \rVert_{L^\infty(\mathbb{R})}}{\eta} \leq C \delta_S,
\end{equation}
such that the graph
\begin{equation*}
\mathcal{G}_+(x) := \{ (Z_{sc},Z_+): Z_+ = \widetilde{\Phi}_+ (x) Z_{sc} \}
\end{equation*}
is invariant under the flow of \eqref{Z_system}. Similarly, we write $Z=(Z_-,Z_{cu})$ with $Z_{cu}:= (Z_0,Z_+)^{\mathrm{tr}} \in \mathbb{R}^2$. Setting $M_{cu}:=\operatorname{diag}(0,\sigma_+)$, we have
\begin{equation*}
\lvert e^{\sigma_-(x-y)}\rvert\lVert e^{-M_{cu}(x-y)}\rVert\leq e^{-2\eta(x-y)},\quad x\geq y.
\end{equation*}
Thus, applying the same proposition, we also obtain a $C^1$ family of linear maps
\begin{equation} \label{Phi-_bd}
\widetilde{\Phi}_-(x): \mathbb{R}^2 \to \mathbb{R}, \quad \sup_{x \in \mathbb{R}}{\lVert \widetilde{\Phi}_- (x) \rVert} \leq C \delta_S
\end{equation}
such that the graph
\begin{equation*}
\mathcal{G}_-(x) := \left\{ (Z_-,Z_{cu}): Z_- = \widetilde{\Phi}_-(x) Z_{cu} \right\}
\end{equation*}
is invariant under the flow of \eqref{Z_system}.

Now we claim that a solution of \eqref{Z_system} bounded on $\mathbb R$ must lie in both graphs. To this end, we set
\begin{equation*}
w_+:=Z_+-\widetilde{\Phi}_+Z_{sc},\quad w_-:=Z_--\widetilde{\Phi}_-Z_{cu}.
\end{equation*}
By the definitions of $\mathcal G_\pm$, it suffices to show that $w_\pm\equiv0$ for any solution $Z$ bounded on $\mathbb R$. We first write \eqref{Z_system} as
\begin{equation} \label{ZscZ+}
\begin{split}
Z_{sc}'&=\left(M_{sc}+\begin{pmatrix}\widetilde{\Theta}_{11}&\widetilde{\Theta}_{12}\\\widetilde{\Theta}_{21}&\widetilde{\Theta}_{22}\end{pmatrix}\right)Z_{sc}+\begin{pmatrix}\widetilde{\Theta}_{13}\\\widetilde{\Theta}_{23}\end{pmatrix}Z_+,\\
Z_+'&=(\widetilde{\Theta}_{31},\widetilde{\Theta}_{32})Z_{sc}+(\sigma_++\widetilde{\Theta}_{33})Z_+.
\end{split}
\end{equation}
By the invariance of $\mathcal G_+$, the relation $Z_+=\widetilde{\Phi}_+Z_{sc}$ holds along solutions in $\mathcal G_+$. Differentiating this relation and using the equations \eqref{ZscZ+}, we have
\begin{equation*}
\begin{split}
\widetilde{\Phi}_+'Z_{sc} = Z_+'-\widetilde{\Phi}_+Z_{sc}' &=(\widetilde{\Theta}_{31},\widetilde{\Theta}_{32})Z_{sc}+(\sigma_++\widetilde{\Theta}_{33})\widetilde{\Phi}_+Z_{sc}\\
&\quad-\widetilde{\Phi}_+\left(M_{sc}+\begin{pmatrix}\widetilde{\Theta}_{11}&\widetilde{\Theta}_{12}\\\widetilde{\Theta}_{21}&\widetilde{\Theta}_{22}\end{pmatrix}\right)Z_{sc}-\widetilde{\Phi}_+\begin{pmatrix}\widetilde{\Theta}_{13}\\\widetilde{\Theta}_{23}\end{pmatrix}\widetilde{\Phi}_+Z_{sc}.
\end{split}
\end{equation*}
Here we note that this identity holds for any $Z_{sc}\in\mathbb{R}^2$. Hence,
\begin{equation} \label{Phi+id}
\begin{split}
\widetilde{\Phi}_+' &= (\widetilde{\Theta}_{31}, \widetilde{\Theta}_{32}) - \widetilde{\Phi}_+ \left(M_{sc} + \begin{pmatrix}
\widetilde{\Theta}_{11} & \widetilde{\Theta}_{12} \\
\widetilde{\Theta}_{21} & \widetilde{\Theta}_{22}
\end{pmatrix}\right) + \left( \sigma_+ + \widetilde{\Theta}_{33} - \widetilde{\Phi}_+ \begin{pmatrix}
\widetilde{\Theta}_{13} \\ 
\widetilde{\Theta}_{23}
\end{pmatrix} \right) \widetilde{\Phi}_+,
\end{split}
\end{equation}
where $\widetilde{\Phi}_{\pm,j}$ denotes the $j$-th entry of $\widetilde{\Phi}_\pm$. For any solution of \eqref{Z_system}, the identities \eqref{ZscZ+} and \eqref{Phi+id} give
\begin{equation} \label{w+ODE}
\begin{split}
w_+' &=Z_+'-\widetilde{\Phi}_+'Z_{sc}-\widetilde{\Phi}_+Z_{sc}' \\
& = \left( \sigma_+ + \widetilde{\Theta}_{33} - \widetilde{\Phi}_+ \begin{pmatrix}
\widetilde{\Theta}_{13} \\ 
\widetilde{\Theta}_{23}
\end{pmatrix} \right) (Z_+-\widetilde{\Phi}_+ Z_{sc}) \\
& = \left( \sigma_+ + \widetilde{\Theta}_{33} - \widetilde{\Phi}_{+,1} \widetilde{\Theta}_{13} - \widetilde{\Phi}_{+,2} \widetilde{\Theta}_{23} \right) w_+.
\end{split}
\end{equation}
For $w_-$, similarly, we obtain from the invariance of $\mathcal G_-$ that
\begin{equation} \label{w-ODE}
w_-'=\left(\sigma_-+\widetilde{\Theta}_{11}-\widetilde{\Phi}_{-,1}\widetilde{\Theta}_{21}-\widetilde{\Phi}_{-,2}\widetilde{\Theta}_{31}\right)w_-.
\end{equation}
On the other hand, by \eqref{Theta_bd} and the bounds on $\widetilde{\Phi}_\pm$, it holds that
\begin{equation*}
\begin{split}
\sigma_+ + \widetilde{\Theta}_{33} - \widetilde{\Phi}_{+,1} \widetilde{\Theta}_{13} - \widetilde{\Phi}_{+,2} \widetilde{\Theta}_{23} &\geq \sigma_+ - C\delta_S \geq \eta, \\
\sigma_- + \widetilde{\Theta}_{11} - \widetilde{\Phi}_{-,1} \widetilde{\Theta}_{21} - \widetilde{\Phi}_{-,2} \widetilde{\Theta}_{31} &\leq \sigma_- + C\delta_S \leq - \eta
\end{split}
\end{equation*}
for sufficiently small $\delta_S>0$. Using these bounds in \eqref{w+ODE}--\eqref{w-ODE}, we have
\begin{equation} \label{wpm_bd}
\begin{split}
|w_+(x)| &\geq e^{\eta(x-y)}|w_+(y)|, \quad x\geq y, \\
|w_-(x)| &\geq e^{\eta(y-x)}|w_-(y)|, \quad x\leq y.
\end{split}
\end{equation}
Suppose now that $Z$ is bounded on $\mathbb R$. 
The definitions of $w_\pm$ and the bounds on $\widetilde{\Phi}_\pm$ imply that $w_\pm$ are also bounded on $\mathbb R$. For any fixed $y\in\mathbb R$, letting $x\to+\infty$ in the first inequality of \eqref{wpm_bd} and $x\to-\infty$ in the second yields $w_\pm(y)=0$. Since $y$ is arbitrary, $w_\pm\equiv0$. Thus, $Z$ lies in both graphs, which proves the claim.

By the above claim, any bounded solution $Z$ of \eqref{Z_system} satisfies
\begin{equation}\label{graph_rel}
\begin{split}
Z_+(0)&=\widetilde{\Phi}_+(0)Z_{sc}(0)=\widetilde{\Phi}_+(0)(Z_-(0),Z_0(0))^{\mathrm{tr}},\\
Z_-(0)&=\widetilde{\Phi}_-(0)Z_{cu}(0)=\widetilde{\Phi}_-(0)(Z_0(0),Z_+(0))^{\mathrm{tr}}.
\end{split}
\end{equation}
These relations yield a linear system for $(Z_+(0),Z_-(0))$:
\begin{equation} \label{Z+Z-0system}
\begin{pmatrix}
1 & - \widetilde{\Phi}_{+,1}(0) \\
- \widetilde{\Phi}_{-,2}(0) & 1
\end{pmatrix} \begin{pmatrix}
Z_+(0) \\ Z_-(0)
\end{pmatrix} = \begin{pmatrix}
\widetilde{\Phi}_{+,2}(0) \\ \widetilde{\Phi}_{-,1}(0)
\end{pmatrix} Z_0(0).
\end{equation}
Since $\lVert\widetilde{\Phi}_\pm(0)\rVert\leq C\delta_S$ by \eqref{Phi+_bd}--\eqref{Phi-_bd}, the coefficient matrix in \eqref{Z+Z-0system} is invertible for sufficiently small $\delta_S$ and hence $(Z_+(0),Z_-(0))$ is uniquely determined by $Z_0(0)$. By uniqueness of solutions to \eqref{Z_system}, the space of solutions bounded on $\mathbb R$ has dimension at most one. On the other hand, $\widetilde{R}^{-1}\bar V'$ is a nontrivial bounded solution of \eqref{Z_system}. Returning to the original variables, we conclude that any bounded solution of \eqref{v_eq} is a multiple of $\bar V'$.

Finally, let $W=(v,u,\tilde u',\phi,\phi')^{\mathrm{tr}}$ be a solution of \eqref{5} at $\lambda=0$ that decays exponentially as $x\to\pm\infty$. By the construction in Section~\ref{Construction}, the solution $W$ is a linear combination of the fast modes decaying as $x\to+\infty$ and hence satisfies \eqref{Aconstraints}. Thus, $V=(u,\phi,\phi')^{\mathrm{tr}}$ is a bounded solution of \eqref{v_eq} and, by the above result, $V=c\bar V'$ for some constant $c$. Then it follows from \eqref{sprel} and \eqref{Aconstraints} that $W=cW_0$. Therefore, $\Gamma\neq0$ by Proposition~\ref{D'rel}, which completes the proof.
\end{proof}

\section*{Declaration of AI assistance}
During the revision of this manuscript, the author used ChatGPT (OpenAI, GPT-6 Pro) to assist with language refinement and with checking calculations and proofs. All mathematical statements, calculations, and proofs were independently verified by the author, who takes full responsibility for the content of the manuscript.

\end{document}